\documentclass[a4paper,11pt]{article}
\usepackage{pdfsync}

\usepackage[plainpages=false]{hyperref}
\usepackage{amsfonts,amsmath,amssymb,amsthm}
\usepackage{latexsym,lscape,rawfonts,mathrsfs}
\usepackage[utf8]{inputenc}
\usepackage[overload]{empheq}
\usepackage{cases}

\usepackage[dvips]{color}
\usepackage{multicol}

\usepackage[all]{xy}
\usepackage{eufrak}
\usepackage{makeidx}
\usepackage{graphicx,psfrag}
\usepackage{pstool}
\usepackage{float}

\usepackage{array,tabularx}

\usepackage{setspace}

\usepackage{appendix}

\usepackage{txfonts}

\newcommand{\R}{\ensuremath{\mathbb{R}}}
\newcommand{\Z}{\ensuremath{\mathbb{Z}}}

\newcommand{\ba}{\begin{align*}}
\newcommand{\ea}{\end{align*}}

\newcommand{\norm}[2]{{ \ensuremath{\left\|} #1 \ensuremath{\right\|}}_{#2}}
\newcommand{\snorm}[2]{{ \ensuremath{\left |} #1 \ensuremath{\right |}}_{#2}}

\makeatletter
\def\ExtendSymbol#1#2#3#4#5{\ext@arrow 0099{\arrowfill@#1#2#3}{#4}{#5}}

\makeatother

\makeatletter
\def\ExtendSymbol#1#2#3#4#5{\ext@arrow 0099{\arrowfill@#1#2#3}{#4}{#5}}
\newcommand\longright[2][]{\ExtendSymbol{-}{-}{\rightarrow}{#1}{#2}}
\makeatother

\definecolor{orange}{rgb}{1,0.5,0}
\definecolor{brown}{rgb}{0.48,0.33,0.19}
\definecolor{magenta}{rgb}{1,0,1}
\definecolor{miao}{cmyk}{0.5,0,0.2,0.2}
\definecolor{qiao}{gray}{0.96}

\newtheorem{prop}{Proposition}[section]

\newtheorem{proposition}[prop]{Proposition}

\newtheorem{theorem}[prop]{Theorem}

\newtheorem{lemma}[prop]{Lemma}
\newtheorem{claim}[prop]{Claim}

\newtheorem{corollary}[prop]{Corollary}

\newtheorem{remark}[prop]{Remark}

\newtheorem{definition}[prop]{Definition}

\numberwithin{equation}{section}

\title{On the structure of Ricci shrinkers}
\author{Haozhao Li, \footnote{Supported by NSFC grant No. 12071449 
 and the Fundamental Research Funds for the Central Universities.} 
 \quad Yu Li, \quad Bing Wang \footnote{Supported by  NSFC grant No. 11971452.}}
\date{\today}

\begin{document}
\maketitle

\begin{abstract}
We develop a structure theory for non-collapsed Ricci shrinkers without any curvature condition.
As applications, we obtain some curvature estimates of the Ricci shrinkers depending only on the non-collapsing constant.
\end{abstract}

\tableofcontents

\section{Introduction}
A Ricci shrinker $(M^m, g, f)$ is a complete Riemannian manifold $(M^m,g)$ together with a smooth function $f: M \to \mathbb R$ such that
\begin{equation}
\text{Rc}+\text{Hess} f=\frac{1}{2}g.
\label{eqn:PC27_0}
\end{equation}
Direct calculation shows that $\nabla ( R+|\nabla f|^2-f)=0$.
Adding a constant to $f$ if necessary, we assume throughout that
\begin{align}
R+|\nabla f|^2=f.
\label{eqn:PC27_1}
\end{align}
Under this normalization condition, the functional $\boldsymbol{\mu}=\boldsymbol{\mu}(g)$ is defined as a constant achieving the following identity
\begin{equation}
\int_M e^{-f}(4 \pi)^{-m/2} \,dv=e^{\boldsymbol{\mu}}.
\label{eqn:PD15_1}
\end{equation}
In light of the work of Cao-Zhou~\cite{CaoZhou10},  $f$ always achieves a minimum somewhere.
By $(M,p,g,f)$ we mean a Ricci shrinker $(M,g,f)$ such that $f$ achieves its minimum at point $p$.
For more necessary background and setup of Ricci shrinkers, we refer the readers to the introduction of Cao~\cite{Cao09} and Kotschwar-Wang~\cite{KW15}.

On the one hand, Ricci shrinkers are a natural generalization of Einstein manifolds with positive scalar curvature. On the other hand, Ricci shrinkers are models of the finite-time singularities of the Ricci flow, see \cite{Na10}\cite{EMT11}\cite{MM15}, etc. Staying at the intersection of two important subjects, Ricci shrinkers have attracted a lot of attention.
Note that Ricci shrinkers up to dimension $3$ are completely classified (cf. Cao-Chen-Zhu~\cite{CCZ08} and the introduction of Munteanu-Wang~\cite{MW16}).
A logical next step is to classify Ricci shrinkers in dimension
\begin{align}
m \geq 4,
\label{eqn:PE07_8}
\end{align}
which is a default assumption throughout this paper.
However, even the classification of 4-dimensional Einstein manifolds is far from being complete, not to mention high dimensional Ricci shrinkers.
Therefore, it seems more realistic to develop uniform estimates for the Ricci shrinkers, which is this paper's aim.


\begin{theorem}[\textbf{Weak-compactness of the moduli of Ricci shrinkers}]
Let $(M_i^m, p_i, g_i,f_i)$ be a sequence of Ricci shrinkers.
Let $d_i$ be the length structure induced by $g_i$. By passing to a subsequence if necessary, we have
\begin{align}
(M_i, p_i, d_i,f_i) \longright{pointed-Gromov-Hausdorff} \left(M_{\infty}, p_{\infty}, d_{\infty}, f_{\infty} \right), \label{eqn:CB08_4}
\end{align}
where $(M_{\infty}, d_{\infty})$ is a length space, $f_{\infty}$ is a Lipschitz function on $(M_{\infty}, d_{\infty})$.
The space $M_{\infty}$ has a natural regular-singular decomposition $M_{\infty}=\mathcal{R} \cup \mathcal{S}$.
Suppose further that
\begin{align}
(M_i^m, p_i, g_i,f_i) \in \mathcal{M}_{m}(A), \quad \forall \; i \in \Z^{+}, \label{eqn:PI07_1}
\end{align}
where
\begin{align}
\mathcal{M}_{m}(A) \coloneqq \textrm{the class of Ricci shrinkers $(M^m, p, g,f)$ satisfying} \; \boldsymbol{\mu} \geq -A.
\label{eqn:PD16_1}
\end{align}
Then the following properties are satisfied.
\begin{itemize}
\item[(a).] The singular part $\mathcal{S}$ is a closed set of Minkowski codimension at least $4$. Namely, we have
\begin{align}
\dim_{\mathcal{M}} \mathcal{S} \leq m-4. \label{eqn:PI06_01}
\end{align}
\item[(b).] The regular part $\mathcal{R}$ is an open manifold with smooth metric $g_{\infty}$ satisfying Ricci shrinker equation
$Rc_{\infty} + \emph{Hess}_{g_{\infty}} f_{\infty}-\frac{g_{\infty}}{2}=0$.
\end{itemize}
The convergence (\ref{eqn:CB08_4}) can be improved to
\begin{align}
(M_i, p_i, g_i,f_i) \longright{pointed-\hat{C}^{\infty}-Cheeger-Gromov} \left(M_{\infty}, p_{\infty}, g_{\infty}, f_{\infty} \right). \label{eqn:CA25_10}
\end{align}
Furthermore, the metric structure induced by smooth curves in $(\mathcal{R}, g_{\infty})$ coincides with $d_{\infty}$.

\label{thmin:a}
\end{theorem}

Theorem~\ref{thmin:a} is basically a compactness theorem, which inherits many ingredients from previous compactness theorems, e.g.  Cao-Sesum~\cite{CaoSe07},  Chen-Wang~\cite{CW12}, Haslhofer-M\"uller~\cite{HM11},~\cite{HM15}, Tian-Zhang~\cite{TianZhang},  B. Weber~\cite{Weber}, X. Zhang~\cite{ZhangXi}, Z. Zhang~\cite{ZhangZL} . We are more interested in the case when each $M_i$ is a complete non-compact manifold. This point was first observed by the fundamental work of Haslhofer-Muller~\cite{HM11} ,
which affected the current paper in many perspectives.   Actually, the Gromov-Hausdorff compactness (\ref{eqn:CB08_4}) was already obtained by~\cite{HM11} in the non-collapsing case.
The major contribution of Theorem~\ref{thmin:a} is the regularity improvement for (\ref{eqn:CA25_10}).

Note that the condition $\boldsymbol{\mu} \geq -A$ is equivalent to (cf. Lemma~\ref{lma:PE01_3}) the condition
\begin{align*}
|B(p, 1)| \geq V_0
\end{align*}
for some positive $V_0$ depending only on $m$ and $A$. Therefore, $\mathcal{M}_{m}(A)$ can be understood as the class of non-collapsed Ricci shrinkers.
The ``pointed-$\hat{C}^{\infty}$-Cheeger-Gromov convergence" in (\ref{eqn:CA25_10}) was sometimes
called ``Cheeger-Gromov convergence away from singularity" in literatures (cf. p. 795 of Cao-Sesum~\cite{CaoSe07}).
Roughly speaking, the ``pointed-$\hat{C}^{\infty}$-Cheeger-Gromov" convergence is an improvement of the ``pointed-Gromov-Hausdorff" convergence.
The notation ``\;$\hat{}$\;" is used to indicate the possible existence of singularity.
The notation ``$C^{\infty}$" is applied to mean the smooth convergence on regular part.
The notation ``Cheeger-Gromov'' means ``modulo diffeomorphisms'' around regular part.
Such notations were first introduced in Chen-Wang~\cite{CW17A}.
The full definitions of ``pointed-$\hat{C}^{\infty}$-Cheeger-Gromov" will be explained in section~\ref{sec:CGconv} around Definition~\ref{dfn:PC25_1}.\\

The equation (\ref{eqn:CA25_10}) combines the Gromov-Hausdorff convergence with the PDE theory.
It provides a way to generalize the classical theorems for manifolds with Ricci lower bound to Ricci shrinkers, with much better regularities.
For example, a fundamental theorem in the Cheeger-Colding theory is the volume continuity theorem of T.H. Colding (cf. \cite{Colding}).
Based on Theorem~\ref{thmin:a}, we can generalize this theorem to the following version.

\begin{theorem}[\textbf{Continuity of measures}]
For each $\theta \in (0,2]$, the measure
\begin{align}
dv' \coloneqq |Rm|^{2-\theta} e^{-f}dv
\label{eqn:PE26_5}
\end{align}
satisfies the following properties.
\begin{itemize}
\item[(a).] $dv'$ is uniformly locally finite in the sense that
\begin{align}
\int_{B(p,D)} dv' < C(m,D,\theta,A)
\label{eqn:PE26_4}
\end{align}
for each Ricci shrinker $(M,p,g,f) \in \mathcal{M}_{m}(A)$.
\item[(b).] $dv'$ is continuous with respect to the Gromov-Hausdorff topology in the sense that
\begin{align}
\int_{B(q_{\infty},r)} dv_{\infty}' \coloneqq
\int_{\mathcal{R} \cap B(q_{\infty},r)} |Rm|^{2-\theta} e^{-f_{\infty}} dv_{\infty}=
\lim_{i \to \infty} \int_{B(q_i,r)} dv_i'. \label{eqn:PE26_2}
\end{align}
\item[(c).] If $\theta=2$, we further have
\begin{align}
\int_{M_{\infty}} dv_{\infty}'\coloneqq \int_{\mathcal{R}} |Rm|^{2-\theta} e^{-f_{\infty}} dv_{\infty}=\lim_{i \to \infty} \int_{M_i} dv_i'.
\label{eqn:PE26_3}
\end{align}
\end{itemize}
\label{thmin:b}
\end{theorem}

As an application of Theorem~\ref{thmin:a} and Theorem~\ref{thmin:b}, we can derive a uniform scalar curvature lower bound around base points.

\begin{theorem}[\textbf{Scalar curvature lower bound}]
For each $m,A$, there is a constant $\epsilon=\epsilon(m,A)>0$ with the following properties.

Suppose $(M^{m}, p, g, f) \in \mathcal{M}_{m}(A)$ and $M$ is a closed manifold. Then we have
\begin{align}
\inf_{B(p,1)} R(x) \geq \epsilon. \label{eqn:PC25_10}
\end{align}
\label{thmin:c}
\end{theorem}

Before we discuss the proof of the previous theorems, we first set up the framework and point out the critical difficulties to prove them.

In this paper, we make use of the local conformal transformation to bridge the study of Ricci shrinkers with the classical Cheeger-Colding theory.
Actually, it is observed by Z. Zhang~\cite{ZhangZL} that if we set $\tilde{g}=e^{-\frac{2f}{m-2}}g$, then we have
\begin{align*}
\widetilde{Rc}=\frac{1}{m-2} \left\{ df \otimes df + (m-1-f) e^{\frac{2f}{m-2}} \tilde{g} \right\}.
\end{align*}
If we have uniform diameter bound of $M$, then a direct calculation shows that $|\widetilde{Rc}|_{\tilde{g}}$ is uniformly bounded.
Consequently, one can use the classical Cheeger-Colding theory to study the uniform geometric estimates of $(M, \tilde{g})$ and then translate them
to the estimates for $(M, g)$, through the inverse conformal transformation.
This method works perfectly for closed Ricci shrinkers with uniformly bounded diameters.
If the Ricci shrinkers are complete non-compact manifolds, then some technical issues appear.
Due to the quadratic behavior of $f$ around infinity (cf. Lemma~\ref{L100}), it is easy to see that $(M, \tilde{g})$ is not a complete metric space (cf. Haslhofer-M\"{u}ller~\cite{HM15} for example).
We observe that $(M, \tilde{g})$ has other properties worse than non-completeness.
For example, it is possible to find two points in $M$ which cannot be connected by minimizing geodesics in $(M, \tilde{g})$.
See Proposition~\ref{prn:PE27_1} for more details.

The bad behavior of $(M, \tilde{g})$ suggests that we should perform the conformal transformation locally.
For each $q \in M$, we define $\bar{f}=f-f(q)$ and $\bar{g}=e^{-\frac{2\bar{f}}{m-2}}g$.
Then for each scale $r$ small enough, say $r<\frac{1}{100m D}$ for $D=d(p, q)+10m$, we show that the ball $B_{\bar{g}}(q,r)$ is relatively compact and $|\overline{Rc}|_{\bar{g}}r^2$ is uniformly bounded.
Moreover,
\begin{align}
r^{-1}d_{GH} \left( B_{\bar{g}}(q,r), \; B_{g}(q,r) \right)< 2D r. \label{eqn:PI06_1}
\end{align}
Therefore, one can use the classical Cheeger-Colding theory to study the geometry of $B_{\bar{g}}(q,0.1 r)$ and then obtain the information of $B(q, 0.01r)$ through the inverse conformal transformation.
In particular, the local behavior of Ricci shrinkers and the manifolds with bounded Ricci curvature are the same.

Suppose we have a sequence of Ricci shrinkers $(M, p_i, g_i, f_i)$, by the comparison geometry of Bakry-\'{E}mery Ricci tensor developed by Wei-Wylie~\cite{WeiWylie} and the gradient estimate of $f$,
it follows from standard ball-packing argument that (\ref{eqn:CB08_4}) holds.
Namely, $(M, p_i, d_i, f_i)$ sub-converge to a limit $(M_{\infty}, p_{\infty}, d_{\infty}, f_{\infty})$. On the limit space $M_{\infty}$, a point $q$ is called regular if all tangent spaces at $q$ is the same Euclidean space and $q$ is called singular if it is not regular (see Definition \ref{dfn:PC24_2} for the precise definitions of these concepts). Therefore, $M_{\infty}$ has a natural regular-singular decomposition $M_{\infty}=\mathcal{R} \cup \mathcal{S}$.
In general, the structure of $\mathcal{R}$ is obscure. However, under the non-collapsing condition (\ref{eqn:PI07_1}), we can apply the non-collapsed Cheeger-Colding theory and (\ref{eqn:PI06_1}) to obtain
every tangent space is a metric cone. Up to conformal transformation, Cheeger-Naber's work implies (\ref{eqn:PI06_2}), the codimension estimate of $\mathcal{S}$.

The technical challenge then is to show the smooth manifold structure of $\mathcal{R}$ and the smooth convergence around $\mathcal{R}$.
In the closed manifold case, this issue is solved by Z. Zhang through a smart application of the pseudo-locality theorem of Perelman (cf. Theorem 10.1 of Perelman~\cite{Pe1}), as Ricci shrinkers can be regarded as the time $t=0$ slice of a Ricci flow.
Basically, suppose $M_i \ni q_i \to q_{\infty} \in M_{\infty}$ and $q_{\infty}$ is a regular point, then there is an $r$ independent of $i$ such that $B(q_i, r)$ has isoperimetric constant very close to the Euclidean
isoperimetric constant. By the pseudo-locality theorem, the geometry around $q_i$ at time-slice $t=\epsilon r^2$ has uniformly bounded regularities. However, the geometry of $g(\epsilon r^2)$ and $g(0)$ are the same up to diffeomorphism and rescaling. Therefore, one obtains uniform regularities around $q_i$, with respect to the metrics $g_i=g_i(0)$.

However, unless the pseudo-locality theorem is proved for non-compact Ricci shrinkers (cf. Remark~\ref{rmk:PI06_1}), the above argument does not apply here directly.
A regularity improvement argument independent of the pseudo-locality theorem is needed.
We notice that the method of local conformal transformation should be combined with the method of M. Anderson~\cite{An90} to study
the regularity issue, in the harmonic chart of the conformal metric $\bar{g}=e^{-\frac{2f}{m-2}}g$, rather than of the metric $g$.
(This is an important point. See Remark~\ref{rmk:PI07_2} for more details.)
We regard $(\bar{g}, \bar{f})$ as a couple and study the elliptic system
\begin{subequations}
\begin{align}[left = \empheqlbrace \,]
&\overline{Rc}=\frac{1}{m-2} \left\{ d\bar{f} \otimes d\bar{f} + (m-1-f(q)-\bar{f}) e^{\frac{2(f(q)+\bar{f})}{m-2}} \bar{g} \right\}, \label{eqn:CB04_2a}\\
&\overline{\Delta} \bar{f}=e^{\frac{2(f(q)+\bar{f})}{m-2}} \left( \frac{m}{2}-f(q)-\bar{f} \right). \label{eqn:CB04_2b}
\end{align}
\label{eqn:CB04_2}
\end{subequations}
\\
\noindent
Instead of the system (\ref{eqn:CB04_2}), only the equation (\ref{eqn:CB04_2a}) was emphasized in the literature (cf. for example~\cite{ZhangZL},~\cite{HM15}).
In the ball $B(p, D)$, $f(q)$ is uniformly bounded by $\frac{D^2}{4}$.
Starting from $C^{1,\alpha}$ metric $\bar{g}$ and Lipschitz function $\bar{f}$, standard elliptic PDE bootstrapping argument implies that $(\bar{g}, \bar{f})$ has all high order derivatives' estimates,
which guarantee the smooth convergence of $(\bar{g}_i, \bar{f}_i)=\left(e^{-\frac{2\bar{f}_i}{m-2}}g_i, \bar{f}_i \right)$.
Consequently, after the inverse conformal transformation, we know $(g_i, f_i)$ converges to $(g_{\infty}, f_{\infty})$ smoothly.
Such ideas are refined further in Section~\ref{sec:regularity}, where we systematically investigate the different radii, including (restricted) Gromov-Hausdorff radius, volume radius, harmonic radius and strongly-convex radius,
under the metric $g$ and $\bar{g}$.
We show all of them are equivalent. The technical heart of the equivalence is the argument mentioned above for regularity improvement.

After we settle down the regularity issue, the improvement from (\ref{eqn:CB08_4}) to (\ref{eqn:CA25_10}) seems well-known to experts.
However, it is not easy to find a clear reference with full details to obtain such an improvement.
For the sake of self-containedness, we provide details and enough necessary references.
We also have many new observations (for example, Remark~\ref{rmk:PE30_1} and Remark~\ref{rmk:PF24_1}) which are not available in the literature.
We believe such new observations will be important for further developing the weak compactness theory of critical metrics.
Theorem~\ref{thmin:b} and Theorem~\ref{thmin:c} can be regarded as the applications of Theorem~\ref{thmin:a}. \\

Now we briefly discuss the proof of the previous theorems.

\begin{proof}[Outline of the proof of Theorem~\ref{thmin:a}:]
As discussed before, it is relatively easy to obtain (\ref{eqn:CB08_4}) and obtain the regular-singular decomposition $M_{\infty}=\mathcal{R} \cup \mathcal{S}$.
Using the local conformal transformation and the distance estimate (\ref{eqn:PI06_1}), it follows from Cheeger-Colding theory that
every tangent space of $M_{\infty}$ is a metric cone. Moreover, for each point $y \in \mathcal{S}$, there is a uniform gap from the tangent space of $M_{\infty}$ at $y$
and the standard $\R^m$. The codimension estimate (\ref{eqn:PI06_01}) follows from Cheeger-Naber's result and the conformal transformation.
The remaining issue is the ``smoothness" of $\mathcal{R}$ and the smooth convergence around $\mathcal{R}$.
For each point $y \in \mathcal{R}$, through the volume continuity theorem of Colding and the gap lemma of Anderson, we can construct a uniform sized harmonic
coordinate chart with respect to $\bar{g}_i$, around $y_i$ where $M_i \ni y_i \to y$. Inside the harmonic coordinate charts, we run the regularity improvement argument and obtain all higher-order derivatives of $\bar{g}_i$ and $\bar{f}_i$.
Then we apply the inverse conformal transformation to obtain the uniform regularities of $g_i,f_i$, within the harmonic coordinate chart of $\bar{g}_i$.
After a careful choice of covering, we can patch the different harmonic coordinates together to obtain a smooth Riemannian manifold structure of $\mathcal{R}$.
Also, the uniform estimates inside harmonic coordinate charts can be used to construct diffeomorphisms to achieve the $C^{\infty}$-Cheeger-Gromov convergence
away from singularities. We show that the diffeomorphisms match the original Gromov-Hausdorff approximation maps very well (cf. (\ref{eqn:PE25_2}) and Remark~\ref{rmk:PF24_1}).
At this point, it is still not clear if the length structure $d_{\infty}$ is the same as the length structure induced from the smooth metric $g_{\infty}$.
This coincidence will be proved by the fact that every tangent space is a metric cone (cf. Proposition~\ref{prn:CB08_1}).
The final proof of Theorem~\ref{thmin:a} can be found in section~\ref{sec:coincidence}.
\end{proof}

\begin{proof}[Outline of the proof of Theorem~\ref{thmin:b}:]
Modulo slight terminology difference, the Minkowski codimension estimate (\ref{eqn:PI06_01}) has a qualitative version, which claims that
\begin{align}
\int_{B(p, D)} hr^{2\theta-4} dv<C(m,D,\theta), \label{eqn:PI06_2}
\end{align}
where $hr$ is the harmonic radius. The above inequality (\ref{eqn:PI06_2}) is called the density estimate (cf. Proposition 2.57 of Chen-Wang~\cite{CW17A}).
See also the discussion and Figure 1 of Chen-Wang~\cite{CW6} for more information.
Then Theorem~\ref{thmin:b} is a direct corollary of (\ref{eqn:PI06_2}) and the smooth convergence (\ref{eqn:CA25_10}).
Full details can be found in section~\ref{sec:coincidence}.
\end{proof}

\begin{proof}[Outline of the proof of Theorem~\ref{thmin:c}:]
Theorem~\ref{thmin:c} is proved through a contradiction argument. Suppose the theorem fails, one can obtain a sequence of shrinkers $(M_i, p_i, g_i, f_i) \in \mathcal{M}_{m}(A)$ such that $R(q_i) \to 0$
for some $q_i \in B(p_i, 1)$. By Theorem~\ref{thmin:a}, we may assume that $(M_i, p_i, g_i, f_i)$ converges to a limit space $(M_{\infty}, p_{\infty}, g_{\infty}, f_{\infty})$, and $q_i$ converges to $q_{\infty}$.
Suppose we have
\begin{align}
q_{\infty} \in \mathcal{R}(M_{\infty}), \quad \label{eqn:PI04_8}
\end{align}
then $R(q_{\infty})=0$ and the strong maximum principle guarantees that $R \equiv 0$ on $\mathcal{R}(M_{\infty})$. Consequently, $|Rc| \equiv 0$ on $\mathcal{R}(M_{\infty})$.
By the Ricci shrinker equation on $\mathcal{R}(M_{\infty})$, we obtain $\text{Hess} f_{\infty}=\frac{g_{\infty}}{2}$. The high codimension of $\mathcal{S}(M_{\infty})$ and the regularity of $\mathcal{R}(M_{\infty})$
then implies that $M_{\infty}$ is a metric cone. In particular, $M_{\infty}$ contains a point $z_{\infty}$ which has very large harmonic radius. Then by the smooth convergence around $\mathcal{R}(M_{\infty})$,
we know $M_i$ contains a point $z_i$ with very large harmonic radius, which contradicts the pseudo-locality theorem (cf. Proposition~\ref{prn:PC30_3}).
Therefore, the proof is finished if we can obtain (\ref{eqn:PI04_8}).
Without further estimates, (\ref{eqn:PI04_8}) may fail. Namely, it is possible that $q_i \to q_{\infty} \in \mathcal{S}(M_{\infty})$.
It becomes a key new technical difficulty to obtain (\ref{eqn:PI04_8}).
Via analysis related to $\boldsymbol{\mu}$-functional and Sobolev constant, we obtain mean value properties for $-R$.
The mean value property, together with Theorem~\ref{thmin:b}, guarantees that (\ref{eqn:PI04_8}) always hold up to a perturbation of $q_i$.
The complete details of the proof can be found in section~\ref{sec:gap}.
\end{proof}

For the convenience of the readers, we highlight some places needing particular attention.

\begin{remark}
The limit space in (\ref{eqn:CA25_10}) is actually an $RCD(\frac{1}{2}, \infty)$ space, by the work of Gigli-Mondino-Savar\'{e}~\cite{GMS}.
It is also a Riemannian conifold (cf. Theorem~\ref{thm:PH26_1}) in the sense of Chen-Wang~\cite{CW17A}.
\label{rmk:PI09_1}
\end{remark}

\begin{remark}
Note that (\ref{eqn:CB08_4}) and the estimate $\dim_{\mathcal{H}} \mathcal{S} \leq m-4$ can be obtained by the work of Wang-Zhu~\cite{WZ13} and Zhang-Zhu~\cite{ZZ17}.
Then there is an alternative approach to prove (\ref{eqn:CA25_10}), as pointed out by Guo-Phong-Song-Sturm~\cite{GPSS}.
We emphasize that each $M_i$ in Theorem~\ref{thmin:a} may not have bounded curvature.
Namely, it is possible that $$\sup_{x \in M_i} |Rm|(x)=\infty.$$
Without the bounded curvature condition, the pseudo-locality theorem is not available in the literature (cf. Theorem 8.1 of Chau-Tam-Yu~\cite{CTY} and line 11 on page 365 of B. Chen~\cite{CBL07}).
When $m=4$, the pseudo-locality theorem can be avoided in the proof of (\ref{eqn:CA25_10}) by using $\epsilon$-regularity of ``energy" $\int |Rm|^2$
and some a priori energy bound due to Cheeger-Naber~\cite{CN2}, as done in Haslhofer-M\"{u}ller~\cite{HM15}.
However, such method relies heavily on the scaling invariance of $\int |Rm|^2$, which fails in general dimension.
\label{rmk:PI06_1}
\end{remark}

\begin{remark}
If we study Ricci shrinkers nearby a given geometric object, then (\ref{eqn:PD16_1}) will be guaranteed automatically in many cases (cf. Proposition~\ref{prn:PE29_3}).
From the proof of Theorem~\ref{thmin:c}, we can actually show that every cone $\R^{m}/\Gamma$ with $\Gamma \subset O(m)$
cannot be approximated by compact Ricci shrinkers(cf. Corollary~\ref{cly:PC30_1}), which is conjectured to be true in Li-Wang~\cite{LWs1}.
We also believe that the closed manifold assumption in Theorem~\ref{thmin:c} is redundant.
\label{rmk:PI06_2}
\end{remark}

The necessary notations in this paper are listed below with pointers to their definitions. \\

\noindent
\textbf{List of notations}

\begin{itemize}
\item $D$: a constant larger than $10m$. First appears before (\ref{eqn:PE03_6a}).
\item $\bar{f}$: relative Ricci shrinker potential function with respect to point $q$. Defined in (\ref{eqn:PE28_3}).
\item $g$: the default Ricci shrinker metric.
\item $\tilde{g}$: the conformal change of $g$. Defined in (\ref{eqn:PE27_61}).
\item $\bar{g}$: the relative conformal change of $g$ with respect to point $q$. Defined in (\ref{eqn:PE27_7}).
\item $\hat{g}$: the rescaling of $\bar{g}$. Defined in (\ref{eqn:PH14_1}).
\item $m$: the real dimension of $M$ and it is assumed $m \geq 4$ in (\ref{eqn:PE07_8}).
\item $\mathcal{M}_{m}(A)$: the moduli of $m$-dimensional Ricci shrinkers satisfying $\boldsymbol{\mu} \geq -A$. First appears and defined in (\ref{eqn:PD16_1}).
\item $\mathcal{R}$: regular part of the limit space. First appears in Theorem~\ref{thmin:a}. Defined in Definition~\ref{dfn:PC24_2}.
\item $Rc_f$: Bakry-Emery Ricci tensor. Defined and first appears in Lemma~\ref{lma:PE01_1}.
\item $\mathcal{S}$: singular part of the limit space. First appears in Theorem~\ref{thmin:a}. Defined in Definition~\ref{dfn:PC24_2}.
\item $\mathcal{S}_r$: $r$-neighborhood of the singular set $\mathcal{S}$. First appears and defined in (\ref{eqn:PE27_4}).
\item $\omega_m$: the volume of standard unit ball in $\R^m$. First appears and defined in Lemma~\ref{lma:PE04_1}.
\item $\square$: Heat operator $\partial_t-\Delta$. First appears in Lemma~\ref{lma:PD16_1}.
\item $\longright{pointed-\hat{C}^{\infty}-Cheeger-Gromov}$: Convergence in pointed-$\hat{C}^{\infty}$-Cheeger-Gromov topology. First appears in the discussion after Theorem~\ref{thmin:a}.
Defined in Definition~\ref{dfn:PC25_1}.
\item $\longright{pointed-Gromov-Hausdorff}$: Convergence in pointed-Gromov-Hausdorff topology. First appears in the discussion after Theorem~\ref{thmin:a}. Defined in Definition~\ref{dfn:PC24_1}.
A similar notation $\longright{Gromov-Hausdorff}$ is also defined there.
\end{itemize}

This paper is organized as follows. In section~\ref{sec:pre}, we provide previously known results about Ricci shrinkers and some elementary estimates.
In section~\ref{sec:conformal}, we study conformal transformation and show why it should be considered locally.
In section~\ref{sec:regularity}, we show the equivalence of many radii and prove a regularity improvement theorem under harmonic coordinate. This section is the technical heart of this paper.
In section~\ref{sec:weakexistence}, we develop general estimates which transform the local geometry of Ricci shrinkers to the classical Cheeger-Colding theory.
In section~\ref{sec:reglimit}, we provide details and elementary steps to show the regular part of the limit space in Theorem~\ref{thmin:a} is a smooth manifold.
In section~\ref{sec:CGconv}, we show the smooth convergence on the regular part. Section 6 and section 7 provide detailed argument and some new ingredients for the
Cheeger-Gromov convergence.
Then the proof of Theorem~\ref{thmin:a} and Theorem~\ref{thmin:b} are finished in section~\ref{sec:coincidence}.
The proof of Theorem~\ref{thmin:c} is given in section~\ref{sec:gap}.
Finally, we discuss some further possible development in section~\ref{sec:further}.\\

\section{Preliminaries}
\label{sec:pre}

Let ${\psi_t}: M \to M$ be a $1$-parameter family of diffeomorphisms generated by $X(t)=\dfrac{1}{1-t}\nabla_gf$. That is
\begin{align*}
\frac{\partial}{\partial t} {\psi_t}(x)=\frac{1}{1-t}\nabla_g f\left({\psi_t}(x)\right).
\end{align*}
It is well-known that the rescaled pull-back metric
\begin{align}
g(t)\coloneqq (1-t) (\psi_t)^*g, \label{eqn:PH29_3}
\end{align}
satisfies the Ricci flow equation
\begin{align*}
\partial_t g=-2Rc(g(t))
\end{align*}
for any $-\infty <t<1$.
Moreover, the pull-back function
\begin{align}
f(t)\coloneqq (\psi_t)^*f \label{eqn:PH29_2}
\end{align}
satisfies the evolution equation
\begin{align*}
\partial_tf(t)=|\nabla_{g(t)}f(t)|_{g(t)}^2.
\end{align*}
Therefore, we can regard $(M,g,f)$ as the $t=0$ time-slice of the Ricci flow space-time $M \times (-\infty, 1)$.
On a general time-slice, we have
\begin{align}
Rc(t) + \text{Hess} _{g(t)}f(t)=\frac{g(t)}{2(1-t)} \label{eqn:PH29_1}
\end{align}
and the identity
\begin{align} \label{eqn:PH29_1a}
R(t)+|\nabla_{g(t)}f(t)|_{g(t)}^2=\frac{f(t)}{1-t}.
\end{align}
For any $x \in M$ and $t \le 0$,
\begin{align*}
0 \le \partial_tf(x,t)=|\nabla_{g(t)}f(x,t)|_{g(t)}^2\le \frac{f(x,t)}{1-t}.
\end{align*}
Therefore, we have on $M \times (-\infty,0]$ that
\begin{align} \label{X001}
\frac{f(x,0)}{1-t} \le f(x,t) \le f(x,0).
\end{align}

According to our choice, $f-\boldsymbol{\mu}$ satisfies the normalization condition
\begin{align}
\int_{M} (4\pi)^{-\frac{m}{2}} e^{-f-\boldsymbol{\mu}} dv=1. \label{eqn:PE27_3}
\end{align}
We also have
\begin{align}
& R+\Delta f -\frac{m}{2}=0, \label{eqn:PC06_2} \\
& R+|\nabla f|^2-f=0. \label{eqn:PC06_3}
\end{align}

The following type estimates are first proved by Cao-Zhou~\cite{CaoZhou10} and improved by Haslhofer-M\"{u}ller~\cite{HM11}.
\begin{lemma}[Lemma 2.1 of ~\cite{HM11} and Theorem 1.1 of~\cite{CaoZhou10}]
Let $(M^m,g,f)$ be a Ricci shrinker. Then there exists a point $p \in M$ where $f$ attains its infimum and $f$ satisfies the quadratic growth estimate
\begin{equation}
\frac{1}{4}\left(d(x,p)-5m \right)^2_+ \le f(x) \le \frac{1}{4} \left(d(x,p)+\sqrt{2m} \right)^2
\label{E106}
\end{equation}
for all $x\in M$, where $a_+ :=\max\{0,a\}$.
\label{L100}
\end{lemma}

In light of (\ref{E106}), inside the ball $B(p,1)$, we have
\begin{subequations}
\begin{align}[left = \empheqlbrace \,]
&0 \leq f \leq \frac{1}{4} \left( 1+\sqrt{2m} \right)^2 \leq m,\label{eqn:PE01_11a}\\
&|\nabla f| \leq \sqrt{m} \leq m, \label{eqn:PE01_11b} \\
&0 \leq R \leq m. \label{eqn:PE01_11c}
\end{align}
\label{eqn:PE01_11}
\end{subequations}
If we choose $D>10m$, then (\ref{eqn:PC06_3}) and (\ref{E106}) imply the following estimates on $B(p,D)$:
\begin{subequations}
\begin{align}[left = \empheqlbrace \,]
&0 \leq f < \frac{1}{2} D^2,\label{eqn:PE03_6a}\\
&|\nabla f| < \sqrt{f} < D, \label{eqn:PE03_6b} \\
&0 < R <f < D^2. \label{eqn:PE03_6c}
\end{align}
\label{eqn:PE03_6}
\end{subequations}
On the set $M \backslash B(p, D)$, we have
\begin{align}
\frac{1}{16} d^2 \leq &f(x) \leq \frac{1}{2} d^2, \quad \textrm{on} \; M \backslash B(p, D); \label{eqn:PE01_3}\\
&|\nabla f|(x) \leq d, \quad \textrm{on} \; M \backslash B(p, D). \label{eqn:PE01_5}
\end{align}

The following comparison theorem is due to Wei-Wylie~\cite{WeiWylie}. We modify some sentences to reduce notation conflicts.

\begin{lemma}[Part of Theorem 1.2 of~\cite{WeiWylie}]
Let $(M^{m}, g, e^{-f} dv)$ be complete smooth metric measure space with $Rc_{f} \coloneqq Rc + \emph{Hess} f \geq (m-1)H$.
Fix $q \in M$. If $\langle \nabla f, \dot{\gamma} \rangle \geq -a$ along each unit-speed minimal geodesic segment $\gamma$ from $q$,
then for $r \geq \rho >0$(assume $r \leq \frac{\pi}{2\sqrt{H}}$ if $H>0$),
\begin{align}
\frac{\int_{B(p,r)} e^{-f} dv}{\int_{B(p, \rho)} e^{-f} dv} \leq e^{ar} \frac{Vol_{H}^{m}(r)}{Vol_{H}^{m}(\rho)}
\label{eqn:PE01_2}
\end{align}
where $Vol_{H}^{m}(r)$ is the volume of the radius $r$-ball in the model space $M_{H}^{m}$.
\label{lma:PE01_1}
\end{lemma}

Lemma~\ref{lma:PE01_1} provides the necessary tools to obtain volume ratio upper bound of geodesic balls.
In our setting, $M$ is the Ricci shrinker, $f$ is the shrinker potential and $p$ is the base point of $M$.
It is clear that
\begin{align*}
Rc_{f}=Rc + \text{Hess} f =\frac{1}{2} g \geq 0.
\end{align*}
Therefore, in (\ref{eqn:PE01_2}), we can choose $H=0$.
Suppose $f_{min} \leq f \leq f_{max}$ in $B(q, r)$ and $|\nabla f| \leq U$, then we have
\begin{align}
\frac{|B(q,r)|}{|B(q,\rho)|}&=\frac{\int_{B(q,r)} 1 dv}{\int_{B(q,\rho)} 1 dv} \leq \frac{\int_{B(q,r)} e^{-f+f_{max}} dv}{\int_{B(q,\rho)} e^{-f+f_{min}} dv}
=e^{f_{max}-f_{min}} \frac{\int_{B(p,r)} e^{-f} dv}{\int_{B(p,\rho)} e^{-f} dv} \leq e^{Ur+f_{max}-f_{min}} \left( \frac{r}{\rho} \right)^m \notag\\
&\leq e^{Ur + Osc_{B(q,r)} f} \left( \frac{r}{\rho} \right)^m. \label{eqn:PE03_5}
\end{align}
If we set $B(q,r)=B(p,1)$, then $|\nabla f| \leq m$ by (\ref{eqn:PE01_11b}) and $Osc_{B(p,1)} f \leq m$ by (\ref{eqn:PE01_11a}), we have
\begin{align}
\frac{|B(p,1)|}{|B(p, 0.5)|} \leq e^{2m} 2^{m} < 2^{4m}. \label{eqn:PE04_4}
\end{align}
If we choose $B(q,r) \subset B(p, D)$ for some $D>10m$, then $|\nabla f| \leq D$ by (\ref{eqn:PE03_6b}), we have
\begin{align}
\frac{|B(q,r)|}{|B(q, 0.5r)|} \leq e^{3 Dr} 2^{m}. \label{eqn:PE03_3}
\end{align}
The inequalities (\ref{eqn:PE04_4}) and (\ref{eqn:PE03_3}) are relative volume ratio bounds. The absolute volume ratio bounds are provided by the following Lemma.

\begin{lemma}
For each geodesic ball $B(q,r) \subset B(p, D) \subset M$, we have volume ratio upper bound
\begin{align}
r^{-m}|B(q,r)| \leq \min \left\{ \left( r^{-m} e^{\frac{1}{2}D^2}\right) \cdot (4\pi)^{\frac{m}{2}} \cdot e^{\boldsymbol{\mu}}, e^{2Dr} \omega_{m}\right\}, \label{eqn:PE04_1}
\end{align}
where $\omega_m$ is the volume of standard unit ball in $\R^{m}$.  In particular, we have
\begin{align}
 \frac{\left| B(p,r) \right|}{e^{\boldsymbol{\mu}}} \le C e^{\frac{1}{2}r^2}.
\label{eqn:CB09_2}
\end{align}
\label{lma:PE04_1}
\end{lemma}

\begin{proof}
It suffices to prove (\ref{eqn:PE04_1}).
By the definition of $\boldsymbol{\mu}$, it follows from (\ref{eqn:PD15_1}) that
\begin{align*}
|B(q,r)|=\int_{B(q,r)}1 d v \leq \int_{B(q,r)} e^{-f+f_{max}} dv \leq e^{f_{max}} \int_{M} e^{-f} dv=(4\pi)^{\frac{m}{2}} e^{f_{max}+\boldsymbol{\mu}},
\end{align*}
where $f_{max}$ is the supremum of $f$ on $B(q, r)$. Plugging (\ref{eqn:PE03_6a}) into the above inequality, we obtain
\begin{align}
r^{-m}|B(q,r)| \leq \left( r^{-m} e^{\frac{1}{2}D^2}\right) \cdot (4\pi)^{\frac{m}{2}} \cdot e^{\boldsymbol{\mu}}. \label{eqn:PE04_2}
\end{align}
On the other hand, by setting $H=0$ in (\ref{eqn:PE01_2}), we have
\begin{align*}
r^{-m} e^{-f_{max}} |B(q,r)| \leq r^{-m} \int_{B(q,r)} e^{-f} dv \leq e^{Ur} \lim_{\rho \to 0} \rho^{-m} \int_{B(q,r)} e^{-f} dv=e^{Ur} \omega_{m} e^{-f(q)},
\end{align*}
where $U$ is an upper bound of $|\nabla f|$ in $B(q,r)$. Since $B(q,r) \subset B(p, D)$, $U$ can be chosen as $D$ by (\ref{eqn:PE03_6b}). It follows that
\begin{align}
r^{-m} |B(q,r)| \leq e^{Dr} \omega_{m} e^{f_{max}-f(q)} \leq e^{2Dr} \omega_{m}. \label{eqn:PE04_3}
\end{align}
Therefore, (\ref{eqn:PE04_1}) follows from the combination of (\ref{eqn:PE04_2}) and (\ref{eqn:PE04_3}).
\end{proof}

Note that () is a slight improvement of Lemma 2.2 of~\cite{HM11} and Theorem 1.2 of~\cite{CaoZhou10}, where it is shown that
\begin{align*}
\left| B(p,r) \right| \le C(m) r^m.
\end{align*}

By the fundamental work of Perelman~\cite{Pe1}, it is well-known that $\boldsymbol{\mu}$ is closely related to the volume ratio of geodesic balls.
We shall follow the argument in Theorem 3.3 of~\cite{BWang17}, which originates from~\cite{Pe1} and~\cite{KL}, to calculate explicitly the
lower bound of the volume of $B(q, 1)$.

\begin{lemma}
For each geodesic ball $B(q,r) \subset B(p, D) \subset M$ and $r \in (0, 1]$, we have
\begin{align}
r^{-m}|B(q,r)| \geq e^{-e^{m+4D}} \cdot e^{\boldsymbol{\mu}}. \label{eqn:PE04_5}
\end{align}
\label{lma:PE04_2}
\end{lemma}

\begin{proof}
Let us first assume that $r=1$.
Let $\eta$ be a cutoff function which equals $1$ on $(-\infty, 0.5]$, decreases to $0$ on $[0.5, 1]$ and equals $0$ on $[1, \infty)$.
Furthermore, $|\eta'| \leq 4$. We then set
\begin{align*}
L \coloneqq \int_{M} \eta^{2}\left( d \right) dv, \quad \varphi \coloneqq L^{-\frac{1}{2}} \eta\left( d \right).
\end{align*}
It follows from (\ref{eqn:PE03_3}) and the above equations that
\begin{align}
&\int_{M} \varphi^2 dv =1, \label{eqn:PE04_6}\\
&|B(q,1)| \geq L \geq |B(q, 0.5)| \geq e^{-3D} 2^{-m}|B(q, 1)|. \label{eqn:PE04_7}
\end{align}
Since $\varphi$ satisfies the normalization condition (\ref{eqn:PE04_6}), we can apply the logarithmic isoperimetric inequality (cf. item (ii) of Theorem 1.1 of~\cite{CN09}) to obtain
\begin{align}
&\quad \boldsymbol{\mu}+ m +\frac{m}{2} \log 4\pi \notag\\
&\leq \int_{B(q, 1)} -\varphi^2 \log \varphi^2 dv + \int_{B(q,1) \backslash B(q, 0.5)} 4|\nabla \varphi|^2 dv + \int_{B(q, 1)} R \varphi^2 dv \notag\\
&\leq \log L + \frac{\int_{B(q,1)} \left\{-\eta^2 \log \eta^2 \right\} dv}{L} + \frac{\int_{B(q,1) \backslash B(q, 0.5)}4|\eta'|^2 dv}{L} +R_{max}. \label{eqn:PE04_8}
\end{align}
Recall that $-\eta^2 \log \eta^2 \leq e^{-1}$ and $|\eta'| \leq 4$.
Plugging (\ref{eqn:PE04_7}) into the above inequality, we have
\begin{align*}
\boldsymbol{\mu} +\frac{m}{2} \log 4\pi +m -R_{max} &\leq \log L + \left( e^{-1} +64(1-2^{-m}e^{-3U}) \right) \frac{|B(q,1)|}{L}\\
&\leq \log L + 65 \cdot 2^{m} \cdot e^{3U} \leq \log |B(q,1)| + 2^{m+7} e^{3U}.
\end{align*}
Note that $\displaystyle R_{max}=\max_{B(q,1)} R \leq \max_{B(q,1)} f =f_{max} \leq \frac{D^2}{2}$.
Also, we have
\begin{align*}
U=\sup_{B(q,1)} |\nabla f| \leq \sqrt{\sup_{B(q,1)} f} \leq \frac{D}{\sqrt{2}}.
\end{align*}
According to the choice of $D>10m$, it is also clear that
\begin{align*}
\boldsymbol{\mu} \leq \log |B(q,1)| + e^{m+3D}.
\end{align*}
Consequently, we have
\begin{align*}
|B(q,1)| \geq e^{-e^{m+3D}} \cdot e^{\boldsymbol{\mu}}.
\end{align*}
Now we consider the general case $r \in (0, 1]$. Note that $|\nabla f| <D$ on $B(q,r)$.
It follows from (\ref{eqn:PE03_5}) and the above inequality that
\begin{align*}
r^{-m}|B(q,r)| \geq e^{-3U} |B(q,1)| \geq e^{-e^{m+4D}} \cdot e^{\boldsymbol{\mu}}.
\end{align*}
\end{proof}

Note that Lemma~\ref{lma:PE04_2} improves Lemma 2.3 of~\cite{HM11} by providing explicit non-collapsing constant.
It has an advantage when we study the collapsing case.
Combining Lemma~\ref{lma:PE04_1} and Lemma~\ref{lma:PE04_2}, by setting $r=1$, we have
\begin{align}
e^{-e^{m+4D}} \leq \frac{|B(q,1)|}{e^{\boldsymbol{\mu}}} \leq (4\pi)^{\frac{m}{2}} e^{\frac{1}{2}D^2} \label{eqn:PE04_61}
\end{align}
for each $B(q,1) \subset B(p, D)$. If we choose $q=p$, then the estimate (\ref{eqn:PE04_61}) can be further improved.

\begin{lemma}
The functional $\boldsymbol{\mu}$ and the volume of $B(p,1)$ are related by
\begin{align}
e^{-2^{4m+7}} \leq \frac{|B(p,1)|}{(4\pi)^{\frac{m}{2}} e^{\boldsymbol{\mu}}} \leq e^{m}.
\label{eqn:PE01_9}
\end{align}
In other words, $|B(p,1)|$ and $e^{\boldsymbol{\mu}}$ are almost equivalent.
\label{lma:PE01_3}
\end{lemma}

\begin{proof}
In light of the above estimates, the second inequality in (\ref{eqn:PE01_9}) is relatively easy to prove:
\begin{align}
|B(p,1)|=\int_{B(p,1)}1 d v \leq \int_{B(p,1)} e^{-f+m} dv \leq e^{m} \int_{M} e^{-f} dv=(4\pi)^{\frac{m}{2}} e^{m+\boldsymbol{\mu}}
\label{eqn:PE03_1}
\end{align}
where we used (\ref{eqn:PE01_11a}).
The proof of the first inequality in (\ref{eqn:PE01_9}) is similar to the proof of (\ref{eqn:PE04_5}).
Choose $\eta$ as the same cutoff function as in the proof of Lemma~\ref{lma:PE04_2} and set
\begin{align*}
L \coloneqq \int_{M} \eta^{2}\left( d \right) dv, \quad \varphi \coloneqq L^{-\frac{1}{2}} \eta\left( d \right).
\end{align*}
Then $\int_{M} \varphi^2 dv =1$. It follows from (\ref{eqn:PE04_4}) that
\begin{align}
|B(p,1)| \geq L \geq |B(p, 0.5)| \geq 2^{-4m} |B(p, 1)|. \label{eqn:PE04_9}
\end{align}
Note that $R \leq m$ in $B(p, 1)$ by (\ref{eqn:PE01_11c}). Plugging this inequality and (\ref{eqn:PE04_9}) into (\ref{eqn:PE04_8}), we obtain
\begin{align*}
\boldsymbol{\mu} +\frac{m}{2} \log 4\pi &\leq \log L + \left( e^{-1} +64(1-2^{-4m}) \right) \frac{|B(p,1)|}{L}\\
&\leq \log L + 65 \cdot 2^{4m} \leq \log |B(p,1)| + 2^{4m+7}.
\end{align*}
Consequently, we have
\begin{align*}
|B(p,1)| \geq (4\pi)^{\frac{m}{2}} \cdot e^{-2^{4m+7}} \cdot e^{\boldsymbol{\mu}},
\end{align*}
which is equivalent to the first inequality of (\ref{eqn:PE01_9}).
Combining (\ref{eqn:PE03_1}) with the above inequality, we arrive at (\ref{eqn:PE01_9}).
\end{proof}

Combining (\ref{eqn:PE04_61}) and (\ref{eqn:PE01_9}), we have
\begin{align}
(4\pi)^{-\frac{m}{2}} e^{-m-e^{m+4D}} \leq \frac{|B(q,1)|}{|B(p,1)|} \leq (4\pi)^{\frac{m}{2}} e^{\frac{1}{2}D^2 + 2^{4m+7}}
\label{eqn:PE04_10}
\end{align}
whenever $B(q,1) \subset B(p, D)$. Basically, (\ref{eqn:PE04_10}) means that inside the ball $B(p,D)$, all unit balls have volumes uniformly equivalent to $|B(p, 1)|$.

%


\section{Conformal transformation}
\label{sec:conformal}

The application of conformal transformation in the study of Riemannian geometry is well-known.
For example, see Theorem 1.159, part (b) of~\cite{Besse} for a formula of Ricci curvature change under
conformal transformation. Such formula was exploited by Z. Zhang (cf. Section 3.2 of~\cite{ZhangZL}) to study the degeneration of compact Ricci shrinkers.
We shall follow the notations in~\cite{ZhangZL} and denote
\begin{align}
\tilde{g} \coloneqq e^{-\frac{2f}{m-2}} g. \label{eqn:PE27_61}
\end{align}
If the diameter of $(M,g)$ is uniformly bounded, then $f$ is uniformly bounded by Lemma~\ref{L100}. Therefore, the metric $\tilde{g}$ is uniformly equivalent to $g$.
Moreover, the metric $\tilde{g}$ has the advantage that $\left|\widetilde{Rc} \right|$ is uniformly bounded. Therefore, one can study the geometry of $(M, \tilde{g})$ by Cheeger-Colding theory
and then turn back to study $(M, g)$. However, this method cannot be applied directly to study non-compact Ricci shrinker since $(M, \tilde{g})$ may not be complete, even if $(M, g)$ is very simple.
This can be easily seen in the following Proposition.

\begin{prop}
Let $\left(\R^m, 0,g_{E}, \frac{|x|^2}{4} \right)$ be the Gaussian soliton and $\tilde{g}=e^{-\frac{|x|^2}{2(m-2)}} g_E$.
Then there is a constant $L$ such that $x$ and $-x$ cannot be connected by a minimal geodesic in $(\R^m, \tilde g)$ if $|x|>L$.
\label{prn:PE27_1}
\end{prop}

\def\erf{\mathrm {erf}}
\def\erfc{\mathrm {erfc}}
On $(\R^m, g_E)$ we write $r=|x|$ where $x=(x_1, x_2, \cdots, x_m).$
We denote by $(r, \theta_1, \cdots, \theta_{m-1})$
the polar coordinates, which is defined by
\begin{subequations}
\begin{align*}
&x_1=r\sin\theta_{m-1}\cdots \sin\theta_2\sin\theta_1, \\
&x_2=r\sin\theta_{m-1}\cdots \sin\theta_2\cos\theta_1,\\
&\cdots\\
&x_{m-1}=r\sin\theta_{m-1}\cos\theta_{m-2},\\
&x_m=r\cos\theta_{m-1}.
\end{align*}
\end{subequations}
Here $r\geq 0, 0\leq \theta_1\leq 2\pi, 0\leq \theta_i\leq
\pi$ for $i=2, 3, \cdots, m-1$.
Under the polar coordinate, the metric $g_E$ can be written as
$$g_E=dr^2+r^2 g_{S^{m-1}},$$
where $g_{S^{m-1}}$ is the canonical metric on $S^{m-1}(1).$ Therefore, we can write $\tilde g$ as
$$\tilde g=e^{-\beta \,r^2}(dr^2+r^2g_{S^{m-1}}),$$
where $\beta=\frac 1{2(m-2)}.$
Recall the following result:
\begin{lemma}\label{lem:erfc} (cf. Philip \cite{Philip}) We define
\begin{equation}
\erfc(x)=\frac 2{\sqrt{\pi}}\int^{\infty}_x\,e^{-t^2}\,dt,\quad A(x)=\erfc^{-1}(x), \quad B(x)=\frac 2{\sqrt{\pi}}\,e^{-A(x)^2}.
\end{equation} Then we have

\begin{subequations}
\begin{align}
& A'(x)=-\frac 1B,\quad A''(x)=\frac {2A}{B},\\
&B'(x)=2A,\quad B''(x)=2A',\\
&\lim_{x\rightarrow 0^+}\frac {A(x)}{\sqrt{\log \frac 1x}}=1,\\
&\lim_{x\rightarrow 0^+}\frac {B(x)}{2x\sqrt{\log \frac 1x}}=1.
\end{align}
\end{subequations}
\end{lemma}

\begin{proof}[Proof of Proposition~\ref{prn:PE27_1}:]
Let
$$s=\int_r^{\infty}\,e^{-\frac {\beta}2 \rho^2}\,d\rho=\sqrt{\frac {\pi}{2\beta}}\,\erfc\Big(
\sqrt{\frac {\beta}2}r\Big). $$
Then the metric $\tilde g$ can be written as
\begin{equation}
\tilde g=ds^2+\varphi(s)^2g_{S^{m-1}}, \label{eqn:005}
\end{equation}
where
\begin{equation}\varphi(s)=r\,e^{-\frac {\beta}2 r^2}=\frac 1{a}A(as)B(as), \label{eqn:006} \end{equation}
with $a=\sqrt{\frac {2\beta}{\pi}}.$ Then the function $\varphi$ of (\ref{eqn:005}) satisfies
$$\varphi(0)=0, \quad \varphi'(s)=2A^2-1.$$
By Lemma \ref{lem:erfc}, we have $\lim_{s\rightarrow 0^+}\varphi'(s)=+\infty$.
Therefore, we can find $s_0>0$ such that for any $s\in [0, s_0]$ \begin{equation}\varphi(s)\geq s.\label{eqn:001}\end{equation}
Let $B_{\tilde g}(0, r)$ denote the geodesic ball centered at $s=0$ of radius $r$ with respect to the metric $\tilde g$.
Note that for any two points $p, q\in B_{\tilde g}(0, \frac {s_0}4)$, the geodesic with respect to the metric $\tilde g$ connecting $p$ and $q$ must be contained in $B_{\tilde g}(0, s_0)$.
\def\un{\underline}

Let $ \varphi_0(s)=s$ and $\un g=ds^2+\varphi_0(s)^2g_{S^{m-1}}.$ Since the geodesic distances with respect to $\tilde g$ and $\un g$ from any point $p$ to the point $s=0$ are the same, for any $r>0$ we have $B_{\un g}(0, r)=B_{\tilde g}(0, r)$. Under the polar coordinates, we choose
$$p=(\epsilon, 0, 0, \cdots, 0),\quad q=(\epsilon, \pi, 0, \cdots, 0)$$
where $\epsilon\in (0, \frac {s_0}4)$. Let $\gamma(t)(t\in [0, 1])$ be a geodesic with respect to $\tilde g$ from $p$ to $q$. In the polar coordinates, we write
$$\gamma(t)=(\gamma_0(t), \gamma_1(t), \cdots, \gamma_{m-1}(t)),\quad \forall\;t\in [0, 1]$$ and
$\hat \gamma(t)=(\gamma_1(t), \cdots, \gamma_{m-1}(t))$ the projection of $\gamma(t)$ on $S^{m-1}(1)$.
Assume that $\gamma$ doesn't pass through the point $s=0.$ Let $L_{\tilde g}(\gamma)$ denote the length of the curve $\gamma$ with respect to the metric $\tilde g$. Then
\begin{subequations}
\begin{align}
&L_{\tilde g}(\gamma)=\int_0^1\,\sqrt{\gamma_0\,'(t)^2+\varphi(\gamma(t))^2\,g_{S_{m-1}}(\hat \gamma\,'(t), \hat \gamma\,'(t))}\label{eqn:002a}\\ &\geq \int_0^1\,\sqrt{\gamma_0\,'(t)^2+ \varphi_0(\gamma(t))^2\,g_{S_{m-1}}(\hat \gamma\,'(t), \hat \gamma\,'(t))}=L_{\un g}(\gamma). \label{eqn:002b}
\end{align}
\end{subequations}
where we used the inequality (\ref{eqn:001}) and the definition of $\varphi_0$. Let $\un \gamma(t)(t\in (0, 1))$ be the line connecting $p$ and $q$ with the middle point $s=0$. In other words, we define $\un \gamma=c_1\cup c_2\cup\{s=0\}$ where
\begin{subequations}
\begin{align}
&c_1(t)=\Big((1-2t)\epsilon, 0, \cdots, 0\Big), \quad \forall \;t\in [0, \frac 12),\\&c_2(t)=\Big((2t-1)\epsilon, \pi, 0, \cdot, 0\Big),\quad \forall\, t\in (\frac 12, 1].
\end{align}
\end{subequations}
Then we have
\begin{equation}
L_{\un g}(\gamma)> L_{\un g}(\un \gamma)=L_{\tilde g}(\un \gamma).\label{eqn:003}
\end{equation}
Combining (\ref{eqn:002a})(\ref{eqn:002b}) with (\ref{eqn:003}), we have
$$L_{\tilde g}(\gamma)> L_{\tilde g}(\un \gamma) .$$
Thus, if there is a minimal geodesic connecting $p$ and $q$ in $(B_{\tilde g}(0, s_0), \tilde g)$, then it must pass through the added point $s=0$.
In other words, there is no minimal geodesic in $(\R^m, \tilde g)$ connecting $p$ and $q$.
\end{proof}

Because of Proposition~\ref{prn:PE27_1}, it is better to study conformal transformation locally.
So we introduce a new metric
\begin{align}
\bar{g} \coloneqq e^{\frac{2(f(q)-f)}{m-2}} g, \label{eqn:PE27_7}
\end{align}
where $q \in M$. The metric $\bar{g}$ can be applied to study the geometry around the point $q$.
Accordingly, we define a new potential function
\begin{align}
\bar{f} \coloneqq f-f(q). \label{eqn:PE28_3}
\end{align}
Clearly, we know $\bar{f}(q)=0$. By the gradient estimate (\ref{eqn:PE03_6b}), we have
\begin{align}
- Dr \leq \bar{f}=f-f(q) \leq Dr, \quad \textrm{on} \; B(q,r) \subset B(p, D). \label{eqn:PE06_8}
\end{align}
In the remainder discussion of this section, we always assume (\ref{eqn:PE06_8}) and the condition
\begin{align}
&0<r \leq 1; \label{eqn:PE28_4}\\
&D>10m. \label{eqn:PE28_5}
\end{align}
For simplicity of notations, \textbf{we regard $g$ as the default metric if the metric is not given explicitly}.

\begin{lemma}
We have
\begin{align}
B_{\bar{g}}\left(q, e^{\frac{-Dr}{m-2}} r \right) \subset B(q, r) \subset B_{\bar{g}}\left(q, e^{\frac{Dr}{m-2}}r \right). \label{eqn:PE18_4}
\end{align}

\label{lma:PE04_4}
\end{lemma}

\begin{proof}

Fix $x \in B\left(q, r\right)$ and let $\gamma$ be a unit-speed shortest geodesic connecting $x,q$ under the metric $g$.
It is clear that $\gamma \subset B(q,r)$.
In light of (\ref{eqn:PE06_8}), we have
\begin{align*}
d_{\bar{g}}(x,q) \leq |\gamma|_{\bar{g}} \leq e^{\frac{Dr}{m-2}} |\gamma|_{g} \leq e^{\frac{Dr}{m-2}} d(x,q),
\end{align*}
whence we arrive at
\begin{align}
B(q,r) \subset B_{\bar{g}}\left(q, e^{\frac{Dr}{m-2}}r \right). \label{eqn:PE06_90}
\end{align}
On the other hand, let $\rho$ be the maximum such that $B_{\bar{g}}(q,\rho) \subset B(q,r)$.
According to the choice of $\rho$, we can find a point $y \in \partial B(q,r) \cap \partial B_{\bar{g}}(q,\rho)$.
Let $\bar{\gamma}$ be the shortest geodesic connecting $q$ and $y$ with respect to the metric $\bar{g}$.
Applying (\ref{eqn:PE06_8}) on $\bar{\gamma}$ we obtain
\begin{align*}
\rho=d_{\bar{g}}(q,y)=|\bar{\gamma}|_{\bar{g}} \geq e^{\frac{-Dr}{m-2}} |\bar{\gamma}|_{g} \geq e^{\frac{-Dr}{m-2}} d(q,y)=e^{\frac{-Dr}{m-2}} r.
\end{align*}
Using the definition of $\rho$ again, we obtain
\begin{align}
B_{\bar{g}}\left(q, e^{\frac{Dr}{m-2}}r \right) \subset B(q,r).
\label{eqn:PE06_9}
\end{align}
It is clear that (\ref{eqn:PE18_4}) follows from the combination of (\ref{eqn:PE06_90}) and (\ref{eqn:PE06_9}).
\end{proof}

\begin{lemma}
For each two points $x,y \in B\left(q, 0.1 r \right)$, we have
\begin{align}
e^{\frac{-Dr}{m-2}} d(x,y) \leq d_{\bar{g}}(x,y) \leq e^{\frac{Dr}{m-2}} d(x,y).
\label{eqn:PE06_12}
\end{align}
In particular, for each $\rho <\min\{0.1r, D^{-1}\}$, we have
\begin{align}
d_{GH} \left( B_{\bar{g}}(q, \rho), B(q, \rho) \right)< 2D r^2. \label{eqn:PE04_12}
\end{align}
\label{lma:PE18_1}
\end{lemma}

\begin{proof}
We first show (\ref{eqn:PE06_12}).
Fix $x,y \in B\left(q, 0.1r \right)$ and let $\gamma$ be a unit-speed shortest geodesic connecting $x,y$ under the metric $g$.
Then triangle inequality implies that $\gamma \subset B(q, r)$. In light of (\ref{eqn:PE06_8}), we have
\begin{align}
d_{\bar{g}}(x,y) \leq |\gamma|_{\bar{g}} \leq e^{\frac{Dr}{m-2}} |\gamma|_{g} \leq e^{\frac{Dr}{m-2}} d(x,y).
\label{eqn:PE19_1}
\end{align}
On the other hand, if $x,y \in B\left(q, 0.1r \right)$, then triangle inequality implies that
\begin{align*}
B\left( y, 0.2r \right) \subset B(q,r) \subset B(p, D).
\end{align*}
For the ball $B\left( y, 0.2 r \right)$, similar to (\ref{eqn:PE18_4}), we have
\begin{align}
B_{\bar{g}}\left(y, e^{\frac{-0.2Dr-f(y)+f(q)}{m-2}} r \right) \subset B(y, 0.2r). \label{eqn:PE19_2}
\end{align}
Suppose $x \notin B_{\bar{g}}\left(y, e^{\frac{-0.2Dr-f(y)+f(q)}{m-2}} r \right)$, by the gradient estimate (\ref{eqn:PE03_6b}), we have
\begin{align*}
d_{\bar{g}}(x,y) \geq e^{\frac{-0.2Dr-f(y)+f(q)}{m-2}} r > e^{\frac{-Dr}{m-2}} d(x,y),
\end{align*}
which implies the first inequality in (\ref{eqn:PE06_12}) already.
Therefore, Without loss of generality, we can assume that
\begin{align}
x \in B_{\bar{g}}\left(y, e^{\frac{-0.2Dr-f(y)+f(q)}{m-2}} r \right)
\label{eqn:PE19_3}
\end{align}
Because of (\ref{eqn:PE19_2}) and (\ref{eqn:PE19_3}), there exists a unit-speed shortest geodesic $\bar{\gamma}$ connecting $x$ and $y$ under the metric $\bar{g}$.
Furthermore, we have
\begin{align*}
\bar{\gamma} \subset B_{\bar{g}}\left(y, e^{\frac{-0.2Dr-f(y)+f(q)}{m-2}} r \right) \subset B(y, 0.2r) \subset B(q,r).
\end{align*}
We remind the readers that the existence of $\bar{\gamma}$ is a nontrivial conclusion since in general $\left( M, \bar{g} \right)$ is not a complete manifold.
Applying (\ref{eqn:PE06_8}) on $\bar{\gamma}$ we obtain
\begin{align}
d_{\bar{g}}(x,y)=|\bar{\gamma}|_{\bar{g}} \geq e^{\frac{-Dr}{m-2}} |\bar{\gamma}|_{g} \geq e^{\frac{-Dr}{m-2}} d(x,y).
\label{eqn:PE19_4}
\end{align}
Combining (\ref{eqn:PE19_1}) and (\ref{eqn:PE19_4}), we obtain (\ref{eqn:PE06_12}).

Then we show (\ref{eqn:PE04_12}).
Since $x,y \in B(q,\rho)$, triangle inequality implies that that $d(x,y) \leq 2\rho$.
Since $\rho \in (0, D^{-1})$, it is clear that $\frac{D\rho}{m-2}<\frac{1}{m-2}\leq 0.5$, which implies that
\begin{align*}
\max \left\{ e^{\frac{D\rho}{m-2}}-1, \; 1- e^{-\frac{D\rho}{m-2}} \right\} < \frac{2D\rho}{m-2} \leq D\rho.
\end{align*}
Therefore, it follows from (\ref{eqn:PE06_12}) and the above inequalities that
\begin{align*}
\left| d_{\bar{g}}(x,y) - d(x,y)\right|&< \max \left\{ e^{\frac{D\rho}{m-2}}-1, \; 1- e^{-\frac{D\rho}{m-2}} \right\} d(x,y) \leq 2D\rho^2,
\end{align*}
whence we obtain (\ref{eqn:PE04_12}).
\end{proof}

\begin{lemma}
We have
\begin{align}
\overline{Rc}=\frac{1}{m-2} \left\{ df \otimes df + (m-1-f) e^{\frac{2\bar{f}}{m-2}} \bar{g} \right\}. \label{eqn:PE28_1}
\end{align}
In particular, we have
\begin{align}
\snorm{\overline{Rc}}{\bar{g}} < D^2, \quad \textrm{on} \quad B_{\bar{g}}\left(q, \frac{r}{10D} \right). \label{eqn:PE06_7}
\end{align}
\label{lma:PE06_3}
\end{lemma}

\begin{proof}
Equation (\ref{eqn:PE28_1}) follows from the formula in Theorem 1.159, part (b) of~\cite{Besse}. Similar calculation can be found in section 3.2 of ~\cite{ZhangZL}.
It follows from (\ref{eqn:PE28_1}) that
\begin{align*}
(m-2)^2 \snorm{\overline{Rc}}{\bar{g}}^2&=\snorm{df}{\bar{g}}^{4} + 2(m-1-f)e^{\frac{2\bar{f}}{m-2}}\snorm{df}{\bar{g}}^2 + m(m-1-f)^2 e^{\frac{4\bar{f}}{m-2}}\\
&=e^{\frac{4\bar{f}}{m-2}} \left\{ |\nabla f|_{g}^{4}+ 2(m-1-f) |\nabla f|_{g}^{2} + m(m-1-f)^2 \right\}\\
&\leq e^{\frac{4\bar{f}}{m-2}} \left\{ f^2+ 2(m-1-f) f + m(m-1-f)^2 \right\}\\
&=e^{\frac{4\bar{f}}{m-2}} \left\{ (m-1)^{2} + (m-1)(m-1-f)^{2} \right\}.
\end{align*}
Note that we have used the fact $|\nabla f|^2 <f$ in the above calculation.
It follows from (\ref{eqn:PE06_8}) that
\begin{align*}
|\bar{f}|<D \cdot \frac{r}{10D} \leq 0.1 r \leq 0.1.
\end{align*}
Combining the previous inequalities, we obtain
\begin{align}
\snorm{\overline{Rc}}{\bar{g}} \leq \frac{m-1}{m-2} \left\{ 1+ \frac{|m-1-f|}{\sqrt{m-1}} \right\} \cdot e^{\frac{2}{5(m-2)}}, \label{eqn:PE07_11}
\end{align}
whence we obtain (\ref{eqn:PE06_7}), as $B_{\bar{g}}(q, \frac{r}{10D}) \subset B(q,r)$ and $D>10m$.
\end{proof}

The previous lemmas can be summarized as:

\begin{theorem}
Suppose $(M,p,g,f)$ is a Ricci shrinker, $B(q,r) \subset B(p, D)$ for some $D>10m$ and $r \in (0,1]$.
Then the following properties hold.
\begin{itemize}
\item[(a).] We have
\begin{align}
B_{\bar{g}} \left(q, e^{-\frac{Dr}{m-2}} r \right) \subset B(q,r) \subset B_{\bar{g}}\left(q, e^{\frac{Dr}{m-2}}r \right).
\label{eqn:PE27_8}
\end{align}
Moreover, each ball in the above list is relatively compact in $M$.
\item[(b).] For each $x,y \in B(q, 0.1r)$, we have
\begin{align}
e^{-\frac{Dr}{m-2}} d(x,y) \leq d_{\bar{g}}(x,y) \leq e^{\frac{Dr}{m-2}} d(x,y).
\label{eqn:PE27_9}
\end{align}
\item[(c).] For each $\rho \in (0, D^{-1}r)$, we have
\begin{align}
d_{GH} \left( B_{\bar{g}}(q, \rho), B(q, \rho) \right)< 2D \rho^2.
\label{eqn:PE27_10}
\end{align}
\item[(d).] We have
\begin{align}
\left|\overline{Rc} \right|_{\bar{g}} < D^2, \quad \textrm{on} \; B_{\bar{g}}\left(q, \frac{r}{10D} \right).
\label{eqn:PE27_11}
\end{align}
\end{itemize}
\label{thm:PE27_1}
\end{theorem}


\section{A local regularity theorem}
\label{sec:regularity}

The purpose of this section is to generalize the regularity improvement theorems of Ricci-flat manifolds to Ricci shrinkers.

We first recall some definitions of different radii.
\begin{definition}[\textbf{Volume radius}]
On a Riemannian manifold $(M^m, g)$, we say $vr_{g, \delta}(x) > r$ if
\begin{align}
\omega_m^{-1} r^{-m} |B(x,r)| > 1- \delta. \label{eqn:PH10_1}
\end{align}
The volume radius of $(M, g)$ at $x$ is defined as the supremum of all $r$ such that $vr_{g,\delta}(x) >r$. This volume radius is denoted by $vr_{g,\delta}(x)$.
\label{dfn:PH10_1}
\end{definition}

\begin{definition}[\textbf{Gromov-Hausdorff radius}]
On a Riemannian manifold $(M^m, g)$, we say $gr_{g, \epsilon}(x) > r$ if
\begin{align}
r^{-1} d_{GH} \left( B_{g}(x, r), B(0,r) \right)< \epsilon \label{eqn:PH10_11}
\end{align}
The Gromov-Hausdorff radius of $(M, g)$ at $x$ is defined as the supremum of all $r$ such that $gr_{g,\epsilon}(x) >r$. This Gromov-Hausdorff radius is denoted by $gr_{g,\delta}(x)$.
\label{dfn:PH10_2}
\end{definition}


\begin{definition}[\textbf{Harmonic radius}, cf. p. 436 of~\cite{An90} and Definition 5 of~\cite{HeHe}]
On a Riemannian manifold $(M^m, g)$, we say $hr_{g}(x) > r$ if we can find a harmonic diffeomorphism $\varphi: B(x, r) \to \Omega=\varphi(B(x,r)) \subset \R^m$ such that
\begin{subequations}
\begin{align}[left = \empheqlbrace \,]
&0.5 g_E \leq g \leq 2 g_E, \textrm{as bilinear forms on} \; \Omega;\label{eqn:PD03_2a}\\
&\norm{g}{C^{1,\frac{1}{2}}(\Omega)} \leq 1 \label{eqn:PD03_2b}
\end{align}
\label{eqn:PD03_2}
\end{subequations}
where we have denoted $\varphi_* g$ by $g$ for simplicity of notations.
The inequality (\ref{eqn:PD03_2b}) means that
\begin{align}
r \sum_{|\beta|=1}\sup_{w \in \Omega} |\partial^{\beta} g_{ij}(w)|
+ r^{1.5} \sum_{|\beta|=1} \sup_{y,z \in \Omega, y \neq z} \frac{|\partial^{\beta} g_{ij}(y)-\partial^{\beta} g_{ij}(z)|}{d_{g}^{0.5}(y,z)} \leq 1,
\label{eqn:PD03_3}
\end{align}
where $d_{g}$ is the distance associated to $g$.
The harmonic radius of $(M, g)$ at $x$ is defined as the supremum of all $r$ such that $hr_{g}(x) >r$. This harmonic radius is denoted by $hr_g(x)$.
The subindex $g$ will be omitted when it is clear in the context.
\label{dfn:PC09_1}
\end{definition}

\begin{definition}[\textbf{Strongly-convex radius}]
On a Riemannian manifold $(M^{m}, g)$, we say that $sr_{g}(x)>r$ if the exponential map
\begin{align*}
Exp_{x}: B(0,10r) \subset (\R^{m}, g_{E}) \mapsto Exp_{x}(B(0,10r)) \subset (M^{m}, g)
\end{align*}
is a diffeomorphism such that
\begin{align}
\sum_{k=1}^{5}r^{k} \left\{ \sum_{|\beta|=k} \sup_{w \in B(0,10r)} |\partial^{\beta} h_{ij}(w)| \right\} + \sup_{w \in B(0,10r)} |h_{ij}(w)-\delta_{ij}(w)|<10^{-m}
\label{eqn:PD15_4}
\end{align}
where $h=(Exp_{x})^{*} g$. The strongly-convex radius of $(M,g)$ at $x$ is defined as the supremum of all $r$ such that $sr_{g}(x)>r$. This strongly-convex radius is denoted by $sr_{g}(x)$.
It is also called convex radius for the simplicity of notations.
\label{dfn:PD15_1}
\end{definition}

Let us say a few words about why we call $sr_{g}$ strongly-convex radius.
Suppose $r<sr_{g}(x)$, by (\ref{eqn:PD15_4}), it is easy to see that $B(x,r)$ is strongly-convex.
Namely, for every two points $q_1, q_2 \in B(x,r)$, there exists a unique minimizing geodesic $\gamma \subset B(x,r)$ connecting $q_1$ and $q_2$.
The proof of this fact follows from the standard argument of Riemannian geometry (cf. Section 3.4 of~\cite{DoCarmo}).
Note that the argument is local and requires only $C^3$-regularity of the exponential map. It works here since we have $C^{5}$-regularity of
the exponential map. The constant $5$ is chosen large enough to guarantee the previously mentioned argument, and it can be replaced by any other fixed $k$ greater than $3$.
In all the previous definitions, the subindex $g$ will be omitted when it is clear in the context. On a complete Riemannian manifold, each radius is a positive function.
We introduce the concept of ``domination" to study the relationship of different radii.

\begin{definition}[\textbf{Domination}]
Suppose $r_1$ and $r_2$ are two positive functions on $M^m$.
We say that $r_1$ dominates $r_2$ if we have
\begin{align*}
r_2 < c^{-1} r_1
\end{align*}
for some constant $0<c=c(m)<1$.
We say that $r_1$ and $r_2$ are equivalent if $r_1$ dominates $r_2$ and $r_2$ dominates $r_1$. In other words, there exists a uniform small constant $c=c(m)$ such that
\begin{align*}
c r_1 < r_2 < c^{-1} r_1
\end{align*}

\label{dfn:PG27_1}
\end{definition}

It is interesting to compare the four radii $gr, vr, hr$ and $sr$ on complete Ricci-flat manifolds.
The choice of $\delta$ and $\epsilon$ in Definition~\ref{dfn:PH10_1} and Definition~\ref{dfn:PH10_2} depends on
the fundamental volume convergence theorem of Colding and the gap theorem of Anderson.
We list them below as lemmas for the convenience of the readers.

\begin{lemma}[\textbf{Anderson's gap lemma}, cf. Lemma 3.1 of~\cite{An90}]
Suppose $(N^{m}, g)$ is a complete Ricci-flat, non-flat manifold, then we have
\begin{align}
AVR(N) \coloneqq \lim_{r \to \infty} \omega_{m}^{-1} r^{-m} |B(x,r)| < 1-\delta_0 \label{eqn:PD30_4}
\end{align}
for some $\delta_0=\delta_0(m)$. Note that the $AVR(N)$ in (\ref{eqn:PD30_4}) is the asymptotic volume ratio, which is independent of the choice of $x$ by the Bishop-Gromov volume comparison theorem.
The constant $\delta_0$ in (\ref{eqn:PD30_4}) is called the gap constant of Anderson.
\label{lma:PD30_3}
\end{lemma}

\begin{lemma}[\textbf{Colding's volume continuity}, cf. ~\cite{Colding}]
For each $\delta>0$, there exists an $\epsilon=\epsilon(m, \delta)$ with the following properties.

Suppose $(N^{m}, g)$ is a complete Riemannian manifold, $x \in N$. Suppose
\begin{align}
& Rc \geq -(m-1), \quad \textrm{on} \; B(x, 1); \label{eqn:PH11_1}\\
& d_{GH} \left( B_{g_E}(0,1), B(x,1)\right)<\epsilon. \label{eqn:PH11_2}
\end{align}
Here $B_{g_E}(0,1)$ is the standard unit ball in the Euclidean space $\R^{m}$. Then we have
\begin{align*}
\omega_{m}^{-1} \cdot 10^{m} \cdot |B(x, 0.1)| \geq 1-\delta.
\end{align*}
\label{lma:PH11_1}
\end{lemma}

In the remainder discussion of this section, we shall fix $\delta$ as $0.1 \delta_0$ where $\delta_0$ is Anderson's gap constant in Lemma~\ref{lma:PD30_3}.
Correspondingly, we fix $\epsilon$ as $\epsilon(m,\delta)$ by Lemma~\ref{lma:PH11_1}.
Then the following equivalence theorem is well-known.
\begin{theorem}
Suppose $(M^m,g)$ is a complete Ricci-flat manifold. Then $gr, vr, hr$ and $sr$ are all equivalent (cf. Definition~\ref{dfn:PG27_1}).
\label{thm:PH11_1}
\end{theorem}
\begin{table}[hh]
\begin{displaymath}
\xymatrix{ & gr \ar@{->}[d]^{1} & & sr\ar@{->}[ll]_{4}\\
& vr \ar@{->}[rr]^{2}& & \ar@{->}[u]_{3} hr } \\
\end{displaymath}
\caption{The relationship among radii on a Ricci-flat manifold}
\label{table:Ricci-flat-radii}
\end{table}

\begin{proof}[Sketch of the proof:]
The proof of Theorem~\ref{thm:PH11_1} can be described by Table~\ref{table:Ricci-flat-radii}.
They are well-known in the literature. So we only give a sketch of the proof and leave the details to be filled in by the interested readers.
In Table~\ref{table:Ricci-flat-radii}, $\to$ means ``is dominated by". The first column of Table~\ref{table:Ricci-flat-radii}, or step $1$, means that $gr$ is dominated by $vr$.
This step is nothing but Colding's volume continuity theorem (cf. Lemma~\ref{lma:PH11_1}).
Step 2 means $vr$ is dominated by $hr$, which can be obtained by a classical blowup argument together with Anderson's gap theorem (cf. Lemma~\ref{lma:PD30_3}).
Step 3 says $hr$ is dominated by $sr$, and this follows from the standard Schauder estimate since the Ricci-flat equation is a uniformly elliptic system of metrics under the harmonic coordinate chart.
Step 4 means $sr$ is dominated by $gr$. It follows from the definition and direct calculations.
After we run through the circle $1234$, it is clear that all of the radii: $gr,vr,hr$ and $sr$ are all equivalent.
\end{proof}
We now move on to study the relationship among different radii on Ricci shrinkers.
From the discussion in Section~\ref{sec:conformal}, we know that under the scale
$\frac{1}{10D}$, the Riemannian metric $g$ is almost isometric to $\bar{g}=e^{-\frac{2\bar{f}}{m-2}} g$,
which has bounded Ricci curvature (cf. Theorem~\ref{thm:PE27_1}).
Therefore, it is reasonable to consider the restricted radii for both $g$ and $\bar{g}$.
To distinguish from the classical radii, we write all restricted radii in the bold symbol.

\begin{definition}
We define
\begin{align}
&\mathbf{gr}(x) \coloneqq \sup \left\{ r \left| 0<r<\frac{1}{100D}, \quad r^{-1} d_{GH}\left( B(0,r), B_{g}(x,r) \right)<\frac{\epsilon}{100} \right. \right\}, \label{eqn:PH02_1} \\
&\mathbf{vr}(x) \coloneqq \sup \left\{ r \left| 0<r<\frac{1}{100D}, \quad \omega_{m}^{-1}r^{-m} |B_{g}(x,r)|_{g} \geq 1-\frac{\delta}{100} \right. \right\}, \label{eqn:PH02_2}\\
&\mathbf{hr}(x) \coloneqq \max\left\{ hr_{g}(x), \frac{1}{100D} \right\}, \label{eqn:PH02_3} \\
&\mathbf{sr}(x) \coloneqq \max \left\{ sr_{g}(x), \frac{1}{100D} \right\}. \label{eqn:PH02_4}
\end{align}
Similarly, we define
\begin{align}
&\overline{\mathbf{gr}}(x) \coloneqq \sup \left\{ r \left| 0<r<\frac{1}{100D}, \quad r^{-1} d_{GH}\left( B(0,r), B_{\bar{g}}(x,r) \right)<\epsilon \right. \right\},\label{eqn:PH02_5} \\
&\overline{\mathbf{vr}}(x) \coloneqq \sup \left\{ r \left| 0<r<\frac{1}{100D}, \quad \omega_{m}^{-1}r^{-m} |B_{\bar{g}}(x,r)|_{\bar{g}} \geq 1-\delta \right. \right\},
\label{eqn:PH02_6}\\
&\overline{\mathbf{hr}}(x) \coloneqq \max\left\{ hr_{\bar{g}}(x), \frac{1}{100D} \right\}, \label{eqn:PH02_7} \\
&\overline{\mathbf{sr}}(x) \coloneqq \max \left\{ sr_{\bar{g}}(x), \frac{1}{100D} \right\}. \label{eqn:PH02_8}
\end{align}
\label{dfn:PH02_1}
\end{definition}
The following theorem is the main theorem of this section.
\begin{theorem}
Suppose $(M^m,p,g,f)$ is a Ricci shrinker, $D(x)=d(x,p)+10m$. Then the radii $\mathbf{gr}, \overline{\mathbf{gr}}, \mathbf{vr}, \overline{\mathbf{vr}}, \mathbf{hr}, \overline{\mathbf{hr}}, \mathbf{sr}, \overline{\mathbf{sr}}$
are all equivalent (cf. Definition~\ref{dfn:PG27_1}).
\label{thm:PH10_1}
\end{theorem}

\begin{table}[h]
\begin{displaymath}
\xymatrix{
& \mathbf{gr} \ar@{->}[ddd]_{1} & & &\\
&&&&\\
& &\mathbf{vr} \ar@{->}[d]_{8} & \mathbf{hr}\ar@{->}[l]_{7} \ar@{->}[lluu]_{9} &\mathbf{sr}\ar@{->}[l]_{6} \\
& \overline{\mathbf{gr}} \ar@{->}[r]_{2}&\overline{\mathbf{vr}} \ar@{->}[r]_{3} & \overline{\mathbf{hr}} \ar@{->}[r]_{4} &\overline{\mathbf{sr}} \ar@{->}[u]_{5}} \\
\end{displaymath}
\caption{The relations among radii on a Ricci shrinker}
\label{table:Ricci-shrinker-radii}
\end{table}

The following lemma is the key estimate for proving Theorem~\ref{thm:PH10_1}.

\begin{lemma}
There is a small positive constant $c=c(m)$ such that
\begin{align}
\overline{\mathbf{sr}}(q) > c \cdot \overline{\mathbf{hr}}(q) \quad \forall \; q \in M. \label{eqn:PH07_1}
\end{align}
\label{lma:PH04_1}
\end{lemma}

\begin{proof}
We rescale the metric by setting
\begin{align}
\hat{g} \coloneqq \rho^{-2}\bar{g}=\rho^{-2} e^{-\frac{2\bar{f}}{m-2}}g, \label{eqn:PH14_1}
\end{align}
where $\rho=\overline{\mathbf{hr}}(x) \leq \frac{1}{100D}$,
$\bar{f}=f-f(q)$. Then it follows from \eqref{eqn:PC06_2}, \eqref{eqn:PC06_3} and \eqref{eqn:PE28_1} that the following identities hold
under the metric $\hat{g}$:
\begin{subequations}
\begin{align}[left = \empheqlbrace \,]
&\widehat{Rc}=\frac{1}{m-2} \left\{ d\bar{f} \otimes d\bar{f} + \rho^2 \left(m-1-\bar{f}-f(q) \right) e^{\frac{2\bar{f}}{m-2}} \hat{g} \right\}, \label{eqn:PH04_3A}\\
&\hat{\Delta} \bar{f}=e^{\frac{2\bar{f}}{m-2}} \rho^2 \left( \frac{m}{2} -\bar{f} - f(q)\right ). \label{eqn:PH04_3B}
\end{align}
\label{eqn:PH04_3}
\end{subequations}
It is clear that
\begin{align}
&|\hat{\nabla} \bar{f}| + |\bar{f}| \leq 100, \quad \textrm{in} \; B_{\hat{g}}(q, 1); \label{eqn:PH06_1}\\
&\rho^2 f(q) \leq 100. \label{eqn:PH06_2}
\end{align}
It follows from the definition of harmonic radius that we can find a harmonic diffeomorphism
$$\varphi: B_{\hat{g}}(q,1) \to \Omega=\varphi(B(q,1)) \subset \R^{m}$$
such that
\begin{align*}
& 0.5 g_{E} \leq \hat{g} \leq 2 g_{E}, \\
&\norm{\hat{g}}{C^{1,\frac12}(\Omega)} \leq 1.
\end{align*}
Express metrics $\hat{g}$ under this harmonic coordinate, we have $\hat{\Gamma}^{i}=\hat{g}^{jk} \hat{\Gamma}_{jk}^{i}=0$.
Consequently, one can express the Ricci curvature tensor as matrix valued functions (cf. Lemma 4.1 of Deturck-Kazdan~\cite{DeKa}) satisfying
\begin{align}
\widehat{Rc}_{ij}=-\frac{1}{2}\hat{g}^{kl} \frac{\partial^2 \hat{g}_{ij}}{\partial x^k \partial x^l} + P\left( \frac{\partial \hat{g}_{uv}}{\partial x^w}\right) \label{eqn:PC21_1}
\end{align}
where $P$ is a quadratic term of the first derivatives of $\hat{g}$. Therefore, we can write the system (\ref{eqn:PH04_3}) as
\begin{subequations}
\begin{align}[left = \empheqlbrace \,]
&\hat{g}^{kl} \frac{\partial^2 \hat{g}_{ij}}{\partial x^k \partial x^l}=P_1\left( \frac{\partial \hat{g}_{uv}}{\partial x^w}, \frac{\partial \bar{f}}{\partial x^{a}}\right)
+\frac{2\rho^2(\bar{f}+f(q)-m+1)}{m-2} e^{\frac{2\bar{f}}{m-2}} \hat{g}_{ij}, \label{eqn:CA25_9a}\\
&\hat{g}^{kl} \frac{\partial^2 \bar{f}}{\partial x^k \partial x^l}=P_2\left( \frac{\partial \hat{g}_{uv}}{\partial x^w}, \frac{\partial \bar{f}}{\partial x^{a}}\right)
+ e^{\frac{2\bar{f}}{m-2}} \rho^2 \left( \frac{m}{2}-\bar{f}-f(q) \right). \label{eqn:CA25_9b}
\end{align}
\label{eqn:CA25_9}
\end{subequations}
Here $P_1, P_2$ are quadratic terms of the first derivatives of $\bar{g}$ and $\bar{f}=f-f(q)$. Now we run the bootstrapping argument for the system (\ref{eqn:CA25_9}).
For simplicity of notations, we denote $0.5$ by $\alpha$.
By (\ref{eqn:PH06_1}), we have uniform $C^1$-norm of $\bar{f}$.
Together with (\ref{eqn:PH06_2}), it implies that the right hand side of (\ref{eqn:CA25_9b}) is uniformly bounded.
Since $\hat{g}$ is $C^{1,\alpha}$, applying standard Schauder estimate (cf. for example, Theorem 6.17 of Gilbarg-Trudinger~\cite{GT01} and Theorem 6.2.6 of~\cite{Morrey}) on (\ref{eqn:CA25_9b}), we obtain
$\bar{f}$ is $C^{2,\alpha}$.
Plugging this estimate into the right hand side of (\ref{eqn:CA25_9a}) and noting that $\hat{g}$ is $C^{1,\alpha}$ by assumption, we obtain the the right hand side of (\ref{eqn:CA25_9a}) is $C^{\alpha}$.
It follows from Schauder estimate that $\hat{g}$ is $C^{2,\alpha}$.
So we know that the couple $\left( \hat{g}, \bar{f} \right)$ is $C^{2,\alpha}$.
We move on to study (\ref{eqn:CA25_9b}) and obtain $\bar{f}$ is $C^{3,\alpha}$.
Then we go to (\ref{eqn:CA25_9a}) to obtain that $\hat{g}$ is $C^{3,\alpha}$.
In short, the pair $(\hat{g}, \bar{f})$ is $C^{3,\alpha}$.
Following the order of (\ref{eqn:CA25_9b}), (\ref{eqn:CA25_9a}) and (\ref{eqn:CA25_9b}) again, the Schauder estimates tell us that the pair $(\hat{g}, \bar{f})$ is $C^{4,\alpha}$.
Repeating this process routinely, we obtain all the higher order derivatives estimate of $(\hat{g}, \bar{f})$.
In particular, we can find a constant $C=C(m)$ such that
\begin{align}
\norm{\bar{f}}{C^{5}(\Omega')} + \norm{\hat{g}}{C^{5}(\Omega')} < C, \label{eqn:PH07_2}
\end{align}
where $\Omega'=\varphi(B_{\hat{g}}(q, 0.5))$.
Note that $B(0, 0.1) \subset B_{\hat{g}}(q, 0.5)$. Therefore, we can find a small constant $c=c(m,C)=c(m)$ such that
\begin{align*}
\sum_{k=1}^{5} r^{k} \sum_{|\beta|=k} \sup_{B(0, 10r)} |\partial^{\beta} \hat{g}_{ij}| + \sup_{B(0, 10r)} |\hat{g}_{ij}-\delta_{ij}| < 10^{-m}, \quad \forall \; r \in (0, c].
\end{align*}
Therefore, by (\ref{eqn:PD15_4}), it follows from Definition~\ref{dfn:PD15_1} that
$$sr_{\hat{g}}(q) \geq c,$$
which is equivalent to (\ref{eqn:PH07_1}) by Definition~\ref{dfn:PH02_1}.
\end{proof}

We are now ready to prove Theorem~\ref{thm:PH10_1}. We shall follow the route described by Table~\ref{table:Ricci-shrinker-radii}.

\begin{proof}

\textit{Step 1: $\mathbf{gr} \to \overline{\mathbf{gr}}$.}

Note that there is a gap between constants in (\ref{eqn:PH02_1}) and (\ref{eqn:PH02_5}).
Let $r=\mathbf{gr}(q)$.
It follows from (\ref{eqn:PE27_10}) in Theorem~\ref{thm:PE27_1} and the triangle inequality that
\begin{align*}
d_{GH} \left( B_{\bar{g}}(q,r), B(0,r)\right)&<d_{GH} \left( B_{\bar{g}}(q,r), B_{g}(q,r)\right)+d_{GH}\left( B_{g}(q,r), B(0,r)\right)\\
&<\left(2D r + \frac{\epsilon}{100} \right) r < \epsilon r,
\end{align*}
which means that
\begin{align*}
\overline{\mathbf{gr}}(x) \geq r=\mathbf{gr}(q).
\end{align*}

\textit{Step 2: $\overline{\mathbf{gr}} \to \overline{\mathbf{vr}}$.}
This step is basically the application of Colding's volume convergence theorem.
Let $r=\overline{\mathbf{gr}}(q)$ and let $g' \coloneqq r^{-2}\bar{g}$.
It follows from Definition~\ref{dfn:PH02_1} that
\begin{align*}
d_{GH} \left( B_{g'}(q,1), B(0,1)\right)<\epsilon.
\end{align*}
Since $r \leq \frac{1}{100D}$, by part (d) of Theorem~\ref{thm:PE27_1}, it is clear that
\begin{align}
|Rc'|_{g'}<D^2r^2<1, \quad \textrm{on} \; B_{g'}(q,1) \subset B_{g'}\left(q, \frac{1}{10Dr} \right). \label{eqn:PH10_2}
\end{align}
Therefore, we can apply the volume convergence theorem of Colding (cf. Lemma~\ref{lma:PH11_1}) to obtain
\begin{align*}
\omega_{m}^{-1} \cdot 10^{m} \cdot |B_{g'}(q,0.1)|_{dv_{g'}} \geq 1-0.01\delta.
\end{align*}
It follows from Definition~\ref{dfn:PH10_1} that
\begin{align*}
vr_{g'}(q) \geq 0.1,
\end{align*}
which is equivalent to
\begin{align*}
\overline{\mathbf{vr}} \geq 0.1 \overline{\mathbf{gr}}.
\end{align*}
So we finish the proof of Step 2.

\textit{Step 3: $\overline{\mathbf{vr}} \to \overline{\mathbf{hr}}$.}
Let $r=\overline{\mathbf{vr}}(q)$ and let $g' \coloneqq r^{-2}\bar{g}$. We have the local Ricci curvature estimate (\ref{eqn:PH10_2}).
Following the classical blowup argument in the proof of Theorem 3.2 of Anderson~\cite{An90}, we have
\begin{align*}
hr_{g'}(q) \geq c
\end{align*}
for some $c=c(m)$. The above inequality is the same as
\begin{align*}
\overline{\mathbf{hr}} \geq c \cdot \overline{\mathbf{vr}}.
\end{align*}
Therefore, $\overline{\mathbf{vr}}$ dominates $\overline{\mathbf{hr}}$ and the proof of Step 3 is finished.

\textit{Step 4: $\overline{\mathbf{hr}} \to \overline{\mathbf{sr}}$.}

This is the key estimate. It follows directly from Lemma~\ref{lma:PH04_1}.

\textit{Step 5: $\overline{\mathbf{sr}} \to \mathbf{sr}$.}
Let $r=\overline{\mathbf{sr}}$, $\bar{g}'=r^{-2}\bar{g}$ and $g'=r^{-2}g$. We can follow the regularity improvement argument in Lemma~\ref{lma:PH04_1} to obtain that $(\bar{g}', \bar{f})$ is a couple with uniform
regularity on $B(0, 0.1)$, through the pull-back of exponential map. Note that $g'=e^{\frac{2\bar{f}}{m-2}} \bar{g}'$. It is clear that under the same coordinate chart, $\norm{g'}{C^k(B(0, 0.1))}$ is uniformly bounded for each $k$.
Then it is easy to see that
\begin{align*}
sr_{g'} \geq c
\end{align*}
for some $c=c(m)$. Consequently, we have $\mathbf{sr} \geq c \cdot \overline{\mathbf{sr}}$ and finish the proof of step 5.

\textit{Step 6: $\mathbf{sr} \to \mathbf{hr}$.}

It follows from the standard construction of harmonic chart when curvature and injectivity radius are all bounded (cf. S.Peters~\cite{SPeters} and Green-Wu~\cite{GreenWu}).

\textit{Step 7: $\mathbf{hr} \to \mathbf{vr}$.}

Trivial.

\textit{Step 8: $\mathbf{vr} \to \overline{\mathbf{vr}}$.}
This follows from the volume ratio estimate (cf. (\ref{eqn:PE27_7}), (\ref{eqn:PE28_3}) and (\ref{eqn:PE27_9})). Note that there is a gap of constants between (\ref{eqn:PH02_2}) and (\ref{eqn:PH02_6}).

\textit{Step 9: $\mathbf{hr} \to \mathbf{gr}$.}

Trivial.\\

Now we chase Table~\ref{table:Ricci-shrinker-radii}. Running through the circles $1234569$ and $345678$ in Table~\ref{table:Ricci-shrinker-radii}, we obtain the equivalence
among $\mathbf{gr}, \overline{\mathbf{gr}}, \mathbf{vr}, \overline{\mathbf{vr}}, \mathbf{hr}, \overline{\mathbf{hr}}, \mathbf{sr}, \overline{\mathbf{sr}}$.
The proof of Theorem~\ref{thm:PH10_1} is complete.
\end{proof}

It is easy to see from Definition~\ref{dfn:PC09_1} that harmonic radius satisfies local Harnack inequality
\begin{align*}
\frac{1}{100} r < hr(y)< 100 r, \quad \forall \; y \in B(x, 0.1r), \; r=hr(x).
\end{align*}
On a Ricci shrinker, such inequality naturally bypass (cf. Definition~\ref{dfn:PH02_1}) to
\begin{align}
c r <\mathbf{hr}(y)< \frac{r}{c}, \quad \forall \; y \in B(x, c r), \; r=\mathbf{hr}(x) \label{eqn:PI10_2}
\end{align}
for some small positive constant $c=c(m)$. In terms of the equivalence of the different radii on Ricci shrinkers, we know
$\mathbf{hr}$ in (\ref{eqn:PI10_2}) can be replaced by any other radius appeared in Theorem~\ref{thm:PH10_1}.

\begin{theorem}
There exists a small constant $\xi=\xi(m)>0$ with the following properties.

Suppose $(M, p, g, f)$ is a complete Ricci shrinker, $x \in M$, $D=d(x,p)+10m$.
Suppose $r \in (0, \frac{1}{100D})$ and
\begin{align}
r^{-1} d_{PGH} \left( B(x,r), B(0,r) \right) < \xi.
\label{eqn:PD15_5}
\end{align}
Then the geodesic ball $B(x,\xi r)$ is strongly convex. Furthermore, the exponential map $Exp_{x}: B(0, \xi r) \to M$ is a diffeomorphism such that
\begin{align}
r^k \sum_{|\beta|=k} \sup_{w \in B(0, \xi r)} \left\{ |\partial^{\beta} h_{ij}(w)| + |\partial^{\beta} \bar{l}| \right\} <C_k, \quad \forall \; k \in \Z^{+}, \label{eqn:PH11_3}
\end{align}
where $h=(Exp_{x})^{*} g, \; \bar{l}=(Exp_{x})^{*} \bar{f}$, and $C_k$ is a constant depending only on $m$ and $k$.
\label{thm:PD14_1}
\end{theorem}

\begin{proof}
Let $\xi$ be a very small number whose size will be decided later.
By choosing $\xi<\epsilon$ and Definition~\ref{dfn:PH10_2}, we obtain that $\mathbf{gr}(x) \geq r$.
Note that the constant $\epsilon$ is the one chosen after Lemma~\ref{lma:PH11_1}.
It follows from Theorem~\ref{thm:PH10_1} that $\mathbf{gr}$ and $\mathbf{sr}$ are equivalent.
Therefore, we have
\begin{align*}
\mathbf{sr}(x) \geq c(m) \mathbf{gr}(x)
\end{align*}
for some small positive number $c(m)<1$. Let $\xi=0.01 c(m) \epsilon$. Then we have
\begin{align*}
\mathbf{sr}(x) \geq 100 \xi \mathbf{gr}(x) \geq 100 \xi r.
\end{align*}
Then the estimate (\ref{eqn:PH11_3}) follows from the above inequality (cf. Definition~\ref{dfn:PD15_1}) and the regularity improvement argument in the proof of Lemma~\ref{lma:PH04_1}.
\end{proof}

\begin{remark}
It is a key point to exploit the harmonic coordinate chart of $\bar{g}$ or $\hat{g}=\rho^{-2} \bar{g}$, rather than $g$, to improve regularity.
Actually, in the harmonic coordinate chart of $g$, the corresponding system of (\ref{eqn:PH04_3}) is
\begin{subequations}
\begin{align}[left = \empheqlbrace \,]
&Rc=-\emph{Hess} f +\frac{g}{2}, \notag\\
&\Delta f =\frac{m}{2}-R. \notag
\end{align}
\notag
\end{subequations}
The bootstrapping cannot start if we only have the $C^{1}$-norm bound of $f$.
\label{rmk:PI07_2}
\end{remark}


\section{Relations between weak limits}
\label{sec:weakexistence}

The purpose of this section is to show a general weak compactness theorem of Ricci shrinkers without any non-collapsing assumption and locally relate the weak limit with that of their conformal transformations.

Let us first recall some definitions, which should be well-known to experts in geometric analysis.

\begin{definition}
Suppose $(M_i, p_i, d_i)$ and $(M_{\infty}, p_{\infty}, d_{\infty})$ are complete length spaces.
By the convergence
\begin{align}
\left( M_i, p_i, d_i \right) \longright{pointed-Gromov-Hausdorff} \left( M_{\infty}, p_{\infty}, d_{\infty} \right),
\label{eqn:PC24_3}
\end{align}
we mean that $(M_i, p_i, d_i)$ converges to $(M_{\infty}, p_{\infty}, d_{\infty})$ in the pointed Gromov Hausdorff topology.
Namely, for each fixed $L>0$, we have
\begin{align*}
\left( B(p_i,L), p_i, d_i \right) \longright{Gromov-Hausdorff} (B(p_{\infty},L), p_{\infty}, d_{\infty}).
\end{align*}
Suppose $f_i$ are real valued functions on $M_i$, $f_{\infty}$ is a real valued function on $M_{\infty}$.
By the convergence
\begin{align}
\left( M_i, p_i, d_i, f_{i} \right) \longright{pointed-Gromov-Hausdorff} \left( M_{\infty}, p_{\infty}, d_{\infty}, f_{\infty} \right),
\label{eqn:PC24_4}
\end{align}
we mean (\ref{eqn:PC24_3}) and that $f_{\infty}$ is a limit function of $f_i$. Namely, if $y \in M_{\infty}$ is the limit point of $y_i \in M_i$, then
\begin{align}
f_{\infty}(y)=\lim_{i \to \infty} f_i(y_i). \label{eqn:PC24_5}
\end{align}
\label{dfn:PC24_1}
\end{definition}

For the concept of Gromov-Hausdorff convergence, length space, etc, one can find their definitions, e.g., in the excellent textbook~\cite{BuIv}.
Note that in order the limit function $f_{\infty}$ to be well defined, $f_i$ need to satisfy uniform continuity. For otherwise, it is possible that there are two sequences
$\left\{ y_i \right\}$ and $\left\{ y_i' \right\}$ such that
\begin{align*}
\lim_{i \to \infty} y_i =\lim_{i \to \infty} y_i', \quad \lim_{i \to \infty} f_i(y_i) \neq \lim_{i \to \infty} f_i(y_i').
\end{align*}

It is important to note that there is an equivalent definition of (\ref{eqn:PC24_3}). Namely, we can explain (\ref{eqn:PC24_3}) as for each $L>0$, there exist positive numbers $\epsilon_{L,i} \to 0$ and maps
\begin{align}
\psi_{L,i}: \quad B(p_{\infty}, L) \to B(p_{i}, L+\epsilon_{L}(i))
\label{eqn:PD16_2}
\end{align}
such that
\begin{subequations}
\begin{align}[left = \empheqlbrace \,]
&d \left(\psi_{L,i}(p_{\infty}),p_i \right)<\epsilon_{L,i}; \label{eqn:PD16_3a}\\
&\left| d\left( \psi_{L,i}(x), \psi_{L,i}(y) \right) -d(x,y) \right|<\epsilon_{L,i}, \quad \forall\; x, y \in B(p_{\infty}, L). \label{eqn:PD16_3b}
\end{align}
\label{eqn:PD16_3}
\end{subequations}

\noindent
Each map $\psi_{L,i}$ is called an $\epsilon_{L,i}$-isometry.
For the proof and related definitions, see Definition 7.3.27 and Corollary 7.3.28 of~\cite{BuIv}.
We remind the readers that $\psi_{L,i}$ need not be continuous.

\begin{definition}
Let $(M_{\infty}, d_{\infty})$ be a complete length space, $y \in M_{\infty}$. $(\hat{Y}, \hat{y}, \hat{d})$ is called a tangent space
of $M_{\infty}$ at $y$ if there is a sequence of positive
numbers $r_i \to 0$ such that
\begin{align*}
\left( M_{\infty}, y, r_i^{-1}d_{\infty} \right) \longright{pointed-Gromov-Hausdorff} \left(\hat{Y}, \hat{y}, \hat{d} \right).
\end{align*}
Notice that the tangent space at $y$, if it exists, may not be unique. A point $y \in M_{\infty}$ is called a regular point if all tangent spaces at $y$ is $(\R^{l}, 0, d_{E})$ for some positive integer $l$. A point $y$ is called a singular point if it is not regular.
The collection of regular points in $M_{\infty}$ is denoted by $\mathcal{R}$. The collection of singular points in $M_{\infty}$ is denoted by $\mathcal{S}$.
\label{dfn:PC24_2}
\end{definition}

Next, we prove

\begin{proposition}\label{prn:PC25_1}
Let $(M_i^m, p_i, g_i,f_i)$ be a sequence of Ricci shrinkers.
Let $d_i$ be the length structure induced by $g_i$. By passing to a subsequence if necessary, we have
\begin{align}
(M_i, p_i, d_i,f_i) \longright{pointed-Gromov-Hausdorff} \left(M_{\infty}, p_{\infty}, d_{\infty}, f_{\infty} \right), \label{eqn:PC25_5}
\end{align}
where $(M_{\infty}, d_{\infty})$ is a length space, $f_{\infty}$ is a locally Lipschitz function on $(M_{\infty}, d_{\infty})$.
Furthermore, the limit space $M_{\infty}$ has a natural regular-singular decomposition (cf. Definition~\ref{dfn:PC24_2}):
\begin{align}\label{eqn:PE06_1}
M_{\infty}= \mathcal{R} \cup \mathcal{S}.
\end{align}
\end{proposition}

\begin{proof}
In light of (\ref{eqn:CB09_2}) in Lemma~\ref{lma:PE04_1}, and (\ref{eqn:PE04_5}) in Lemma~\ref{lma:PE04_2}, we have volume ratio two-sided bounds.
Then it follows from the standard ball packing argument of Gromov (cf. Theorem 2.4 of ~\cite{HM11} and Proposition 5.2 of~\cite{Gromov}) that
\begin{align}
\left( M_i, p_i, d_i \right) \longright{pointed-Gromov-Hausdorff} \left(M_{\infty}, p_{\infty}, d_{\infty} \right). \label{eqn:CB06_1}
\end{align}
On the other hand, in view of (\ref{eqn:PC06_3}), (\ref{E106}) and the fact $R \geq 0$, we have
\begin{align}
|\nabla f|(x) \leq \sqrt{f}(x) \leq \frac{1}{2} \left( d(x,p) + \sqrt{2m} \right). \label{eqn:PC25_6}
\end{align}
This guarantees that $f_i$ converges to a limit function $f_{\infty}$ which is locally Lipschitz on $M_{\infty}$.
Therefore, we have proved (\ref{eqn:PC25_5}) by the definition (cf. Definition~\ref{dfn:PC24_1}).
\end{proof}

In the situation of Proposition \ref{prn:PC25_1}, we have

\begin{proposition}\label{prn:PE06_1}
Suppose $B(q_i,1) \subset B(p_i, D)$ for some $D>10m$.
Then by taking subsequence if necessary, we have
\begin{align}
\left( B_{\bar{d}_i}(q_i, 20\epsilon), q_i, d_i \right) \longright{pointed-Gromov-Hausdorff}
\left( B_{\bar{d}_{\infty}}(\bar{q}_{\infty}, 20\epsilon), \bar{q}_{\infty}, \bar{d}_{\infty} \right),
\label{eqn:PE06_3}
\end{align}
where $\epsilon=\frac{1}{100D}$.
The limit space $B_{\bar{d}_{\infty}}(\bar{q}_{\infty}, 20\epsilon)$ also has a regular-singular decomposition
\begin{align}\label{eqn:PE06_4}
B_{\bar{d}_{\infty}}(\bar{q}_{\infty}, 20\epsilon)= \bar{\mathcal{R}} \cup \bar{\mathcal{S}}.
\end{align}
\end{proposition}

\begin{proof}
Since $\epsilon=\frac{1}{100D}$, by (\ref{eqn:PE27_11}) in Theorem~\ref{thm:PE27_1},
there exists a uniform Ricci curvature bound on $B_{\bar{d}_i}(y_i, 100\epsilon)$. Therefore, it follows from standard Cheeger-Colding theory that
(\ref{eqn:PE06_3}) and (\ref{eqn:PE06_4}) hold.
\end{proof}

Under the same assumptions of Proposition \ref{prn:PC25_1} and Proposition \ref{prn:PE06_1} we prove

\begin{proposition}\label{prn:PE29_1}
There exists a bi-Lipschitz homeomorphism map $\Pi_{\infty}$ satisfying
\begin{align}
\Pi_{\infty}: B_{d_{\infty}}(q_{\infty}, 10\epsilon) &\mapsto \Pi_{\infty} \left( B_{d_{\infty}}(q_{\infty}, 10\epsilon)\right) \subset B_{\bar{d}_{\infty}}(\bar{q}_{\infty}, 20\epsilon), \label{eqn:PE06_5aa}\\
q_{\infty}&\mapsto \bar{q}_{\infty}=\Pi_{\infty}(q_{\infty}). \label{eqn:PE06_5a}
\end{align}
For each pair of points $z,w \in B_{d_{\infty}}(q_{\infty}, 10\epsilon)$, denoting $\Pi_{\infty}(z)$ and $\Pi_{\infty}(w)$ by $\bar{z}$ and $\bar{w}$ respectively, we have
\begin{align}
e^{-\frac{10 D\epsilon}{m-2}} \bar{d}_{\infty}(\bar{z}, \bar{w}) \leq \rho \coloneqq d_{\infty}(z,w) \leq e^{\frac{10 D\epsilon}{m-2}} \bar{d}_{\infty}(\bar{z}, \bar{w}). \label{eqn:PE28_7}
\end{align}
For each $z \in B(q_{\infty}, 10\epsilon)$ and $r \in (0, 10\epsilon)$, we have
\begin{align}\label{eqn:PE28_8}
d_{GH} \left( B_{d_{\infty}}(z,r), B_{e^{\frac{f_{\infty}(z)-f_{\infty}(q_{\infty})}{m-2}}\bar{d}_{\infty}}(\bar{z}, r) \right) \leq 2D r^2.
\end{align}
\end{proposition}

\begin{proof}
Let $\Pi_{i}$ be the identity map from $(M_i, q_i, d_i)$ to $(M_i, q_i, \bar{d}_i)$. If we restrict our attention on $B_{d_i}(q_i, 10\epsilon)$, by (\ref{eqn:PE27_9}), we have
\begin{align*}
e^{\frac{-D\epsilon}{m-2}} d_{i}(x,y)
\leq \bar{d}_i \left( \Pi_{i}(x), \Pi_{i}(y) \right)=\bar{d}_i\left( x, y \right) \leq e^{\frac{D\epsilon}{m-2}} d_{i}(x,y), \quad \forall \; x,y \in B_{d_i}(q_i, 10 \epsilon).
\end{align*}
Therefore, $\Pi_{i}$ are uniformly bi-Lipschitz map from $B_{d_i}(q_i, 10 \epsilon)$ to their images. So Arzela-Ascoli lemma implies that we have
a limit map $\Pi_{\infty}$ from $B_{d_{\infty}}(q_{\infty}, 10 \epsilon)$ to a subset in $B_{\bar{d}_{\infty}}(\bar{q}_{\infty}, 20\epsilon)$. Namely, we have obtained (\ref{eqn:PE06_5aa}).
It is also clear from the construction that (\ref{eqn:PE06_5a}) holds.
Fix two points $z,w \in B_{d_{\infty}}(q_{\infty}, 10\epsilon)$, we can find points $z_i \in B_{d_i}(q_i, 10\epsilon)$ and $w_i \in B_{d_i}(q_i,10\epsilon)$ such that
$z_i \to z$ and $w_i \to w$. It follows from (\ref{eqn:PE27_9}) that
\begin{align*}
e^{\frac{-10D\epsilon}{m-2}} d_{i}(z,w)
\leq \bar{d}_i \left( \Pi_{i}(z_i), \Pi_{i}(w_i) \right)=\bar{d}_i\left( z_i, w_i \right) \leq e^{\frac{10D\epsilon}{m-2}} d_{i}(z_i, w_i),
\end{align*}
whence we obtain (\ref{eqn:PE28_7}) by taking limit of the above inequalities.
Fix $z \in B_{d_{\infty}}(q_{\infty},\epsilon)$ and choose $z_i \in B_{d_i}(q_i, \epsilon)$, it follows from (\ref{eqn:PE27_10}) that
\begin{align*}
d_{GH} \left( B_{d_i}(z_i,r), B_{e^{\frac{f_i(z_i)-f_i(q_i)}{m-2}}\bar{d}_i}(\Pi_i(z_i), r) \right) < 2D r^2,
\end{align*}
whose limit reads as (\ref{eqn:PE28_8}). The proof of Proposition~\ref{prn:PE06_1} is complete.
\end{proof}

By interchanging the role of $g_i$ and $\bar{g}_i$, we have the following proposition whose proof is similar to that of Proposition~\ref{prn:PE29_2} and left to the reader.

\begin{proposition}\label{prn:PE29_2}
There exists a bi-Lipschitz homeomorphism map $\pi_{\infty}$ satisfying
\begin{align}
\pi_{\infty}: B_{\bar{d}_{\infty}}(\bar{q}_{\infty}, 10\epsilon) &\mapsto \pi_{\infty} \left( B_{\bar{d}_{\infty}}(\bar{q}_{\infty}, 10\epsilon)\right)
\subset B_{d_{\infty}}(q_{\infty}, 20\epsilon)
\label{eqn:PE29_6aa}\\
\bar{q}_{\infty}&\mapsto q_{\infty}=\pi_{\infty}(\bar{q}_{\infty}). \label{eqn:PE29_6a}
\end{align}
For each pair of points $\bar{z},\bar{w} \in B_{\bar{d}_{\infty}}(\bar{q}_{\infty}, \epsilon)$, denoting $\pi_{\infty}(\bar{z})$ and $\pi_{\infty}(\bar{w})$ by $z$ and $w$ respectively, we have
\begin{align}\label{eqn:PE29_7}
e^{-\frac{20D\epsilon}{m-2}} d_{\infty}(z, w) \leq \bar{\rho} \coloneqq \bar{d}_{\infty}(\bar{z},\bar{w})
\leq e^{\frac{20D\epsilon}{m-2}} d_{\infty}(z, w).
\end{align}
\end{proposition}

A direct consequence of the estimates (\ref{eqn:PE28_8}) and (\ref{eqn:PE29_7}) is the preservation of tangent space by $\Pi_{\infty}$ and $\pi_{\infty}$.
Let us show this property for $\Pi_{\infty}$.
Let $\{r_k\}$ be a sequence of positive numbers such that $r_k \to 0$. Define $\rho_k \coloneqq e^{\frac{-f_{\infty}(z)+f_{\infty}(q_{\infty})}{m-2}} r_k$.
Then we define
\begin{align}
&\left(\hat{Y}, \hat{z}, \hat{d}_{\infty} \right) \coloneqq \lim_{k \to \infty} \left(M_{\infty}, z, r_k^{-1} d_{\infty} \right), \label{eqn:PE31_5}\\
&\left(\hat{\bar{Y}}, \hat{\bar{z}}, \hat{\bar{d}}_{\infty} \right) \coloneqq \lim_{k \to \infty} \left(\bar{M}_{\infty}, \bar{z}, \rho_k^{-1} \bar{d}_{\infty} \right), \label{eqn:PE31_6}
\end{align}
if the limits exist. Fixing $L>1$ and replacing $r$ by $Lr_k$ in (\ref{eqn:PE28_8}), we have
\begin{align*}
&\quad d_{GH} \left( B_{r_k^{-1}d_{\infty}}(z,L), B_{\rho_k^{-1}\bar{d}_{\infty}}(\Pi_{\infty}(z), L) \right)\\
&= r_k^{-1} d_{GH} \left( B_{d_{\infty}}(z,Lr_k), B_{e^{\frac{f_{\infty}(z)-f_{\infty}(q_{\infty})}{m-2}} \bar{d}_{\infty}}(\Pi_{\infty}(z), Lr_k) \right) < 2D L^2 r_k \to 0.
\end{align*}
Letting $L \to \infty$, the above equation implies that
\begin{align}
d_{PGH} \left\{ \left(\hat{Y}, \hat{z}, \hat{d}_{\infty} \right), \left(\hat{\bar{Y}}, \hat{\bar{z}}, \hat{\bar{d}}_{\infty} \right)\right\}=0, \label{eqn:PE31_7}
\end{align}
which is the desired isometry. Via this isometry, it is clear that $\Pi_{\infty}$ preserves the regular-singular decomposition.
In other words, if $x \in \mathcal{R}$, then $\bar{x}=\Pi_{\infty}(x) \in \bar{\mathcal{R}}$; if $x \in \mathcal{S}$, then $\bar{x}=\Pi_{\infty}(x) \in \bar{\mathcal{S}}$.
Same argument obviously applies for the study of $\pi_{\infty}$.

We now summarize the previous discussion in this section into the following theorem.

\begin{theorem}
Suppose $(M_i,p_i,g_i,f_i) \in \mathcal{M}_{m}$ and $q_i \in B(p_i, D-1)$ for some $D>10m$.
Suppose $\epsilon=\frac{1}{100D}$. Then by taking subsequence if necessary, we have
\begin{align}
&(M_i, p_i, d_i) \longright{pointed-Gromov-Hausdorff} \left( M_{\infty}, p_{\infty}, d_{\infty} \right), \label{eqn:PE30_7} \\
&\left( B_{\bar{d}_i}(q_i, 20\epsilon), q_i, \bar{d}_i \right) \longright{pointed-Gromov-Hausdorff}
\left( B_{\bar{d}_{\infty}}(\bar{q}_{\infty}, 20\epsilon), \bar{q}_{\infty}, \bar{d}_{\infty} \right). \label{eqn:PE30_8}
\end{align}
Let $M_{\infty}=\mathcal{R} \cup \mathcal{S}$ and $B_{\bar{d}_{\infty}}(\bar{q}_{\infty}, 20\epsilon)=\bar{\mathcal{R}} \cup \bar{\mathcal{S}}$ be the
regular-singular decompositions of the limit spaces. Let $q_{\infty}$ be the limit point of $q_i$ along the convergence (\ref{eqn:PE30_7}).
Let $\Pi_{\infty}$ and $\pi_{\infty}$ be the natural bi-Lipschitz homeomorphism maps defined in Proposition~\ref{prn:PE29_1} and Proposition~\ref{prn:PE29_2}.
Then the following properties hold.
\begin{itemize}
\item[(a).] $\pi_{\infty} \circ \Pi_{\infty}=Id$ is well-defined on $B(q_{\infty}, \epsilon)$, and $\Pi_{\infty} \circ \pi_{\infty}=Id$ is
well-defined on $B_{\bar{d}_{\infty}}(\bar{q}_{\infty}, \epsilon)$.
\item[(b).] $\pi_{\infty}$ and $\Pi_{\infty}$ preserve tangent space in the sense of (\ref{eqn:PE31_5}), (\ref{eqn:PE31_6}) and (\ref{eqn:PE31_7}).
\item[(c).] $\pi_{\infty}$ and $\Pi_{\infty}$ preserve the regular-singular decomposition.
\end{itemize}
\label{thm:PE30_1}
\end{theorem}

We conclude this section with the following proposition.

\begin{proposition}\label{prn:PE29_3}
Suppose $(M_i^{m},p_i,g_i,f_i)$ is a sequence of Ricci shrinkers satisfying
\begin{align*}
(M_i, p_i, d_i) \longright{pointed-Gromov-Hausdorff} (M_{\infty}, p_{\infty}, d_{\infty}).
\end{align*}
If there is a point $y \in M_{\infty}$ such that some tangent space of $M_{\infty}$ at the point $y$ is isometric to $\R^{m}$, then there is a constant $A$ such that
\begin{align}\label{eqn:PE29_1}
\boldsymbol{\mu}(M_i, g_i) \geq -A.
\end{align}
\end{proposition}

\begin{proof}
Define
\begin{align}
D \coloneqq d_{\infty}(y, p_{\infty})+20m, \quad \epsilon=\frac{1}{100D}. \label{eqn:PE30_1}
\end{align}
Then can choose $y_i \in M_i$ such that $y_i \to y$. According to the choice of $D$ and $y_i$, it is clear that $B(y_i,1) \subset B(p_i,D)$.
Let $\bar{g}_i=e^{\frac{2(f(y_i)-f)}{m-2}}g_i$ and let $\bar{d}_{i}$ be the length structure induced from $\bar{g}_i$. Then it follows from Proposition~\ref{prn:PE06_1} and
Proposition~\ref{prn:PE29_1} that
\begin{align}
\left( B_{\bar{d}_i}(q_i, 2\epsilon), y_i, d_i \right) \longright{pointed-Gromov-Hausdorff}
\left( B_{\bar{d}_{\infty}}(\bar{y}, 2\epsilon), \bar{y}, \bar{d}_{\infty} \right).
\label{eqn:PE29_3}
\end{align}
By Proposition~\ref{prn:PE29_1}, the tangent space of $B_{\bar{d}_{\infty}}(\bar{y}, 2\epsilon)$ at $\bar{y}$ is isometric to the tangent space
of $M_{\infty}$ at $y$ and consequently is isometric to $\R^{m}$.
By the defintion of tangent space, there exists $r_k \to 0$ such that
\begin{align*}
\left( B_{\bar{d}_{\infty}}(\bar{y}, 2\epsilon), \bar{y}, r_k^{-1} \bar{d}_{\infty} \right) \longright{pointed-Gromov-Hausdorff} \left( \R^{m}, 0, d_{E} \right).
\end{align*}
Fix $\xi$ small, we can find large $k$ such that $2\xi^{-1}r_{k}<\epsilon$ and
\begin{align*}
d_{GH}\left\{ \left( B_{r_k^{-1}\bar{d}_{\infty}}(y, 2\xi^{-1}), r_k^{-1} \bar{d}_{\infty} \right), \left( B_{d_{E}}(0, 2\xi^{-1}), d_{E} \right) \right\}<0.5\xi.
\end{align*}
Then we fix this $k$. As $\bar{d}_{\infty}$ is the limit of $\bar{d}_i$, we can choose $i$ large enough such that
\begin{align}
d_{GH}\left\{ \left( B_{r_k^{-1}\bar{d}_i}(y_{i}, 2\xi^{-1}), r_k^{-1} \bar{d}_i \right), \left( B_{d_{E}}(0, 2\xi^{-1}), d_{E} \right) \right\}<0.5\xi.
\label{eqn:PE29_4}
\end{align}
As $2\xi^{-1} r_k<\epsilon$, it follows from part (a) of Theorem~\ref{thm:PE27_1} that $B_{r_k^{-1}\bar{d}_i}(y_{i}, 2\xi^{-1})=B_{\bar{d}_i}(y_i, 2\xi^{-1}r_k)$
is a relative compact subset of $M_i$.
Furthermore, by part (d) of Theorem~\ref{thm:PE27_1} , we know $|\bar{Rc}|_{\bar{g}_i}<\epsilon^{-2}$ on this set.
After rescaling, this means that
\begin{align*}
|Rc'|_{g_i'} < \epsilon^{-2} r_k^2, \quad \; \textrm{on} \; B_{r_k^{-1}\bar{d}_i}(y_i, 2\xi^{-1}),
\end{align*}
where $g_i'=r_k^{-2}\bar{g}_i$. It follows from the volume convergence theorem of Colding that (\ref{eqn:PE29_4}) implies that
\begin{align*}
\frac{1}{2} \omega_{m} \leq \left| B_{r_k^{-1}\bar{d}_i}(y_{i}, 1)\right|_{dv_{g_i'}}=r_k^{-m} |B_{\bar{d}_i}(y_i,r_k)|_{dv_{\bar{g}_i}}
\end{align*}
for large $i$. It follows that
\begin{align*}
|B_{\bar{d}_i}(y_i,r_k)|_{dv_{\bar{g}_i}} \geq \frac{1}{2} \omega_{m} r_k^{m}.
\end{align*}
Note that $r_k$ is a fixed small number. So we have
\begin{align*}
B_{\bar{d}_i}(y_i,r_k) \subset B_{d_i}(y_i, 2r_k) \subset B_{d_i}(y_i, 1) \subset B_{d_i}(p_i, D).
\end{align*}
Combining the previous two steps and using Theorem~\ref{thm:PE27_1} again, we obtain
\begin{align*}
|B(y_i,1)|_{dv_{g_i}} \geq \left|B_{\bar{d}_i}(y_i,r_k) \right|_{dv_{g_i}} \geq \frac{1}{2} |B_{\bar{d}_i}(y_i,r_k)|_{dv_{\bar{g}_i}}
\geq \frac{1}{4} \omega_m r_k^{m}.
\end{align*}
Plugging the above inequality into (\ref{eqn:PE04_61}), for all large $i$, we have
\begin{align*}
e^{\boldsymbol{\mu}(M_i,g_i)} \geq (4\pi)^{-\frac{m}{2}} e^{-\frac{D^2}{2}} |B(y_i,1)| > (4\pi)^{-\frac{m}{2}-1} \omega_m e^{-\frac{D^2}{2}} r_k^{m},
\end{align*}
whence we obtain (\ref{eqn:PE29_1}).
\end{proof}

Proposition~\ref{prn:PE29_3} basically means that the non-collapsing condition (\ref{eqn:PE29_1}) is a natural condition.
We shall pursue the structure of the limit space further under this condition.


\section{Regularity of the limit spaces}
\label{sec:reglimit}

In this section, we shall show that $\mathcal{R}(M_{\infty})$ is a smooth manifold under the non-collapsing condition.
Although these should be well-known to experts, we still find new ingredients of this regularity improvement, which are not available
in the literature.

\begin{lemma}
Suppose $(M_i^{m},p_i,g_i,f_i) \in \mathcal{M}(A)$ satisfying
\begin{align*}
(M_i, p_i, d_i) \longright{pointed-Gromov-Hausdorff} (M_{\infty}, p_{\infty}, d_{\infty}).
\end{align*}
Then for each $y \in M_{\infty}$, each tangent space $Y$ of $M_{\infty}$ at $y$ is an irreducible metric cone. Namely, $\mathcal{R}(Y)$ is connected.
\label{lma:PE27_0}
\end{lemma}

\begin{proof}
Note that $\boldsymbol{\mu}(M_i) \geq -A$ implies that we are in the non-collapsing case.
By Proposition~\ref{prn:PE06_1}, the tangent space of $M_{\infty}$ at $y$ is isometric to the tangent space of $B_{\bar{d}_{\infty}}(\bar{y}, \epsilon)$ at $\bar{y}=\Pi_{\infty}(y)$.
However, the local Ricci curvature bound (\ref{eqn:PE27_11}) in Theorem~\ref{thm:PE27_1} assures us to apply the classical Cheeger-Colding theory, which asserts that
the tangent space at $\bar{y}$ is an irreducible metric cone, whence we know the tangent space of $M_{\infty}$ at $y$ is also an irreducible metric cone.
\end{proof}

\begin{lemma}
Same condition as in Lemma~\ref{lma:PE27_0}.
Then for each $q \in M_{\infty}$, there is an $r=r_{q}>0$ such that the following properties hold:
\begin{itemize}
\item[(a).] $\mathcal{R} \cap B(q, r) \neq \emptyset$.
\item[(b).] Every two points $z,w \in B(q, r) \cap \mathcal{R}$ can be connected by a curve $\gamma \subset B(y, 3r) \cap \mathcal{R}$.
\end{itemize}
\label{lma:PE27_1}
\end{lemma}

\begin{proof}
We choose $D$ and $\epsilon$ the same as that in (\ref{eqn:PE30_1}). Then we define
\begin{align}
r \coloneqq \frac{1}{10}\epsilon. \label{eqn:PE30_2}
\end{align}
By Proposition~\ref{prn:PE29_2}, we know that
\begin{align*}
\mathcal{R} \cap B(q, r) \supset \bar{\mathcal{R}} \cap \pi_{\infty} \left\{ B_{\bar{d}_{\infty}}(\bar{q}, 0.1 r) \right\},
\end{align*}
where $\bar{q}=\Pi_{\infty}(q)$. As we are in the non-collapsing case, it follows from the classical Cheeger-Colding theory that
$\bar{\mathcal{R}}$ is a dense subset of $B_{\bar{d}_{\infty}}(\bar{q}, r)$. In particular, we can find a point $\bar{x} \in \bar{\mathcal{R}} \cap B_{\bar{d}_{\infty}}(\bar{q}, 0.1 r)$.
Let $x=\pi_{\infty}(\bar{x})$. Then we have
\begin{align*}
x \in \mathcal{R} \cap \pi_{\infty} \left\{ B_{\bar{d}_{\infty}}(\bar{q}, 0.1 r)\right\} \subset \mathcal{R} \cap B(q,r),
\end{align*}
which finishes the proof of part (a) of the Lemma.

We proceed to prove part (b). Let $\bar{z}=\Pi_{\infty}(z), \bar{w}=\Pi_{\infty}(w)$. In light of Proposition~\ref{prn:PE29_1}, we have
\begin{align*}
\bar{z}, \bar{w} \in \bar{\mathcal{R}} \cap B_{\bar{d}_{\infty}}(\bar{q}_{\infty}, 1.5r).
\end{align*}
Note that we are in the non-collapsing case now. It follows from Cheeger-Colding theory that $B_{\bar{d}_{\infty}}(\bar{q}_{\infty}, 1.5r)$ is path-connected (cf. section 3 of~\cite{CC00}).
Therefore, we can find a curve
\begin{align}
\bar{\gamma} \subset \bar{\mathcal{R}} \cap B_{\bar{d}_{\infty}}(\bar{q}_{\infty}, 1.5r).
\end{align}
Then we let $\gamma \coloneqq \pi_{\infty}(\bar{\gamma})$. Since $\pi_{\infty}$ preserves the regular-singular decomposition by Proposition~\ref{prn:PE29_2}, it
is clear that
\begin{align*}
\gamma \subset \mathcal{R} \cap \pi_{\infty} \left\{ B_{\bar{d}_{\infty}}(\bar{q}_{\infty}, 1.5r) \right\} \subset \mathcal{R} \cap B(q, 2r).
\end{align*}
Clearly, $\gamma$ connects $z=\pi_{\infty}(\bar{z})$ and $w=\pi_{\infty}(\bar{w})$. The proof of the part (b) of the Lemma is complete.
\end{proof}

Suppose $M_k \ni x_k \to x \in \mathcal{R}$ and $M_k \ni y \to y \in \mathcal{R}$, it is unclear whether the distance between $x,y$ is achieved by a smooth geodesic in $\mathcal{R}$.
Actually, let $\gamma_k$ be the shortest geodesic connecting $x_k$ and $y_k$. Passing to limit, we have a limit length-minimizing curve $\gamma_{\infty} \subset M_{\infty}$ with
two end points $x$ and $y$. It is possible that interior part of $\gamma$ may contain singular points, i.e., points in $\mathcal{S}$.
For example, see Figure~\ref{fig:globalbroken} for intuition. However, the bad behavior as in Figure~\ref{fig:globalbroken} is excluded by Lemma~\ref{lma:PE27_0}.
Further improvement of the length structure will be discussed in the next section (cf. Proposition~\ref{prn:CB08_1}). For dimension reason, it is also clear that each regular point $y \in \mathcal{R}(M_{\infty})$ has $\R^{m}$ as tangent space.
This property will be further improved in the following discussion.

\begin{figure}[h]
\begin{center}
\psfrag{S}[c][c]{$\mathcal{S}$}
\psfrag{x}[c][c]{$x$}
\psfrag{y}[c][c]{$y$}
\psfrag{g}[c][c]{$\textcolor{red}{\gamma}$}
\psfrag{M}[c][c]{$M_{\infty}$}
\includegraphics[width=0.5\columnwidth]{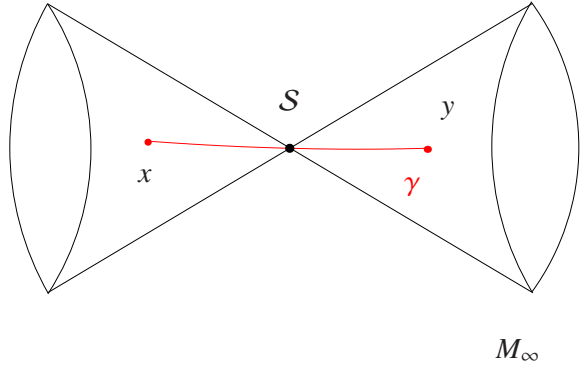}
\end{center}
\caption{The global distance may not be achieved by smooth geodesic}
\label{fig:globalbroken}
\end{figure}

\begin{definition}
Suppose $(M_{\infty}, d_{\infty})$ is a length space of Hausdorff dimension $m$ and has the regular-singular decomposition $M_{\infty}=\mathcal{R} \cup \mathcal{S}$.
Suppose further that with the topology induced by $d_{\infty}$, $\mathcal{R}$ is an open set.
We say that $\mathcal{R}$ has a Riemannian manifold structure if the following properties hold.
\begin{itemize}
\item There is a homeomorphism map $\iota$ from $\mathcal{R}$ to a smooth manifold $N^{m}$.
\item There is a smooth Riemannian metric $g$ on $N^{m}$ such that $\iota$ is a \textbf{local} isometry from $\left( \mathcal{R}, d_{\infty} \right)$
to $(N, d_{g})$. Namely, for each point $y \in \mathcal{R}$, there exists $r=r_y$ such that $\iota$ is an isometry map from $B(y,r) \subset \mathcal{R}$ to
$B(\iota(y), r) \subset N^{m}$.
\end{itemize}
For simplicity of notations, we may identify $\mathcal{R}$ and $N$ and omit $\iota$. Then $(\mathcal{R}, d_{\infty})$ is locally isometric to $(\mathcal{R}, g)$.
\label{dfn:PC24_3}
\end{definition}

\begin{proposition}
Suppose $(M_{\infty}, p_{\infty}, d_{\infty}, f_{\infty})$ is the limit in (\ref{eqn:PC25_5}).
Suppose $M_{\infty}=\mathcal{R} \cup \mathcal{S}$ is the regular-singular decomposition of $M_{\infty}$.
Then $\mathcal{R}$ is an open set with respect to the topology induced by $d_{\infty}$. Furthermore, $\mathcal{R}$ has a smooth Riemannian manifold structure $(\mathcal{R}, g_{\infty})$
which is locally isometric to $\left( \mathcal{R}, d_{\infty} \right)$. The restriction of $f_{\infty}$ on $\mathcal{R}$ is a smooth function satisfying
\begin{align}
Rc_{g_{\infty}} + \emph{Hess} f_{\infty} -\frac{g_{\infty}}{2}=0. \label{eqn:PC25_8}
\end{align}

\label{prn:PC25_2}
\end{proposition}

\begin{proof}
We shall divide the proof into several steps.

\textit{Step 1. For each $x \in \mathcal{R}$, there is a radius $r=r_x>0$ such that $\left( B_{d_{\infty}}(x,r), d_{\infty}\right)$ is isometric to a geodesically convex smooth Riemannian manifold.}

Recall that a point $x \in \mathcal{R}$ if and only if a tangent space of $M_{\infty}$ at $x$ is $\R^{m}$. Namely, there exists a sequence $r_k \to 0$ such that
\begin{align}
\left( M_{\infty}, x, r_k^{-1} d_{\infty} \right) \longright{pointed-Gromov-Hausdorff} \left(\R^{m}, 0, d_{E} \right).
\label{eqn:PC22_1}
\end{align}
Therefore, for each small $\epsilon$, we can find a fixed $r$ such that
\begin{align*}
d_{GH} \left\{ B_{r^{-1}d_{\infty}}(x, 100), B_{d_E}(0, 100)\right\}<\epsilon,
\end{align*}
which is equivalent to
\begin{align*}
r^{-1} d_{GH} \left\{ B_{d_{\infty}}(x, 100r), B_{d_E}(0, 100r)\right\}< \epsilon
\end{align*}
by the scaling invariance of $d_{E}$.
Now we fix $r$. Since $x \in M_{\infty}$, we can assume $x$ is the limit point of $x_i \in M_i$.
For large $i$, the above inequality and convergence implies that
\begin{align*}
r^{-1} d_{GH} \left\{ B_{d_i}(x_i, 100r), B_{d_E}(0, 100r)\right\}< \epsilon.
\end{align*}
Consequently, we can apply Theorem~\ref{thm:PD14_1} to obtain that $\mathbf{r}(x_i)> r'$ for some $r'$ depending on $x$. Choosing $r_x=0.1 r'$. Then it follows from the definition of convex radius and the Ricci shrinker PDE that
\begin{align*}
\left( (Exp_{x_i})^*g_i, (Exp_{x_i})^*f_i\right)
\end{align*}
are smooth couples defined on $B(0, 2r_x) \subset \R^m$ with uniform(independent of $i$) derivatives on each order.
Identifying $T_{x_i}M$ with $\R^m$, we can regard $Exp_{x_i}$ as a smooth map from $\R^m$ to $M_i$.
Rigorously speaking, an orthonormal frame at $T_{x_i}M$ is needed to achieve such identification.
However, since every two orthonormal frames differ by an element in $O(m)$ which is compact, the effect caused by difference of orthonormal frames will vanish by a subsequence argument.
Therefore, we omit this step for the simplicity of the argument.
Let $h_i=(Exp_{x_i})^*g_i$. Then we can take limit in the smooth topology to obtain $h_{\infty}$ on $B(0, 2r_x)$.
Moreover, the uniform derivative estimates of $h_i$ are inherited by $h_{\infty}$. Then we can run the Riemannian geometry argument in polar coordinate, as in section 3.4 of~\cite{DoCarmo}, to obtain that
$(B(0, 1.5 r_x), h_{\infty})$ is strongly convex. Via the estimates to define convex radius, it is clear that the exponential map
\begin{align*}
Exp_{x_i}: \; B(0, r_{x}) \subset \R^m \mapsto B(x_i, r_{x}) \subset M_i
\end{align*}
will converge to a limit map
$$Exp_{x}: B(0,r_x) \subset \R^m \mapsto B(x, r_{x}) \subset M_{\infty}$$
and $Exp_{x}$ is an isometry from $\left( B(0,r_x), h_{\infty}\right)$ to $\Omega_x=B_{d_{\infty}}(x, r_{x}) \subset M_{\infty}$.
In other words, $\left(\Omega_x, d_{\infty} \right)$ is isometric to a smooth Riemannian manifold $\left( B(0,r_x), h_{\infty}\right)$, via the map $Exp_{x}$.
For simplicity of notation, we write directly that $(\Omega_x, g_{\infty})$ is a smooth, open, strongly convex, Riemannian manifold.

\textit{Step 2. $\mathcal{R}$ is a nonempty, connected, open subset of $M_{\infty}$.}

It follows from the first property in Lemma~\ref{lma:PE27_1} directly that $\mathcal{R}$ is non-empty. Up to a standard covering argument, one can obtain the connectedness of
$\mathcal{R}$ by the second property in Lemma~\ref{lma:PE27_1}.
For each $x \in \mathcal{R}$, by step 1, we know $\left( B_{d_{\infty}}(x,r_x), d_{\infty} \right)$ is isometric to a smooth manifold for some $r_{x}>0$.
It is clear that any point nearby $x$ still has $\R^{m}$ as tangent space. Therefore, $\mathcal{R}$ is an open subset of $M_{\infty}$.
This finishes the proof of Step 2.

\textit{Step 3. There is a smooth Riemannian metric $g_{\infty}$ on $\mathcal{R}$ such that $(\mathcal{R}, g_{\infty})$ is a smooth Riemannian manifold.}

This step is more or less standard after the convex radius estimate (cf. Definition~\ref{dfn:PD15_1}). Similar discussions can be found in, for example, Hamilton's work~\cite{Ha95}.
The difference here is that $\mathcal{R}$ may not be complete. However, it will be clear through the proof that the completeness of $\mathcal{R}$ is not needed.
We provide details and necessary references here for the sake of self-containedness.

In step 1, we have shown that locally every regular point has a neighborhood which is isometric to a geodesically convex Riemannian manifold.
We now need to patch these neighborhoods together to obtain a global Riemannian manifold structure on $\mathcal{R}$.
Since $\mathcal{R}$ is open, it is clear that $\mathcal{R}=\bigcup_{x \in \mathcal{R}} \Omega_{x}$. Then ball packing argument implies that
this covering has a countable sub-covering. Namely, we can choose countably many points $x_{\beta} \in \mathcal{R}$ such that
\begin{align*}
\mathcal{R} = \bigcup_{\beta=1}^{\infty} \Omega_{x_{\beta}}, \quad x_{\beta} \in \mathcal{R}.
\end{align*}
We need to equip $\mathcal{R}$ with a smooth manifold structure. Note that $\mathcal{R}$ is covered by countably many open sets $\Omega_{x_{\beta}}$.
Each of them is diffeomorphic to a standard ball in $\R^{m}$, via exponential map. In order to equip $\mathcal{R}$ with a smooth manifold structure, it suffices to define
a smooth transition map $\varphi_{\alpha \beta}$ satisfying
\begin{align}
Exp_{x_{\beta}}^{-1}(\Omega_{x_{\alpha}} \cap \Omega_{x_{\beta}}) \overset{\varphi_{\alpha\beta}}{\to} Exp_{x_{\alpha}}^{-1}(\Omega_{x_{\alpha}} \cap \Omega_{x_{\beta}})
\label{eqn:PC23_2}
\end{align}
whenever $\Omega_{x_{\alpha}} \cap \Omega_{x_{\beta}} \neq \emptyset$.
Fix $y \in \Omega_{x_{\alpha}} \cap \Omega_{x_{\beta}}$. As $\Omega_{x_{\alpha}} \cap \Omega_{x_{\beta}}$ is open, we can find a small $\delta$ such that
\begin{align*}
B(y,\delta) \subset \Omega_{x_{\alpha}} \cap \Omega_{x_{\beta}}.
\end{align*}
For each $i$, we can find $y_i, x_{\alpha,i} \in M_i$ such that
\begin{align*}
y_i \to y, \quad x_{\alpha,i} \to x_{\alpha}, \quad \textrm{as} \; i \to \infty.
\end{align*}
It is clear from the previous two equations that
$$B(y_i, \delta) \subset B(x_{\alpha,i}, r_{x_\alpha}) \cap B(x_{\beta,i}, r_{x_{\beta}}).$$
Therefore we obtain a map
\begin{align}
\varphi_{\alpha \beta}^{(i)}: Exp_{x_{\beta,i}}^{-1} \left( B(y_i, \delta) \right) \to Exp_{x_{\alpha,i}}^{-1} \left( B(y_i, \delta) \right).
\label{eqn:PC23_1}
\end{align}
Recall that
\begin{align*}
Exp_{x_{\beta,i}}^{-1} \left( B(y_i, \delta) \right) \subset B(0,1) \subset \R^{m}
\end{align*}
for each $\beta$. It is clear that both the domain and the image of $\varphi_{\alpha \beta}^{(i)}$ are bounded open sets in $B(0,1)$.
Moreover, the curvature and all higher order curvature derivatives bounds imply that both $\varphi_{\alpha \beta}^{(i)}$ and $\varphi_{\beta \alpha }^{(i)}=\left\{ \varphi_{\alpha \beta}^{(i)} \right\}^{-1}$
have uniform derivative estimates for each order. In particular, $D \varphi_{\alpha \beta}^{(i)}$ and $D \varphi_{\beta \alpha }^{(i)}$ are almost identity matrices.
By Arzela-Ascoli lemma, we have
\begin{align*}
&Exp_{x_{\beta,i}}^{-1} \left( B(y_i, \delta) \right) \to Exp_{x_{\beta}}^{-1} \left( B(y, \delta) \right), \\
&Exp_{x_{\alpha,i}}^{-1} \left( B(y_i, \delta) \right) \to Exp_{x_{\alpha}}^{-1} \left( B(y, \delta) \right),\\
&\varphi_{\alpha \beta}^{(i)} \longright{C^{\infty}} \varphi_{\alpha \beta}, \quad \textrm{as} \; i \to \infty.
\end{align*}
Note that currently the map $\varphi_{\alpha \beta}$ is a smooth diffeomorphism from $Exp_{x_{\beta}}^{-1} \left( B(y, \delta) \right)$ to $Exp_{x_{\alpha}}^{-1} \left( B(y, \delta) \right)$.
Up to a finite covering, the above set $B(y, \delta)$ can be replaced by a compact set $K \subset \Omega_{x_{\alpha}} \cap \Omega_{x_{\beta}}$.
Then by the arbitrary choice of $K$, we can extend the definition of $\varphi_{\alpha \beta}$ to be a smooth diffeomorphism from
$Exp_{x_{\beta}}^{-1} \left( \Omega_{x_{\alpha}} \cap \Omega_{x_{\beta}} \right)$ to $Exp_{x_{\alpha}}^{-1} \left( \Omega_{x_{\alpha}} \cap \Omega_{x_{\beta}} \right)$.
Namely, we have finished the proof of (\ref{eqn:PC23_2}).
Therefore, we know $\mathcal{R}$ inherits a natural smooth structure from the limit process. We then proceed to show that there is a global Riemannian metric $g_{\infty}$ defined on
$\mathcal{R}$. Fix $y \in \Omega_{x_{\alpha}} \cap \Omega_{x_{\beta}}$ as in previous argument.
Note that by the naturality of Riemannian metric $g_i$, we have
\begin{align}
\left\{ \varphi_{\alpha \beta}^{(i)} \right\}^{*} \left( Exp_{x_{\alpha,i}}^{*} g_i \right)= Exp_{x_{\beta,i}}^{*} g_i.
\label{eqn:PC23_3}
\end{align}
Recall from (\ref{eqn:PC23_1}) that $\varphi_{\alpha \beta}^{(i)}$ is a smooth map from an open set in $B(0, 1)$ to an open set in $B(0,1)$, with uniform derivative estimates.
The choice of $r_{\alpha}$ guarantees that $Exp_{x_{\alpha,i}}^{*} g_i$ are positive-definite matrix-valued functions on $B(0,r_{\alpha})$, with uniform estimates of values and all derivatives.
In each coordinate $B(0,r_{\alpha})=Exp_{x_{\alpha}}^{-1}\left( B(x_{\alpha}, r_{\alpha}) \right)$, let $h_{\alpha}$ be the limit of $Exp_{x_{\alpha,i}}^{*} g_i$.
Then the limit of (\ref{eqn:PC23_3}) reads as
\begin{align*}
\left( \varphi_{\alpha \beta} \right)^{*} h_{\alpha} = h_{\beta},
\end{align*}
which allows us to define without ambiguity that
\begin{align*}
g_{\infty}(y)=\left( Exp_{x_{\alpha}} \right)_{*} h_{\alpha}.
\end{align*}
Therefore, we have obtained that $(\mathcal{R}, g_{\infty})$ is a smooth Riemannian manifold.

\textit{Step 4. $f_{\infty}|_{\mathcal{R}}$ is a smooth function satisfying Ricci shrinker equation (\ref{eqn:PC25_8}).}

Fix $y \in \mathcal{R}$. Choose $r=r_{y}$ such that $B(y,r)$ is geodesically convex and all curvature derivatives are uniformly bounded.
Let $y$ be the limit point of $y_i \in M_i$. Then for each $i$, $l_i=Exp_{y_i}^{*}(f_i)$ is a smooth function on the chart $B(0,r) \subset \R^{m}$.
Clearly, the value of $l_{i}$ and all its higher derivatives are all uniformly bounded independent of $i$. So we can take limit to obtain a
smooth function $l_{\infty}$ on $B(0, r)$ satisfying the equation
\begin{align*}
Rc(h) + \text{Hess}_{h} l_{\infty}-\frac{h}{2}=0,
\end{align*}
where $h=(Exp_{y})^* g_{\infty}$. It is clear from construction that
\begin{align*}
f_{\infty}= (Exp_{y})_{*} l_{\infty},
\end{align*}
which implies that $f_{\infty}$ is a smooth function on $B(y,r)$. By the arbitrary choice of $y \in \mathcal{R}$, we obtain that $f_{\infty}|_{\mathcal{R}}$ is smooth and
satisfies the Ricci shrinker equation (\ref{eqn:PC25_8}).
\end{proof}

\begin{remark}
In our proof of showing $\mathcal{R}$ being a smooth manifold in Proposition~\ref{prn:PC25_2}, we avoid the use of direct limit,
as done by Hamilton (cf. p. 557 of~\cite{Ha95}, or Section 4.2 of~\cite{CCGG}).
\label{rmk:PE30_1}
\end{remark}

\begin{theorem}
Same conditions as in Theorem~\ref{thm:PE30_1}. If we further assume the non-collapsing condition $\boldsymbol{\mu}(M_i,g_i) \geq -A$,
then we further have
\begin{itemize}
\item[(1).] $\mathcal{R}$ has a natural Riemannian manifold structure $(\mathcal{R}, g_{\infty})$.
$f_{\infty}|_{\mathcal{R}}$ is a smooth function satisfying the Ricci shrinker equation
\begin{align}
Rc+ \emph{Hess}_{g_{\infty}} f_{\infty} - \frac{g_{\infty}}{2}=0. \label{eqn:PE31_2}
\end{align}
\item[(2).] $\bar{\mathcal{R}}$ has a natural Riemannian manifold structure $(\bar{\mathcal{R}}, \bar{g}_{\infty})$.
$\bar{f}|_{\bar{\mathcal{R}}}$ is a smooth function satisfying the system
\begin{subequations}
\begin{align}[left = \empheqlbrace \,]
&(m-2)\overline{Rc}=d\bar{f}_{\infty} \otimes d\bar{f}_{\infty}
+ \left(m-1-\bar{f}_{\infty} -f_{\infty}(q_{\infty}) \right) e^{\frac{2\bar{f}_{\infty}}{m-2}} \bar{g}_{\infty},\label{eqn:PE31_1a}\\
&\bar{\Delta} f_{\infty}=e^{\frac{2\bar{f}_{\infty}}{m-2}} \left( \frac{m}{2} - \bar{f}_{\infty}-f_{\infty}(q_{\infty}) \right). \label{eqn:PE31_1b}
\end{align}
\label{eqn:PE31_1}
\end{subequations}
\item[(3).] The following maps are smooth diffeomorphisms:
\begin{align}
&\Pi_{\infty}|_{\mathcal{R} \cap B(q_{\infty},\epsilon)}: \quad \mathcal{R} \cap B(q_{\infty},\epsilon)
\to \bar{\mathcal{R}} \cap \Pi_{\infty} \left\{ B(q_{\infty}, \epsilon) \right\}, \label{eqn:PE31_3}\\
&\pi_{\infty}|_{\bar{\mathcal{R}} \cap B_{\bar{d}_{\infty}}(\bar{q}_{\infty}, \epsilon)}:
\quad \bar{\mathcal{R}} \cap B_{\bar{d}_{\infty}}(\bar{q}_{\infty}, \epsilon)
\to \bar{\mathcal{R}} \cap \pi_{\infty}\left\{ B_{\bar{d}_{\infty}}(\bar{q}_{\infty},\epsilon) \right\}.
\label{eqn:PE31_4}
\end{align}
\end{itemize}
\label{thm:PE31_1}
\end{theorem}

\begin{proof}
Part (1) of this theorem is exactly Proposition~\ref{prn:PC25_2}.
By the equivalence of harmonic radius between metric $g_i$ and $\bar{g}_i$, it is easy to see that the proof of part (2) can be finished exactly like that in part (1).
Therefore, we only need to show part (3). This follows from the fact that the identity map from $(M_i, g_i)$ to $(M_i, \bar{g}_i)$ has uniform regularity around
points where harmonic radii are bounded from below.
\end{proof}


\section{Smooth convergence on regular part}
\label{sec:CGconv}

We show that the convergence around $\mathcal{R}(M_{\infty})$ can be improved to stay in the smooth topology.
To put our discussion on a rigorous foundation, we first review some definitions.

\begin{definition}
Suppose $(M_{\infty}, d_{\infty})$ is a length space of Hausdorff dimension $m$ and has the regular-singular decomposition $M_{\infty}=\mathcal{R} \cup \mathcal{S}$.
Suppose further that $(\mathcal{R}, d_{\infty})$ is \textbf{locally} isometric to a Riemannian manifold $\left( \mathcal{R}, g_{\infty} \right)$.
By the convergence
\begin{align}
(M_i, p_i, g_i) \longright{pointed-\hat{C}^{\infty}-Cheeger-Gromov} \left(M_{\infty}, p_{\infty}, g_{\infty} \right), \label{eqn:PC25_2}
\end{align}
we mean the following properties hold.
\begin{itemize}
\item The following convergence holds:
\begin{align}
\left( M_i, p_i, d_i \right) \longright{pointed-Gromov-Hausdorff} \left( M_{\infty}, p_{\infty}, d_{\infty} \right).
\label{eqn:PF24_2}
\end{align}
\item For each positive constant $L>0$ and each compact set $K \subset \mathcal{R} \cap B(p_{\infty}, L)$, by taking subsequence and adjusting $\epsilon_{L,i}$ if necessary,
each $\epsilon_{L,i}$-isometry $\psi_{L,i}$ in (\ref{eqn:PD16_2}) can be chosen to satisfy
\begin{align}
\psi_{L,i}(x)=\varphi_{K,L,i}(x), \quad \forall \; x \in K, \label{eqn:PD29_1}
\end{align}
where $\varphi_{K,L,i}$ is a diffeomorphism from $K$ to its image in $M_i$.
\item The convergence on $K$ is smooth via $\varphi_{K,L,i}$ in the sense that
\begin{align}
\varphi_{K,L,i}^{*} g_i \longright{C^{\infty}} g_{\infty}, \quad \textrm{on}\; K.
\label{eqn:PC25_1}
\end{align}
\end{itemize}
By the convergence
\begin{align}
(M_i, p_i, g_i, f_i) \longright{pointed-\hat{C}^{\infty}-Cheeger-Gromov} \left(M_{\infty}, p_{\infty}, g_{\infty}, f_{\infty} \right), \label{eqn:PC25_3}
\end{align}
we mean that (\ref{eqn:PC25_2}) holds and
\begin{align}
\varphi_{K,L,i}^{*} f_i \longright{C^{\infty}} f_{\infty}. \label{eqn:PC25_4}
\end{align}
\label{dfn:PC25_1}
\end{definition}

One can refer Figure~\ref{fig:hatconvergence} for intuition.
It seems that $p_{\infty} \in \mathcal{R} \backslash K$ in Figure~\ref{fig:hatconvergence}.
However, we remind the reader that we do not put such restriction on $p_{\infty}$ in Definition~\ref{dfn:PC25_1}.
It is also possible that $p_{\infty} \in \mathcal{S}$, or $p_{\infty} \in K \subset \mathcal{R}$.

\begin{figure}[h]
\begin{center}
\psfrag{A}[c][c]{$\psi_{L,i}$}
\psfrag{B}[c][c]{$\textcolor{red}{\varphi_{K,L,i}}$}
\psfrag{C}[c][c]{$\psi_{L,i}|_{K}=\varphi_{K,L,i}$}
\psfrag{D}[c][c]{$\textcolor{red}{K}$}
\psfrag{E}[c][c]{$\textcolor{red}{\psi_{L,i}(K)}$}
\psfrag{B1}[c][c]{$B(p_{\infty}, L)$}
\psfrag{B2}[c][c]{$B(p_i, L+\epsilon_{L,i})$}
\psfrag{P1}[c][c]{$p_{\infty}$}
\psfrag{P2}[c][c]{$p_i$}
\psfrag{S}[c][c]{$\mathcal{S}$}
\includegraphics[width=0.5\columnwidth]{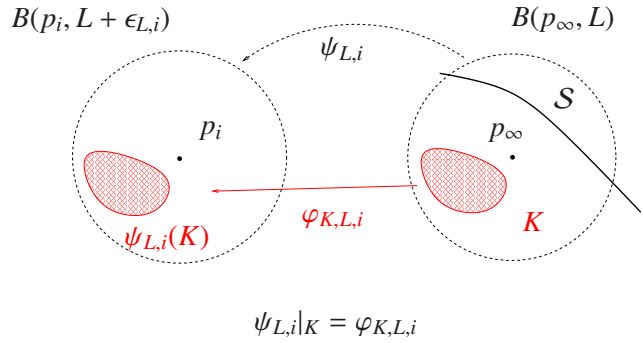}
\end{center}
\caption{The convergence in pointed-$\hat{C}^{\infty}$-Cheeger-Gromov topology}
\label{fig:hatconvergence}
\end{figure}

We say more words on the meaning of (\ref{eqn:PC25_2}). Fix each compact subset $K \subset \mathcal{R}$, we can choose $L_k \to \infty$ such that
$K \subset B(p_{\infty}, L_k) \cap \mathcal{R}$. Then we have $\epsilon_{L_{k},i}$-isometry $\psi_{L_{k}, i}$, whose restriction on $K$ is a diffeomorphism
$\varphi_{K,L_{k}, i}$ satisfying
\begin{align}
\varphi_{K,L_{k},i}^{*} g_i \longright{C^{\infty}} g_{\infty}, \quad \textrm{on}\; K. \label{eqn:PD16_4}
\end{align}
Let $k \to \infty$ and take diagonal sequence, we may denote $\epsilon_{L_{k}, i_k}$ by $\epsilon_k$,
denote $\psi_{L_{k},i_k}$ by $\psi_k$ and denote $\varphi_{K,L_{i_k}, i_k}$ by $\varphi_k$.
Then (\ref{eqn:PD29_1}) becomes
\begin{align}
\psi_k(x)=\varphi_k(x), \quad \forall \; x \in K \label{eqn:PD29_2}
\end{align}
where $\psi_k$ is an $\epsilon_k$-isometry from $B(p_{\infty}, L_k)$ to $B(p_k, L_k+\epsilon_k)$.
The equation (\ref{eqn:PD16_4}) is then simplified as
\begin{align}
\varphi_{k}^{*} g_k \longright{C^{\infty}} g_{\infty}, \quad \textrm{on}\; K. \label{eqn:PD16_5}
\end{align}
Suppose $x \in K$, $y \in M_{\infty}$ and $M_k \ni y_k \to y$. It is clear that
\begin{align*}
\lim_{k \to \infty} d_k \left( y_k, \psi_{k}(y) \right)=0.
\end{align*}
Consequently, as $x \in K$, (\ref{eqn:PD29_2}) implies that
\begin{align}
d_{\infty}(x,y)&=\lim_{k \to \infty} d_{k} \left( \psi_{k}(x), \psi_{k}(y) \right)=\lim_{k \to \infty} d_{k} \left( \varphi_{k}(x), \psi_{k}(y) \right) \notag\\
&=\lim_{k \to \infty} d_{k} \left( \varphi_{k}(x), y_{k} \right).  \label{eqn:PD16_6}
\end{align}
In particular, if we choose $y_k=p_k$ and $y_{\infty}=p_{\infty}$, then the above identity implies that
\begin{align}
d_{\infty}(x, p_{\infty})=\lim_{k \to \infty} d_k\left( \varphi_k(x), p_k \right).
\label{eqn:PD16_7}
\end{align}
It is also clear that if we choose both $x,y \in K$, then (\ref{eqn:PD16_6}) implies that
\begin{align}
d_{\infty}(x,y)=\lim_{k \to \infty} d_k\left( \varphi_k(x), \varphi_k(y) \right).
\label{eqn:PD16_8}
\end{align}
In summary, by (\ref{eqn:PC25_2}), up to taking subsequence if necessary, we can assume that
\begin{subequations}
\begin{align}[left=\empheqlbrace\,]
&M_i \ni \varphi_i(x) \to x, \quad \forall \; x \in K; \label{eqn:PD29_5a}\\
&d_{\infty}\left( x, y \right)=\lim_{i \to \infty} d_i(\varphi_i(x), y_i), \quad \forall\; x \in K, y \in M_{\infty},M_i \ni y_i \to y; \label{eqn:PD29_5b}\\
&\varphi_i^{*} g_i \longright{C^{\infty}} g_{\infty}, \quad \textrm{on each compact subset}\; K. \label{eqn:PD29_5c}
\end{align}
\label{eqn:PD29_5}
\end{subequations}
In particular, letting $y_i$ be $p_i$ or $\varphi_i(z)$ for some $z \in K$, we have
\begin{subequations}
\begin{align}[left=\empheqlbrace\,]
&d_{\infty}(x, p_{\infty}) =\lim_{i \to \infty} d_i(\varphi_i(x), p_i), \quad \forall \; x \in K; \label{eqn:PD27_5a}\\
&d_{\infty}\left( x, z \right)=\lim_{i \to \infty} d_i(\varphi_i(x), \varphi_i(z)), \quad \forall\; x, z \in K. \label{eqn:PD27_5b}
\end{align}
\label{eqn:PE27_5}
\end{subequations}

\begin{remark}
Roughly speaking, (\ref{eqn:PC25_2}) means that \textbf{``the convergence takes place in the $C^{\infty}$-topology''}. Such description can be found, for example,
in Theorem 7.3 of Cheeger-Colding~\cite{CC97}.
\label{rmk:PE01_1}
\end{remark}

\begin{remark}
Because of the existence of diffeomorphisms $\varphi_i$ and the smooth convergence (\ref{eqn:PD29_5c}), we can put more structures on $M_i$ and study the convergence of such structures on $\mathcal{R}$.
For example, if each $M_i$ is a K\"ahler manifold $(M_i, g_i, J_i)$ of complex dimension $n$, we can define the pointed-$\hat{C}^{\infty}$-Cheeger-Gromov convergence of K\"ahler manifolds.
By the convergence
\begin{align}
(M_i, p_i, g_i, J_i) \longright{pointed-\hat{C}^{\infty}-Cheeger-Gromov} \left(M_{\infty}, p_{\infty}, g_{\infty}, J_{\infty}\right), \label{eqn:PE27_6}
\end{align}
we mean (\ref{eqn:PC25_2}) holds and $\displaystyle \varphi_i^*(J_i) \longright{C^{\infty}} J_{\infty}$ on each compact set $K \subset \mathcal{R}$.
Such convergence is essentially used in ~\cite{CW17A} and~\cite{CW17B}.
\label{rmk:PE27_1}
\end{remark}

\begin{proposition}
The pointed-Gromov-Hausdorff convergence in (\ref{eqn:PC25_5}) can be improved to
\begin{align}
(M_i, p_i, g_i,f_i) \longright{pointed-\hat{C}^{\infty}-Cheeger-Gromov} \left(M_{\infty}, p_{\infty}, g_{\infty},f_{\infty} \right). \label{eqn:PC26_8}
\end{align}
\label{prn:PC25_4}
\end{proposition}

\begin{proof}
We separate the proof of (\ref{eqn:PC26_8}) into several steps.

\textit{Step 1. For each compact subset $K \subset \mathcal{R}$, we can define a smooth diffeomorphism $\varphi_{i}$ from $K$ to a subset in $M_i$.}
The construction of $\varphi_i$ is more or less standard, except that we need to handle it locally to avoid singularities.
We shall describe the key steps and refer to the work of A. Katsuda~\cite{Katsuda} and A. Kasue~\cite{Kasue} for more details.

Fix a compact set $K \subset \mathcal{R}$ and a large positive number $L$ such that $K \subset B(p_{\infty}, L) \backslash \mathcal{S}_{L^{-1}}$, where $\mathcal{S}_{c}$ is the $c$-neighborhood of $\mathcal{S}$.
It follows from (\ref{eqn:PC25_5}), Definition~\ref{dfn:PC24_1} and its equivalent version that there exist $\epsilon_{L,i}$-isometries $\psi_{L,i}$
satisfying (\ref{eqn:PD16_2}) and (\ref{eqn:PD16_3}). We need to perturb $\psi_{L,i}$ on $K$ slightly to obtain the desired $\varphi_i$.

Define
\begin{align*}
K' \coloneqq \left\{ x \left|d(x,K) \leq \frac{1}{10L} \right. \right\}, \quad K'' \coloneqq \left\{ x \left|d(x,K) \leq \frac{1}{5L} \right. \right\}.
\end{align*}
Then it is clear that
\begin{align*}
K \Subset K' \Subset K'' \Subset \mathcal{R}.
\end{align*}
For each $x \in K''$, set
\begin{align}
r_x \coloneqq \frac{1}{1000m} \cdot \min \left\{ \textrm{convex radius of} \; x, \; \frac{1}{L} \right\}. \label{eqn:PF07_1}
\end{align}
Clearly, we have $K'' \subset \cup_{x \in K} B(x,r_x)$.
The compactness of $K$ guarantees us to choose a finite cover of $K$ by
\begin{align*}
K'' \subset \bigcup_{\alpha=1}^{Q} B(x_{\alpha}, r_{x_{\alpha}})
\end{align*}
and the balls $B(x_{\alpha}, 0.25 r_{\alpha})$ are disjoint. This can be done through a standard covering argument by G. Vitali.
More details on the construction of the covering can be found, for example, in Lemma 2.6 of~\cite{CGY}.
For simplicity of notations, we denote $r_{x_{\alpha}}$ by $r_{\alpha}$. Then we have
\begin{align}
K'' \subset \bigcup_{\alpha=1}^{Q} B(x_{\alpha}, r_{x_{\alpha}}); \quad B(x_{\alpha}, 0.25 r_{\alpha}) \cap B(x_{\beta}, 0.25 r_{\beta})=\emptyset, \;
\textrm{if} \; \alpha \neq \beta.
\label{eqn:PD28_1}
\end{align}
Let $\eta$ be a non-increasing cutoff function such that $\eta \equiv 1$ on $(-\infty, 1)$, $\eta \equiv 0$ on $(4, \infty)$, and $-1 \leq \eta' \leq 0$.
Then we define cutoff function
\begin{align*}
\eta_{\alpha} \coloneqq \eta\left( \frac{d(\cdot, x_{\alpha})}{r_{\alpha}} \right).
\end{align*}
In light of the covering (\ref{eqn:PD28_1}), it is clear that $\sum_{\alpha} \eta_{\alpha}$ is a positive function on $K''$ such that
\begin{align}
& \frac{1}{C(m)} \leq \sum_{\alpha} \eta_{\alpha} \leq C(m). \label{eqn:PF07_4}\\
& |\nabla \eta_{\alpha}| \leq C(m). \label{eqn:PF07_5}
\end{align}
For each point $x_{\alpha}$, we attach an orthonormal frame $\left\{u_{\alpha}^{(1)}, u_{\alpha}^{(2)}, \cdots, u_{\alpha}^{(m)} \right\}$. Through this orthonormal frame, we can regard
$Exp_{x_{\alpha}}^{-1}$ as a diffeomorphism map from $B(x_{\alpha}, 4r_{\alpha})$ to the standard ball of radius $4r_{\alpha}$ in $\R^{m}$, centered at the origin. So we can define
an $\R^{m(1+Q)}$ vector valued function $\vec{F}$ as follows
\begin{align}
\vec{F}=\left(\eta_1 Exp_{x_1}^{-1}, \eta_2 Exp_{x_2}^{-1}, \cdots, \eta_{Q} Exp_{x_Q}^{-1}; \eta_1, \eta_2, \cdots, \eta_{Q} \right). \label{eqn:PF07_7}
\end{align}

\begin{claim}
$\vec{F}$ is a smooth embedding.
\label{clm:PF24_1}
\end{claim}

For each $x \in K''$, we can find some $\alpha$ such that $x \in B(x_{\alpha}, r_{\alpha})$. For simplicity, let us assume $\alpha=1$.
Since $\eta_1 \equiv 1$ and $Exp_{x_1}^{-1}$ is a diffeomorphism from $B(x_1,r_1)$ to a standard ball in $\R^{m}$, it is obvious that $\vec{F}$ is an immersion.
It is also not hard to see that $\vec{F}$ is injective. For if $\vec{F}(y)=\vec{F}(x)$, then $\eta_{\alpha}(y)=\eta_{\alpha}(x)$ for each $\alpha \in \left\{ 1, 2, \cdots, Q \right\}$.
In particular, we have $\eta_1(y)=\eta_1(x)=1$, which implies that $x,y \in B(x_1,r_1)$.
Then we have $Exp_{x_1}^{-1}(x)=Exp_{x_1}^{-1}(y)$ and consequently $x=y$.
Therefore, $\vec{F}|_{K''}$ is injective. The proof of Claim~\ref{clm:PF24_1} is complete.

Now we parallelly define functions(sets) on each $M_i$. Define
\begin{align*}
K_i \coloneqq \psi_{L,i}(K), \quad K_i' \coloneqq  \left\{x \in M_i | d(x, K_i) \leq \frac{1}{10L} \right\}, \quad K_i'' \coloneqq  \left\{x \in M_i | d(x, K_i) \leq \frac{1}{5L} \right\}.
\end{align*}
Similarly, we define
\begin{align*}
x_{i,\alpha} \coloneqq \psi_{L,i}(x_{\alpha}), \quad \eta_{i,\alpha} \coloneqq \eta \left( \frac{d(\cdot, x_{i, \alpha})}{r_{\alpha}}\right), \quad \xi_{i,\alpha} \coloneqq \frac{\eta_{i,\alpha}}{\sum_{\alpha} \eta_{i,\alpha}}.
\end{align*}
The corresponding estimates of (\ref{eqn:PF07_4}) and (\ref{eqn:PF07_5}) also hold for $\eta_{i,\alpha}$, by the pointed-Gromov-Hausdorff convergence:
\begin{align}
\frac{1}{C(m)} \leq \sum_{\alpha} \eta_{i,\alpha} \leq C(m), \quad |\nabla \eta_{i,\alpha}| \leq C(m). \label{eqn:PF07_6}
\end{align}
According to its choice in (\ref{eqn:PF07_1}) and the uniform curvature and injectivity radius estimate, we know that the convex radius of $x_{i,\alpha}$ is much larger than $r_{\alpha}$.
Therefore, the support set of $\eta_{i,\alpha}$ is inside the convex radius of $x_{i,\alpha}$, we know that $\eta_{i,\alpha}$ actually has all the higher order derivatives' estimates.
We now define
\begin{align}
\vec{F}_i \coloneqq  \left(\eta_{i,1} Exp_{x_{i,1}}^{-1}, \eta_{i,2} Exp_{x_{i,2}}^{-1}, \cdots, \eta_{i,Q} Exp_{x_{i,Q}}^{-1}; \eta_{i,1}, \eta_{i,2}, \cdots, \eta_{i,Q} \right). \label{eqn:PF07_8}
\end{align}
By curvature and injectivity radius estimate, we know both $\vec{F}$ and $\vec{F}_i$ are smooth embeddings, with uniform gradient and higher-order derivative estimates of each order.
The image sets $\vec{F}(K'')$ and $\vec{F}_{i}(K_i'')$ are uniformly bounded. They are $m$-dimensional submanifolds of $\R^{m(1+Q)}$, with uniformly bounded second fundamental forms, and covariant derivatives
of second fundamental forms of each order.
Therefore, by taking subsequence if necessary, we may assume that
\begin{align}
(K_i'', \vec{F}_i) \longright{Gromov-Hausdorff} (K'', \vec{F}). \label{eqn:PF08_1}
\end{align}
In particular, the Hausdorff distance between $\vec{F}(K'')$ and $\vec{F}_i(K_i'')$ is tending to $0$.
Since $K_i' \Subset K_i''$ and $K \Subset K'$, it is clear that $\vec{F}_i(K_i')$ locates in a tiny tubular neighborhood of $\vec{F}(K'')$ and $\vec{F}(K)$ locates in a tiny neighborhood
of $\vec{F}(K_i')$. Then there exist canonical projections $\boldsymbol{P}_i$ from $\vec{F}(K)$ to $\vec{F}_i(K_i')$ and projections $\boldsymbol{p}_i$ from $\vec{F}_i(K_i')$ to $\vec{F}(K'')$.
It is also clear that all the projections $\boldsymbol{P}_i$ and $\boldsymbol{p}_i$ have uniform regularities of each order.
For each $x \in K$, we define $\varphi_{K,L,i}(x)$ to be the unique point $y$ such that
\begin{align*}
\vec{F}_i(y)=\boldsymbol{P}_i(\vec{F}(x)).
\end{align*}
It is clear that $\varphi_{K,L,i}$ is a smooth map. Since $\vec{F}$ is a diffeomorphism from $K$ to $\vec{F}(K)$, $\boldsymbol{P}_i$ is
a diffeomorphism from $\vec{F}(K)$ to $\boldsymbol{P}_i(\vec{F}(K)) \subset \vec{F}_i(K_i')$, and $\vec{F}_i$ is a diffeomorphism from $K_i'$ to $\vec{F}_i(K_i')$,
it then follows from elementary composition that $\varphi_{K,L,i}$ is a diffeomorphism from $K$ to $\varphi_{K,L,i}(K) \subset K_i'$.

\begin{figure}[h]
\begin{center}
\psfrag{A}[c][c]{$\vec{F}(K'')$}
\psfrag{B}[c][c]{$\textcolor{blue}{\vec{F}_i(K_i')}$}
\psfrag{C}[c][c]{$\vec{F}(x)$}
\psfrag{D}[c][c]{$\vec{F}_i(y)$}
\psfrag{E}[c][c]{$\textcolor{red}{\vec{F}(K)}$}
\includegraphics[width=0.3\columnwidth]{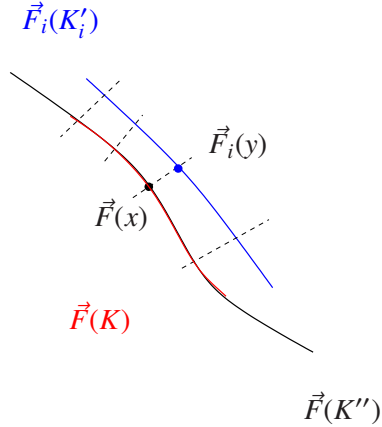}
\end{center}
\caption{Projection as inverse of normal bundle exponential map}
\label{fig:projection}
\end{figure}
For the simplicity of notations, we shall denote $\psi_{L,i}$ by $\psi_{i}$, and denote $\varphi_{K,L,i}$ by $\varphi_i$ in the remainder part of this proof.
Then it is clear that
\begin{align}
&\lim_{i \to \infty} d(p_i, \psi_i(p_{\infty})) + \left| d(\psi_i(z), \psi_i(w))-d(z,w)\right|=0, \quad \forall \; z, w \in M_{\infty}; \label{eqn:PE27_1}\\
&\varphi_i: \; K \mapsto \varphi_i(K) \subset M_i, \quad \textrm{is a smooth diffeomorphism.} \label{eqn:PE27_2}
\end{align}
The proof of Step 1 is complete.

\textit{Step 2. The map $\varphi_{i}|_{K}$ is a small perturbation of $\psi_i|_{K}$. In other words, for each $x \in K$, we have
\begin{align}
\lim_{i \to \infty} d\left( \varphi_i(x), \psi_{i}(x) \right)=0.
\label{eqn:PE25_1}
\end{align}
Consequently, for each $x \in K$ and $w \in M_{\infty}$, we have
\begin{align}
\lim_{i \to \infty} d\left( \varphi_{i}(x), \psi_{i}(w) \right)=d(x,w).
\label{eqn:PE25_2}
\end{align}
}
According to the choice of $\eta_{i, \alpha}$ and $\vec{F}_{i, \alpha}$, it follows from (\ref{eqn:PF08_1}) that
\begin{align*}
\lim_{i \to \infty} \left| \vec{F}(x) - \vec{F}_{i}(\psi_i(x))\right|=0.
\end{align*}
Recall that
\begin{align*}
\lim_{i \to \infty} \left| \vec{F}(x) - \vec{F}_{i}(\varphi_i(x))\right|= \lim_{i \to \infty} \left| \vec{F}(x) - \vec{F}_{i}(y)\right|=0.
\end{align*}
Consequently, we obtain
\begin{align*}
\lim_{i \to \infty} \left| \vec{F}_i(\psi_i(x)) - \vec{F}_{i}(\varphi_i(x))\right|=0.
\end{align*}
As $\vec{F}_i$ are embeddings with uniform gradient estimates, we obtain (\ref{eqn:PE25_1}).

\textit{Step 3. The map $\varphi_{i}$ is a smooth local diffeomorphism from $K$ to $\varphi_{i}(K)$ satisfying
\begin{align}
\varphi_{i}^{*} g_i \longright{C^{\infty}} g_{\infty}, \quad \textrm{on} \; K.
\label{eqn:PD28_3}
\end{align}
}

According to the construction of $\varphi_i$ in step 1, it is clear that $\varphi_i$'s have uniform regularities.
Namely, for each $x \in K$, the maps $Exp_{\varphi_i(x)}^{-1} \circ \varphi_{i} \circ Exp_{x}$ are $\R^{m}$-vector-valued functions
on $B(x, r_x)$ with uniformly bounded derivatives of each order. By taking subsequence if necessary, we obtain (\ref{eqn:PD28_3}).

\textit{Step 4. The convergence holds in pointed-$\hat{C}^{\infty}$-Cheeger-Gromov topology. Namely, (\ref{eqn:PC25_2}) holds.}

We define maps(not necessarily continuous)
\begin{align*}
\bar{\psi}_{i}(x) \coloneqq
\begin{cases}
\psi_{i}(x), \quad \textrm{if} \; x \notin K, \\
\varphi_{i}(x), \quad \textrm{if} \; x \in K.
\end{cases}
\end{align*}
Here $\varphi_i|_{K}$'s are the diffeomorphisms constructed in Step 1.
By (\ref{eqn:PE25_2}), it is clear that $\bar{\psi}_{i}$'s are almost isometries satisfying
\begin{align*}
\bar{\psi}_{i}(p_{\infty}) \to p_i, \quad \bar{\psi}_{i}|_{K}=\varphi_i|_{K}.
\end{align*}
Furthermore, $\varphi_i$ satisfies (\ref{eqn:PD28_3}).
Then it follows from Definition~\ref{dfn:PC25_1} that (\ref{eqn:PC25_2}) holds.

\textit{Step 5. (\ref{eqn:PC25_4}) holds.}

By Definition~\ref{dfn:PC25_1}, in order to show (\ref{eqn:PC25_4}), it suffices to show that
\begin{align}
\varphi_i^{*} f_i \longright{C^{\infty}} f_{\infty}. \label{eqn:PE25_4}
\end{align}
Note that $\varphi_i$ is a smooth map
from $B(x_{\alpha}, r_{\alpha})$ to its image in $B(\varphi_i(x_{\alpha}), 4r_{\alpha})$, with each order derivatives uniformly bounded.
Therefore, it is clear that the value and each order derivatives of $f_i$ are all uniformly bounded, when we regard $f_i$ as functions
on $B(\varphi_i(x_{\alpha}), 4r_{\alpha})$. It follows that $\varphi_i^* f_i$ has values and each order derivatives uniformly bounded on $B(x_{\alpha}, r_{\alpha})$.
Taking subsequence if necessary, we obtain (\ref{eqn:PE25_4}).
Consequently, we know (\ref{eqn:PC25_4}) holds by Definition~\ref{dfn:PC25_1}.
\end{proof}

\begin{remark}
In the proof of showing the convergence being smooth in Proposition~\ref{prn:PC25_4},
we write down the construction of the diffeomorphisms $\varphi_i$ explicitly to avoid singularities. From the construction, one sees that it is unnecessary to require each $M_i$ to be a smooth manifold.
Instead, a regular-singular decomposition for each $M_i$ with the regular part being a smooth manifold is enough for applying the convergence (\ref{eqn:PC25_2}).
Such convergence is used in Chen-Wang~\cite{CW17A} to study the compactness of the moduli of singular-Calabi-Yau spaces.
\label{rmk:PD30_1}
\end{remark}

\begin{remark}
An alternative route to define the diffeomorphism $\varphi_i$ is to use the mass-center-method, as done in S. Peters (cf. \cite{SPeters84},~\cite{SPeters}) and Green-Wu (cf.~\cite{GreenWu}).
In our argument, we instead follow the route of A. Katsuda (cf. \cite{Katsuda}) and A. Kasue (cf. \cite{Kasue}), which garantees (\ref{eqn:PE25_2}) to hold naturally.
\label{rmk:PF24_1}
\end{remark}

\section{Estimates of distance and curvatures}
\label{sec:coincidence}

In this section, we shall provide enough details to finish the proof of Theorem~\ref{thmin:a} and Theorem~\ref{thmin:b}.

\begin{theorem}[\textbf{Density estimate}]
For each $D>10m$, $\theta \in (0,1)$ and $(M, p, g, f) \in \mathcal{M}_{m}(A)$, the following properties hold.

\begin{itemize}
\item[(a).] Then for each $r \in (0, \frac{1}{100D})$ and $q \in B(p, D-1)$, we have
\begin{align}
r^{-2\theta+4-m} \int_{B(q,r)} \mathbf{hr}^{2\theta-4} dv < C \label{eqn:PE18_1}
\end{align}
for some $C=C(m,A,D)$.
\item[(b).] There is a constant $C_{hr}=C_{hr}(m,A,D)$ such that
\begin{align}
\int_{B(p,D)} \mathbf{hr}^{2\theta-4}(x) dv<C_{hr}. \label{eqn:PE18_2}
\end{align}
\item[(c).] There is a constant $C_{Rm}=C_{Rm}(m,A,D)$ depending only on $m$, $A$ and $D$ such that
\begin{align}
\int_{B(p,D)} |Rm|^{2-\theta} dv <C_{Rm}. \label{eqn:PE18_3}
\end{align}
\end{itemize}
\label{thm:density}
\end{theorem}

\begin{proof}

(a). By Theorem~\ref{thm:PE27_1}, it is clear that
\begin{align}
B(q,r) \subset B_{\bar{g}}(q, 2r) \subset B_{\bar{g}}(q, 4r) \subset B(q, 10r). \label{eqn:PH26_1}
\end{align}
Since $100Dr<1$ and $B(q, 100Dr) \subset B(p, D)$, we can replace $r$ by $100Dr$ in Lemma~\ref{lma:PE06_3}.
Then we have
\begin{align}
|\overline{Rc}|_{\bar{g}} \leq D^2, \quad \; \textrm{on} \; B_{\bar{g}} \left( q, 10 r \right). \label{eqn:PE17_2}
\end{align}
On the ball $B_{\bar{g}}(q, 4r)$, we have uniform non-collapsing constant. Namely, we have
\begin{align}
\rho^{-m}|B_{\bar{g}}(y,\rho)| \geq \kappa=\kappa(m,A,D)>0, \quad \forall \; B_{\bar{g}}(y,\rho) \subset B_{\bar{g}}(q, 4r). \label{eqn:PH26_2}
\end{align}
This can be easily derived from (\ref{eqn:PE27_7}), (\ref{eqn:PE06_8}), (\ref{eqn:PH26_1}) and Lemma~\ref{lma:PE04_2}.
Based on the uniform non-collapsing estimate (\ref{eqn:PH26_2}) and the uniform Ricci curvature estimate (\ref{eqn:PE17_2}), we can shrink the size of ball and apply the work of Cheeger-Naber~\cite{CN2} to obtain
\begin{align*}
r^{-2\theta+4-m} \int_{B_{\bar{g}(q, 2r)}} \overline{\mathbf{hr}}^{2\theta-4} dv_{\bar{g}}< C=C(m,A,D).
\end{align*}
Now we can apply the equivalence of $\mathbf{hr}$ and $\overline{\mathbf{hr}}$ (cf. Theorem~\ref{thm:PH10_1}), the equivalence of $dv$ and $dv_{\bar{g}}$ (cf. (\ref{eqn:PE27_7}) and (\ref{eqn:PE06_8})),
and the relationship (\ref{eqn:PH26_1}) to obtain that
\begin{align*}
r^{-2\theta+4-m} \int_{B(q, r)}\mathbf{hr}^{2\theta-4} dv&< C(m) \cdot r^{-2\theta+4-m} \cdot \int_{B_{\bar{g}}(q,2r)} \overline{\mathbf{hr}}^{2\theta-4} dv_{\bar{g}}<C=C(m,A,D),
\end{align*}
whence we arrive (\ref{eqn:PE18_1}).

(b). Note that we have the volume ratio estimates (\ref{eqn:PE04_5}) in Lemma~\ref{lma:PE04_2}, (\ref{eqn:CB09_2}) in Lemma~\ref{lma:PE04_1}.
Therefore, (\ref{eqn:PE18_2}) follows from (\ref{eqn:PE18_1}) by a standard covering argument.

(c). By Lemma~\ref{lma:PH04_1}, there exists a small positive constant $c=c(m)$ such that
\begin{align*}
|Rm|(y) r^2 \leq c, \quad \forall \; y \in B(x,r)
\end{align*}
where $r=\mathbf{hr}(x)$. In particular, we have
$$|Rm|(x)<c \mathbf{hr}^{-2}(x), \quad \forall \; x \in M.$$
Plugging the above inequality into (\ref{eqn:PE18_2}), we obtain (\ref{eqn:PE18_3}).
\end{proof}

\begin{proposition}
Suppose $(M_{\infty}, p_{\infty}, d_{\infty}, f_{\infty})$ is the limit space in (\ref{eqn:PC25_5}) with the regular-singular decomposition $M_{\infty}=\mathcal{R} \cup \mathcal{S}$.
Then $\mathcal{S}$ is a closed set with respect to the topology induced by $d_{\infty}$. Furthermore, we have
\begin{align}
\dim_{\mathcal{M}} \mathcal{S} \leq m-4. \label{eqn:PC26_7}
\end{align}
\label{prn:PC25_3}
\end{proposition}

\begin{proof}
It suffices to show that for each $\theta \in (0,1)$ and $D>10m$, we have
\begin{align}
\limsup_{r \to 0^{+}} r^{2\theta-4} \left| \mathcal{S}_{r} \cap B(p_{\infty}, D) \right| < \infty. \label{eqn:PE17_4}
\end{align}
Fix $r>0$ small. Let $\mathcal{S}_{r}$ be the $r$-neighborhood of $\mathcal{S}$. Namely, we have
\begin{align}
\mathcal{S}_{r} \coloneqq \left\{ x \in M_{\infty} | d(x, \mathcal{S})<r\right\}. \label{eqn:PE27_4}
\end{align}
Then we can choose finite $x_{\alpha} \in \mathcal{S}_{r} \cap B(p_{\infty}, D)$ such that
\begin{align*}
&\mathcal{S}_{r} \cap B(p_{\infty}, D) \subset \bigcup_{\alpha=1}^{\infty} B(x_{\alpha}, r), \\
&B(x_{\alpha}, 0.5r) \cap B(x_{\beta}, 0.5r)=\emptyset, \quad \forall \; \alpha \neq \beta.
\end{align*}
For each $\alpha$, let $x_{\alpha}$ be the limit point of $x_{\alpha, i} \in M_i$.
By the density estimate (\ref{eqn:PE18_1}) in Theorem~\ref{thm:density}, we know that there is a point $y_{\alpha, i} \in B(x_{\alpha, i}, 0.1r)$
such that $\mathbf{hr}(y_{\alpha, i})>c r$ for some small positive constant $c=c(m,A,D)$.
Note that $\mathbf{hr}$ satisfies the local Harnack inequality (\ref{eqn:PI10_2}).
The distance from $y_{\alpha,i}$ to $\mathcal{S}$ is comparable to $r$ by triangle inequality.
This forces that $\mathbf{hr}(y_{i,\alpha})<c^{-1} r$.
Therefore, replacing $c$ by a smaller constant $\delta=\delta(m,A,D)$ if necessary, we have
\begin{align*}
\delta^{-1} r > \mathbf{hr}(x)> \delta r, \quad \forall \; x \in B(y_{\alpha,i}, \delta r).
\end{align*}
Note that $B(y_{\alpha,i},\delta r)$ are disjoint balls. It follows from density estimate again that
\begin{align*}
C(m,D, \delta) \cdot N \cdot r^{2\theta-4+m} \leq \sum_{\alpha=1}^{N} \int_{B(y_{\alpha,i},\delta r)} \mathbf{hr}^{2\theta-4} dv
\leq \int_{B(p,D)} \mathbf{hr}^{2\theta-4} dv < C.
\end{align*}
It follows that
\begin{align}
N r^{2\theta-4+m} \leq C'. \label{eqn:PE17_3}
\end{align}
By taking subsequence if necessary, we can assume $y_{\alpha,i}$ converges to $y_{\alpha} \in B(x_{\alpha}, r)$. Then it is clear that
\begin{align*}
\mathcal{S}_{r} \cap B(p_{\infty}, D) \subset \bigcup_{\alpha=1}^{\infty} B(y_{\alpha}, 3r).
\end{align*}
It follows that
\begin{align*}
\left| \mathcal{S}_{r} \cap B(p_{\infty}, D) \right| \leq \sum_{\alpha=1}^{N} |B(y_{\alpha}, 3r)| \leq N \cdot C \cdot (3r)^{m}
=3^{m}C \cdot N r^{2\theta-4+m} \cdot r^{-2\theta+4}.
\end{align*}
Plugging (\ref{eqn:PE17_3}) into the above inequality, we arrive at
\begin{align*}
\left| \mathcal{S}_{r} \cap B(p_{\infty}, D) \right| \leq C r^{4-2\theta}
\end{align*}
where $C$ depends on $m,D,\theta$. Note that the above inequality implies (\ref{eqn:PE17_4}) and consequently (\ref{eqn:PC26_7}).
\end{proof}

As discussed after Lemma~\ref{lma:PE27_0}, the limit length structure $d_{\infty}$ may globally not coincide with the length structure $d_{g_{\infty}}$
induced by the smooth Riemannian metric $(\mathcal{R}, g_{\infty})$ via approximation.
However, this discrepancy is excluded by the following Proposition.

\begin{prop}
On the limit (\ref{eqn:CB08_4}), the length structure $d_{\infty}$ coincides with the length structure $d_{g_{\infty}}$.
\label{prn:CB08_1}
\end{prop}

\begin{proof}
By Lemma~\ref{lma:PE27_0}, we know that for each $y \in M_{\infty}$, every tangent space $Y$ of $M_{\infty}$ is a metric cone.
Moreover, $\mathcal{R}(Y)$ is connected. By the density estimate (\ref{eqn:PE18_1}), we further know that $\mathcal{R}(Y)$ is a dense subset of $Y$.
Then by a standard covering argument, for each $x,y \in M_{\infty}$ and $\epsilon>0$, one can find a smooth curve $\gamma \subset \mathcal{R}(Y)$ such that
\begin{align*}
d_{\infty}(x,y) \leq |\gamma| \leq (1+\epsilon) d_{\infty}(x,y).
\end{align*}
Full details for the proof of the above inequality can be found in Proposition 3.20 of~\cite{CW17B}. By the arbitrary choice of $\epsilon$, we have
\begin{align*}
d_{g_{\infty}}(x,y)=d_{\infty}(x,y), \quad \forall \; x, y \in M_{\infty}.
\end{align*}
The proof of Proposition~\ref{prn:CB08_1} is complete.
\end{proof}

Now we are ready to finish the proof of Theorem~\ref{thmin:a}.

\begin{proof}[Proof of Theorem~\ref{thmin:a}:]
It follows from the combination of Proposition~\ref{prn:PC25_1}, Proposition~\ref{prn:PC25_2}, Proposition~\ref{prn:PC25_4}, Proposition~\ref{prn:PC25_3} and Proposition~\ref{prn:CB08_1}.
\end{proof}

The following property follows from Theorem~\ref{thmin:a} by trivial rescaling (cf. (\ref{eqn:PH29_3}),(\ref{eqn:PH29_2}), (\ref{eqn:PH29_1}) for notations).

\begin{corollary}
Suppose $(M_i,p_i,g_i, f_i) \in \mathcal{M}_{m}(A)$ and $t_i \in [t_a, t_b] \subset (-\infty, 1)$.
Then by taking subsequence if necessary, we have
\begin{align}
(M_i, p_i, g_i(t_i), f_i(t_i)) \longright{pointed-\hat{C}^{\infty}-Cheeger-Gromov} \left(M_{\infty}, p_{\infty}, g_{\infty}(t_{\infty}), f_{\infty}(t_{\infty}) \right). \label{eqn:PI05_1}
\end{align}
The metric $g_{\infty}(t_{\infty})$ is homothetic to $g_{\infty}$ in (\ref{eqn:CA25_10}).
\label{cly:PH29_1}
\end{corollary}

We can also use the proof of Theorem~\ref{thmin:a} and the density estimate (\ref{eqn:PE18_1}) to show the structure of each tangent space of $M_{\infty}$.

\begin{proposition}
There exist small constants $\epsilon=\epsilon(m), \delta=\delta(m)$ with the following properties.

Suppose $M_{\infty}$ is the limit space in (\ref{eqn:CA25_10}) and $y \in M_{\infty}$. Let $(Y, \hat{y}, \hat{d})$ be a tangent space of $M_{\infty}$ at $y$.
Namely, there exists $r_k \to 0$ such that
\begin{align}
(M_{\infty}, y, r_k^{-1}d_{\infty}) \longright{pointed-Gromov-Hausdorff} (Y, \hat{y}, \hat{d}). \label{eqn:PI09_1}
\end{align}
Then $Y$ is a metric cone with the regular-singular decomposition $Y=\mathcal{R}(Y) \cup \mathcal{S}(Y)$,
where $\mathcal{R}(Y)$ is a smooth Riemannian manifold and $\dim_{\mathcal{M}} \mathcal{S} \leq m-4$.
The convergence (\ref{eqn:PI09_1}) can be improved to
\begin{align}
(M_{\infty}, y, r_k^{-1}g_{\infty}) \longright{pointed-\hat{C}^{\infty}-Cheeger-Gromov} (Y, \hat{y}, \hat{g}). \label{eqn:PI09_2}
\end{align}
The following two statements are equivalent to say that $y \in \mathcal{S}(M_{\infty})$:
\begin{itemize}
\item $d_{PGH} \left\{ (Y, \hat{y}, \hat{d}), (\R^m, 0, d_{E})\right\}>\epsilon$.
\item $\left| B(\hat{y}, 1) \cap \mathcal{R}(Y) \right|<(1-\delta) \omega_m$.
\end{itemize}
\label{prn:PI09_1}
\end{proposition}

\begin{proof}
It follows from Theorem~\ref{thm:PE30_1} that tangent spaces exist.
Moreover, each tangent space $(Y, \hat{y}, \hat{d})$ can be regarded as the tangent cone of some limit space of manifolds with locally uniformly bounded Ricci curvature.
Then the decomposition of $Y$ and equivalent description of $\mathcal{S}(M_{\infty})$ follows from the standard Cheeger-Colding theory (cf. Lemma~\ref{lma:PD30_3} and Lemma~\ref{lma:PH11_1}).
Since $g_{\infty}$ is a smooth Ricci shrinker metric and (\ref{eqn:PE18_1}) can pass to the limit, we can follow the proof of Theorem~\ref{thmin:a} to
obtain (\ref{eqn:PI09_2}) and the dimension estimate $\dim_{\mathcal{M}} \mathcal{S} \leq m-4$.
\end{proof}

There are some non-trivial improvements of Theorem~\ref{thmin:a}.
Note that even though $d_{\infty}=d_{g_{\infty}}$ in Theorem~\ref{thmin:a}, it is not clear whether $ \mathcal{R}(M_{\infty})$ is convex.
Namely, we do not know whether $x,y$ can be connected by a smooth shortest geodesic whenever $x,y \in \mathcal{R}(M_{\infty})$.
This convexity will be proved in a different paper~\cite{HLW}. In fact, in ~\cite{HLW}, it is proved that for every $x \in M_{\infty}$ and $y \in \mathcal{R}(M_{\infty})$,
there is a curve $\gamma$ connecting $x$ and $y$ with smooth interior part such that $d_{\infty}(x,y)$ is achieved by the length of $\gamma$.
Recall that by Proposition~\ref{prn:PI09_1}, each tangent space of $M_{\infty}$ is a metric cone and there is a definite gap between $\R^m$ and any tangent cone of a singular point.
Combining these properties with (a) and (b) of Theorem~\ref{thmin:a}, we know that $M_{\infty}$ is actually a Riemannian conifold in the sense of Chen-Wang (cf. Definition 1.2 of~\cite{CW17A} and the subsequent discussion).
Therefore, Theorem~\ref{thmin:a} can be improved further to the following Theorem.

\begin{theorem}[cf. Theorem 1.1 of~\cite{HLW}]
Let $(M_i^m, p_i, g_i,f_i)$ be a sequence of Ricci shrinkers in $ \mathcal{M}_{m}(A)$.
By passing to a subsequence if necessary, we have
\begin{align*}
(M_i, p_i, g_i,f_i) \longright{pointed-\hat{C}^{\infty}-Cheeger-Gromov} \left(M_{\infty}, p_{\infty}, g_{\infty}, f_{\infty} \right),
\end{align*}
where the limit space $M_{\infty}$ is a Riemannian conifold Ricci shrinker.
\label{thm:PH26_1}
\end{theorem}

The proof of Theorem~\ref{thm:PH26_1} consists of a nontrivial generalization of the fundamental work of Colding-Naber~\cite{ColdNa}.
Partially because of its technical difficulty, we put it in a different paper~\cite{HLW}.

Then we move on to prove Theorem~\ref{thmin:b}.

\begin{proposition}
Same conditions as in Theorem~\ref{thmin:a}. Suppose $\theta \in (0,2]$.

Suppose the ball $B(q_{\infty},r) \subset M_{\infty}$ is the limit of the balls $B(q_i,r) \subset M_i$, then we have
\begin{align}
\lim_{i \to \infty} \int_{B(q_i, r)} |Rm|^{2-\theta} e^{-f_{i}} dv_{i}= \int_{B(q_{\infty}, r) \cap \mathcal{R}} |Rm|^{2-\theta} e^{-f_{\infty}} dv_{\infty}.
\label{eqn:PE20_5}
\end{align}
\label{prn:PE19_2}
\end{proposition}

\begin{proof}
Fix $\epsilon>0$ small. Recall that $\mathcal{S}_{\epsilon}$ is the $\epsilon$-neighborhood of $\mathcal{S}$ by (\ref{eqn:PE27_4}).
Note that $B(q_{\infty}, r) \backslash \mathcal{S}_{\epsilon}$ is a compact subset of $\mathcal{R}$.
It follows from the pointed-$\hat{C}^{\infty}$-Cheeger-Gromov convergence that there are diffeomorphisms $\varphi_i$ from $B(q_{\infty}, r) \backslash \mathcal{S}_{\epsilon}$ to its
image in $M_i$ such that
\begin{align*}
&h_i=\varphi_i^{*} g_i \longright{C^{\infty}} g_{\infty}, \\
&f_i \circ \varphi_i \longright{C^{\infty}} f_{\infty}
\end{align*}
on $B(q_{\infty}, r) \backslash \mathcal{S}_{\epsilon}$. Then we have
\begin{align*}
\int_{B(q_{\infty}, r) \backslash \mathcal{S}_{\epsilon}} |Rm|^{2-\theta} e^{-f_{\infty}} dv_{g_{\infty}}&=
\lim_{i \to \infty} \int_{B(q_{\infty}, r) \backslash \mathcal{S}_{\epsilon}} |Rm|_{h_i}^{2-\theta} e^{-f_i \circ \varphi_i} dv_{h_i}\\
&=\lim_{i \to \infty} \int_{\varphi_i \left( B(q_{\infty}, r) \backslash \mathcal{S}_{\epsilon} \right)} |Rm|_{g_i}^{2-\theta} e^{-f_i} dv_{g_i}.
\end{align*}
Since $\varphi_i$ is the restriction of the $\delta_i$-isometries $\psi_i$ on $B(q_{\infty}, r) \backslash \mathcal{S}_{\epsilon}$ and $\delta_i \to 0$, it is clear that
\begin{align}
B(q_i, r-\delta_i) \backslash \psi_{i}\left( \mathcal{S}_{\epsilon} \right)
\subset \varphi_i \left( B(q_{\infty}, r) \backslash \mathcal{S}_{\epsilon} \right) \subset B(q_i, r+\delta_{i}) \backslash \psi_{i}\left( \mathcal{S}_{0.5\epsilon} \right).
\label{eqn:PE20_6}
\end{align}
Now we choose $\theta' \in (0, \theta)$.
Similar to the deduction of (\ref{eqn:PE17_3}) in Proposition~\ref{prn:PC25_3}, we know that $B(q_i, r-\delta_i) \cap \psi_{i}(\mathcal{S}_{\epsilon})$ can be covered by
$\left\{ B(x_{i,\alpha}, 2j) \right\}_{\alpha=1}^{N}$ such that the number of balls $N$ is dominated by:
\begin{align}
N \leq C \epsilon^{-2\theta'+4-m}
\label{eqn:PE20_1}
\end{align}
for some $C$ independent of $i$ and $\theta'$.
Applying Theorem~\ref{thm:density} on each ball $B(x_{i,\alpha},2\epsilon)$ with respect to $\theta$, we obtain that
\begin{align}
\int_{B(x_{i,\alpha},2\epsilon)} |Rm|^{2-\theta} e^{-f_i} dv_i \leq C \epsilon^{m-4+2\theta}. \label{eqn:PE20_2}
\end{align}
Combining (\ref{eqn:PE20_1}) and (\ref{eqn:PE20_2}), we have
\begin{align}
&\quad \int_{B(q_i, r-\delta_i) \cap \psi_{i}(\mathcal{S}_{\epsilon})} |Rm|^{2-\theta} e^{-f_i} dv_i \notag\\
&\leq \sum_{\alpha=1}^{N} \int_{B(x_{i,\alpha},2\epsilon)} |Rm|^{2-\theta} e^{-f_i} dv_i
\leq C\cdot N \cdot \epsilon^{m-4+2\theta} \notag\\
&\leq C \cdot r^{-2\theta'+4-m} \cdot \epsilon^{m-4+2\theta}
\leq C \epsilon^{2\left( \theta-\theta' \right)}. \label{eqn:PE20_3}
\end{align}
The same argument implies that
\begin{align}
\int_{B(q_i, r+\delta_i) \cap \psi_{i}(\mathcal{S}_{0.5 \epsilon})} |Rm|^{2-\theta} e^{-f_i} dv_i \leq C \epsilon^{2(\theta-\theta')}.
\label{eqn:PE20_4}
\end{align}
On the other hand, it is clear that $|Rm|< C_{\epsilon}$ on
$B(q_i, r-\delta_i) \backslash \psi_{i}(\mathcal{S}_{\epsilon})$ and $B(q_i, r+\delta_i) \cap \psi_{i}(\mathcal{S}_{0.5\epsilon})$.
It follows from this estimate and (\ref{eqn:PE20_3}) that
\begin{align*}
&\quad \int_{B(q_{\infty}, r) \backslash \mathcal{S}_{\epsilon}} |Rm|^{2-\theta} e^{-f_{\infty}} dv_{g_{\infty}}\\
&\leq \liminf_{i \to \infty} \int_{B(q_i, r+\delta_{i}) \backslash \psi_{i}\left( \mathcal{S}_{0.5\epsilon} \right)} |Rm|_{g_i}^{2-\theta} e^{-f_i} dv_{g_i}\\
&\leq \liminf_{i \to \infty} \left\{ \int_{B(q_i, r)} |Rm|_{g_i}^{2-\theta} e^{-f_i} dv_{g_i} + C_{\epsilon} \delta_i + C \epsilon^{2(\theta-\theta')} \right\}\\
&\leq \liminf_{i \to \infty} \int_{B(q_i, r)} |Rm|_{g_i}^{2-\theta} e^{-f_i} dv_{g_i} + C \epsilon^{2(\theta-\theta')}.
\end{align*}
Similarly, by the first relationship in (\ref{eqn:PE20_6}) and (\ref{eqn:PE20_4}), we have
\begin{align*}
\int_{B(q_{\infty}, r) \backslash \mathcal{S}_{\epsilon}} |Rm|^{2-\theta} e^{-f_{\infty}} dv_{g_{\infty}}
\geq \limsup_{i \to \infty} \int_{B(q_i, r)} |Rm|_{g_i}^{2-\theta} e^{-f_i} dv_{g_i} -C \epsilon^{2(\theta-\theta')}.
\end{align*}
Letting $\epsilon \to 0$, the set $B(q_{\infty}, r) \backslash \mathcal{S}_{\epsilon}$ converges to $B(q_{\infty}, r) \backslash \mathcal{S}=B(q_{\infty}, r) \cap \mathcal{R}$,
and the number $\epsilon^{2(\theta-\theta')}$ tends to $0$. Then it follows from the combination of the previous two inequalities that
\begin{align*}
&\quad \limsup_{i \to \infty} \int_{B(q_i, r)} |Rm|_{g_i}^{2-\theta} e^{-f_i} dv_{g_i}\\
&\leq \int_{B(q_{\infty}, r) \cap \mathcal{R}} |Rm|^{2-\theta} e^{-f_{\infty}} dv_{g_{\infty}}\\
&\leq \liminf_{i \to \infty} \int_{B(q_i, r)} |Rm|_{g_i}^{2-\theta} e^{-f_i} dv_{g_i},
\end{align*}
which is equivalent to (\ref{eqn:PE20_5}).
\end{proof}

\begin{proposition}
Same conditions as in Theorem~\ref{thmin:a}. Then we have
\begin{align}
\boldsymbol{\mu}\left( M_{\infty}, g_{\infty} \right)
\coloneqq \log \int_{\mathcal{R}} (4\pi)^{-\frac{m}{2}} e^{-f_{\infty}} dv_{g_{\infty}}
=\lim_{i \to \infty} \boldsymbol{\mu}\left( M_i, g_i \right). \label{eqn:PE20_7}
\end{align}
\label{prn:PE20_1}
\end{proposition}

\begin{proof}
Note that the first equation in (\ref{eqn:PE20_7}) is definition.
In order to show (\ref{eqn:PE20_7}), by the normalization condition (\ref{eqn:PD15_1}), it suffices to prove
\begin{align}
\int_{\mathcal{R}} e^{-f_{\infty}} dv_{g_{\infty}} =\lim_{i \to \infty} \int_{M_i} e^{-f_i} dv_{g_i}.
\label{eqn:PE20_8}
\end{align}
Fix $\epsilon>0$, by the estimate of $f$ and volume of $B(p,r)$, it is clear that there exists a $D>10m$ such that
\begin{align*}
\int_{M \backslash B(p, D)} e^{-f} dv <\epsilon
\end{align*}
for every Ricci shrinker. Same argument also implies that
\begin{align*}
\int_{\mathcal{R} \cap B(p_{\infty}, D)} e^{-f_{\infty}} dv_{g_{\infty}}< \epsilon.
\end{align*}
In (\ref{eqn:PE20_5}), if we choose $\theta=2$, then we obtain
\begin{align*}
\int_{\mathcal{R} \cap B(p_{\infty}, D)} e^{-f_{\infty}} dv_{g_{\infty}}=\lim_{i \to \infty} \int_{B(p_i, D)} e^{-f_i} dv_{g_i}.
\end{align*}
Combining the previous inequalities, we obtain
\begin{align*}
\limsup_{i \to \infty} \int_{M_i} e^{-f_i} dv_{g_i} -\epsilon \leq \int_{\mathcal{R}} e^{-f_{\infty}} dv_{g_{\infty}}
\leq \liminf_{i \to \infty} \int_{M_i} e^{-f_i} dv_{g_i} +\epsilon,
\end{align*}
whence we obtain (\ref{eqn:PE20_8}) by letting $\epsilon \to 0$ in the above inequalities.
\end{proof}

We now finish the proof of Theorem~\ref{thmin:b}.

\begin{proof}[Proof of Theorem~\ref{thmin:b}:]
Since $0 \leq f \leq \frac{D^2}{4}$ on $B(p, D)$ by (\ref{eqn:PE03_6a}), it is clear that $e^{-f}dv$ and $dv$ are equivalent on $B(p, D)$.
Therefore, Part (a) of Theorem~\ref{thmin:b} follows from~\ref{eqn:PE18_3}.
Part (b) is the same as Proposition~\ref{prn:PE19_2}. Part (c) follows from Proposition~\ref{prn:PE20_1}.
\end{proof}

Note that part (b) of Theorem~\ref{thmin:b} fails at the boundary case $\theta=0$. Since there are smooth Einstein complex cubic surfaces converging to an Einstein orbifold in the Gromov-Hausdorff topology, it is clear that there is
a drop of $\int |Rm|^2$ in this process. Part (a) of Theorem~\ref{thmin:b} can be further improved such that the constant $C$ in (\ref{eqn:PE26_4}) is independent of $D$ and $\theta$, even for $\theta=0$.
In other words, we have
\begin{align}
\int_{M} |Rm|^{2-\theta} e^{-f} dv<C(m,A), \quad \forall \; \theta \in [0, 2]. \label{eqn:PH26_3}
\end{align}
However, the estimate (\ref{eqn:PH26_3}) requires delicate heat kernel estimate on Ricci shrinkers, which has many other applications. We refer the interested readers to~\cite{LWkernel} for further development in this direction.


\section{Gap properties of Ricci shrinkers}
\label{sec:gap}

In this section, we shall explore the application of Theorem~\ref{thmin:a} and Theorem~\ref{thmin:b} on compact Ricci shrinkers.
In particular, we shall prove Theorem~\ref{thmin:c}. All the Ricci shrinkers in this section are compact by default.

\begin{proposition}
There is a constant $D=D(m)$ with the following properties.

Suppose $(M^{m},p,g,f)$ is a compact Ricci shrinker. Then we have
\begin{align}
sr(x) < D, \quad \forall \; x \in M. \label{eqn:PC30_5}
\end{align}
Here $sr(x)$ is the strongly-convex radius of $x$ defined in Definition~\ref{dfn:PD15_1}.
\label{prn:PC30_3}
\end{proposition}

\begin{proof}
Suppose $x$ is a point on a Ricci shrinker $(M,p,g,f)$ such that $sr(x)>D$ for some $D$ very large.
Following the definition of $sr$ in Definition~\ref{dfn:PD15_1}, it is clear that
\begin{align*}
|Rm|(y) \leq L^{-2}, \quad inj(y)>L, \quad \forall\; y \in B(x,L)
\end{align*}
for some large $L=L(m,D)$ such that $L \to \infty$ if $D \to \infty$.
Now we naturally identify $(M, g)$ with the time-slice $t=0$ of the induced Ricci flow.
For this flow, we have
\begin{align}
|Rm|(y, 0) \leq L^{-2}, \quad inj(y, 0)>L, \quad \forall\; y \in B(x,L). \label{eqn:PH26_4}
\end{align}
Note that this Ricci flow evolves on a closed manifold. So we can apply Perelman's pseudo-locality theorem (cf. Theorem 10.3 of~\cite{Pe1}).
By the global pseudo-locality theorem of Perelman, we know that
\begin{align*}
|Rm|(x,t) \leq (\epsilon L)^{-2}, \quad \forall \; t \in (0,\epsilon L)
\end{align*}
for some $\epsilon=\epsilon(m)$. Let $L$ be so large such that $\epsilon L>1$. Then we have
\begin{align*}
|Rm|(x,t) \leq 1, \quad \forall \; t \in (0,1),
\end{align*}
which by definition \eqref{eqn:PH29_3} is equivalent to
\begin{align*}
|Rm|\left( \psi_{t}(x), 0 \right) \leq 1-t.
\end{align*}
Since $M$ is a compact manifold, we can choose $t_i \to 1^{-}$ such that $\psi_{t_i}(x) \to z \in M$. It follows from the above inequality that
\begin{align*}
|Rm|(z)=0, \quad \Rightarrow \quad R(z)=0.
\end{align*}
Recall that $\Delta_{f} R=R-2|Rc|^2$.
Then it follows from strong maximum principle that $R \equiv 0$ and consequently $Rc \equiv 0$.
Together with the Ricci shrinker equation, we obtain that $\text{Hess} f=\frac{g}{2}$.
As $M$ is a smooth manifold, we obtain that $(M,g)$ is isometric to $(\R^m, g_{E})$ (see, e.g., Theorem 3.3 of \cite{Li18}), which contradicts the assumption that $M$ is compact.
The proof of Proposition~\ref{prn:PC30_3} is complete.
\end{proof}

\begin{definition}
A metric cone $C(Y)$ of Hausdorff dimension $m$ is called a regular cone if there is a point $z \in C(Y)$ such that one tangent space of $C(Y)$ at $z$ is isometric to $\R^m$.
\label{dfn:PC30_1}
\end{definition}

\begin{theorem}
For each regular cone $C(Y)$ of Hausdorff dimension $m$, there is a constant $\epsilon=\epsilon(Y)>0$ such that
\begin{align}
d_{PGH} \left( \left( M^m,p,g \right), \left( C(Y), 0, g_{c} \right) \right)> \epsilon
\label{eqn:PC30_6}
\end{align}
where $(M^m,p,g)$ is a compact Ricci shrinker.

\label{thm:PC30_1}
\end{theorem}

\begin{proof}
We shall show (\ref{eqn:PC30_6}) by contradiction argument.
For otherwise, we can find a regular cone $C(Y)$ and a sequence of compact Ricci shrinkers $(M_i,p_i,g_i)$ violating (\ref{eqn:PC30_6}).
Namely, we have
\begin{align}
d_{PGH} \left( \left( M_i,p_i,g_i \right), \left( C(Y), 0, g_{c} \right) \right) \to 0.
\label{eqn:PC30_7}
\end{align}
Since $C(Y)$ is a regular cone, there is a point $q \in C(Y)$ such that one tangent space of $q$ is isometric to $\R^m$.
Without loss of generality, we may choose $q$ to be on the link, i.e., $d(q,0)=1$.
By Proposition~\ref{prn:PE29_3}, it is clear that $\boldsymbol{\mu}(M_i)$ is uniformly bounded from below.
Therefore, Theorem~\ref{thmin:a} applies. Then the convergence in (\ref{eqn:PC30_7}) can be improved as
\begin{align*}
(M_i,p_i,g_i) \longright{pointed-\hat{C}^{\infty}-Cheeger-Gromov} \left( C(Y), 0, g_{c} \right).
\end{align*}
Recall that $q \in \partial B(0,1) \subset C(Y)$ is a regular point, so $sr(q)>0$. Let $\gamma$ be a unit-speed ray starting from $0$ and passing through $q$.
It is clear that $\gamma(0)=0$ and $\gamma(1)=q$. Choose $L$ large enough so that $L \cdot sr(q)>2D$, where $D$ is the constant in Proposition~\ref{prn:PC30_3}.
Let $z=\gamma(L)$ and $z_i \in M_i$ such that $z_i \to z$. Clearly, we have $sr(z) \geq 2D$. By the smooth convergence in Theorem~\ref{thmin:a}, we obtain
\begin{align*}
\lim_{i \to \infty} sr(z_i) > D.
\end{align*}
In particular, for large $i$, we can find a point $z_i \in M_i$ such that $sr(z_i)>D$, which contradicts Proposition~\ref{prn:PC30_3}.
The proof of Theorem~\ref{thm:PC30_1} is complete.
\end{proof}

Clearly, each Euclidean quotient space $\R^m/\Gamma$ for a finite subgroup $\Gamma \subset O(m)$ is a regular cone. So we have the following rigidity result,
which confirms Conjecture 6.1 of~\cite{LWs1} by assuming $M$ being compact.

\begin{corollary}
For each $m,l$, there is a constant $\epsilon=\epsilon(m,l)$ with the following properties.
Suppose $\Gamma$ is a finite subgroup of $O(m)$ acting freely on $S^{m-1}$ with $|\Gamma| \leq l$.
Suppose $(M^{m},p,g)$ is a compact Ricci shrinker. Then we have
\begin{align}
d_{PGH}\left( (M,p,g), (\R^{m}/\Gamma, 0, g_{E}) \right)>\epsilon. \label{eqn:PC30_8}
\end{align}
\label{cly:PC30_1}
\end{corollary}

Recall the normalization of $f$ and $\boldsymbol{\mu}$ from (\ref{eqn:PC27_0}) and (\ref{eqn:PC27_1}).
The following identity is known (cf. Corollary 4.1 of~\cite{CN09}) by the work of Carrillo and Ni.
\begin{align}
\boldsymbol{\mu}=\boldsymbol{\mu}(g,1).
\label{eqn:PD15_2}
\end{align}
where $\boldsymbol{\mu}(g,1)$ is Perelman's functional \cite{Pe1}. Furthermore, $f+\boldsymbol{\mu}$ is a minimizer for $\boldsymbol{\mu}(g,1)$.
The equality (\ref{eqn:PD15_2}) can be improved to the following property.

\begin{proposition}
For any compact Ricci shrinker $(M,f,g)$, we have
\begin{align}
\boldsymbol{\mu}(g)=\boldsymbol{\nu}(g) \coloneqq \inf_{\tau>0} \boldsymbol{\mu}(g, \tau). \label{eqn:PC27_2}
\end{align}
\label{prn:PC30_1}
\end{proposition}

\begin{proof}
Clearly, $\boldsymbol{\mu}(g, \tau)$ is a function defined on $(0, \infty)$. In order to prove (\ref{eqn:PC27_2}), it suffices to show
\begin{subequations}
\begin{align}[left = \empheqlbrace \,]
&\boldsymbol{\mu}(g,\tau) \geq \boldsymbol{\mu}(g, 1), \quad \forall \; \tau \in (0, 1);\label{eqn:PC27_8a}\\
&\boldsymbol{\mu}(g, \tau) \geq \boldsymbol{\mu}(g, 1), \quad \forall \; \tau \in (1, \infty). \label{eqn:PC27_8b}
\end{align}
\label{eqn:PC27_8}
\end{subequations}
We shall show (\ref{eqn:PC27_8}) as follows. First, we show that $\boldsymbol{\mu}(g,\tau)$ is an increasing function of $\tau$ on the interval $(1,\infty)$ and it is
a decreasing function of $\tau$ on the interval $(0,1)$. Second, we show that $\boldsymbol{\mu}(g, \tau)$ is a continuous function of $\tau$ at $\tau=1$.
See Figure~\ref{fig:mu} for intuition. The key ingredients of the proof are Perelman's monotonicity formula and a priori estimates of eigenfunctions. The details will be carried out in the following paragraphs separately.

\begin{figure}[h]
\begin{center}
\psfrag{A}[c][c]{$\boldsymbol{\mu}(g,\tau)$}
\psfrag{B}[c][c]{$\tau$}
\psfrag{C}[c][c]{$\boldsymbol{\mu}(g,\tau)=\boldsymbol{\mu}$}
\psfrag{D}[c][c]{$\tau=1$}
\includegraphics[width=0.5\columnwidth]{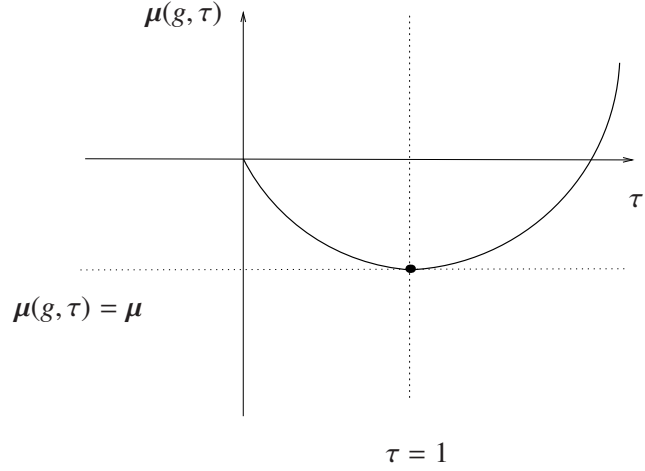}
\end{center}
\caption{$\boldsymbol{\mu}(g,\tau)$ as a function of $\tau$}
\label{fig:mu}
\end{figure}

\textit{Step 1. $\boldsymbol{\mu}(g, \tau)$ is increasing on $(1,\infty)$ and decreasing on $(0,1)$.}

Fix an arbitrary open domain $(\tau_0, \tau_1) \subset (1, \infty)$. We shall show that
\begin{align}
\boldsymbol{\mu}(g, \tau_1) \ge \boldsymbol{\mu}(g, \tau_0).
\label{eqn:PC27_13}
\end{align}
Since $(M, g)$ is compact Ricci shrinker, it can be extended as a Ricci flow solution on a closed manifold and the flow is induced by diffeomorphisms $\psi_{t}$:
\begin{align*}
g(t)=(1-t) \left( \psi_{t} \right)^{*} g, \quad -\infty<t<1,
\end{align*}
which implies that
\begin{align}
\boldsymbol{\mu}\left( g(t), \tau_0-t \right)&=\boldsymbol{\mu}\left( (1-t) \left( \psi_{t} \right)^{*} g, \tau_0-t \right)
=\boldsymbol{\mu}\left( \left( \psi_{t} \right)^{*} g, \frac{\tau_0-t}{1-t} \right) \notag\\
&=\boldsymbol{\mu}\left( g, \frac{\tau_0-t}{1-t} \right)=\boldsymbol{\mu}\left( g, s \right), \label{eqn:PC30_2}
\end{align}
where we define $s=s(t)$ by
\begin{align}
s(t) \coloneqq \frac{\tau_0-t}{1-t}, \quad \forall \; t \in (0,1). \label{eqn:PC27_11}
\end{align}
Clearly, $s=s(t)$ is an increasing function on $(0,1)$ and it maps $(0,1)$ onto $(\tau_0, \infty)$. Let $t=t(s)$ be the inverse map of $s=s(t)$.
Then
\begin{align}
t(s)=\frac{s-\tau_0}{s-1}, \quad \forall \; s \in (\tau_0, \infty). \label{eqn:PC27_12}
\end{align}
By the monotonicity formula of Perelman, $\boldsymbol{\mu}(g(t),\tau_0-t)$ is an increasing function of $t \in (0,1)$.
For each $s \in (\tau_0, \infty)$, we can
choose $t=t(s)$ by (\ref{eqn:PC27_12}) such that
\begin{align*}
\boldsymbol{\mu}(g, s)=\boldsymbol{\mu}\left( g(t), \tau_0-t \right) \ge \boldsymbol{\mu}\left( g, \tau_0 \right).
\end{align*}
In particular, by choosing $s=\tau_1$, we obtain (\ref{eqn:PC27_13}).
By the arbitrary choice of $(\tau_0, \tau_1) \subset (1,\infty)$, it is clear that $\boldsymbol{\mu}(g,\tau)$ is an increasing function of $\tau$ on $(1,\infty)$.

Similarly, we can show that $\boldsymbol{\mu}(g, \tau)$ is decreasing on $(0,1)$. In fact, fixing an arbitrary connected domain $(\eta_1, \eta_0) \subset (0, 1)$, we shall show that
\begin{align}
\boldsymbol{\mu}\left( g, \eta_0 \right) \le \boldsymbol{\mu}\left( g(t), \eta_1 \right).
\label{eqn:PC30_1}
\end{align}
Similar to (\ref{eqn:PC30_2}), we have
\begin{align}
\boldsymbol{\mu}\left( g(t), \eta_0-t \right)&=\boldsymbol{\mu}\left( g, \theta \right) \label{eqn:PC30_3}
\end{align}
where $\theta$ and $t$ satisfy the relationship
\begin{align*}
&\theta= \frac{\eta_0-t}{1-t}, \quad \forall \; t \in [0, \eta_0];\\
&t=\frac{\eta_0-\theta}{1-\theta}, \quad \forall \; \theta \in [0, \eta_0].
\end{align*}
Since $\boldsymbol{\mu}\left( g(t), \eta_0-t \right)$ is an increasing function of $t$, it follows from (\ref{eqn:PC30_3}) that
\begin{align*}
\boldsymbol{\mu}\left( g, \eta_0 \right) \le \boldsymbol{\mu}\left( g(t), \theta(t) \right), \forall \; t \in (0, \eta_0).
\end{align*}
Choose $t=\frac{\eta_1-\eta_0}{\eta_1-1} \in (0, \eta_0)$, then $\theta(t)=\eta_1$.
Therefore, (\ref{eqn:PC30_1}) follows from the above inequality by setting $t=\frac{\eta_1-\eta_0}{\eta_1-1}$.

\textit{Step 2. $\boldsymbol{\mu}(g,\tau)$ is continuous at $\tau=1$. Namely, we have}
\begin{align}
\lim_{\tau \to 1} \boldsymbol{\mu}(g, \tau)=\boldsymbol{\mu}(g, 1).
\label{eqn:PC27_10}
\end{align}

It is clear that (\ref{eqn:PC27_10}) is equivalent to the combination of the following three inequalities:
\begin{subequations}
\begin{align}[left = \empheqlbrace \,]
&\lim_{\tau \to 1} \boldsymbol{\mu}(g, \tau) \leq \boldsymbol{\mu}, \label{eqn:PC27_14a}\\
&\lim_{\tau \to 1^{+}} \boldsymbol{\mu}(g, \tau) \geq \boldsymbol{\mu}, \label{eqn:PC27_14b}\\
&\lim_{\tau \to 1^{-}} \boldsymbol{\mu}(g, \tau) \geq \boldsymbol{\mu}. \label{eqn:PC27_14c}
\end{align}
\label{eqn:PC27_14}
\end{subequations}
We shall prove (\ref{eqn:PC27_14a}), (\ref{eqn:PC27_14b}) and (\ref{eqn:PC27_14c}) separately.
The proof of (\ref{eqn:PC27_14a}) is straightforward.
In fact, for each $\tau>0$, we can choose $f_1=f+\boldsymbol{\mu}-\frac{m}{2} \log \tau$ to satisfy the normalization
\begin{align*}
\int_{M} (4\pi \tau)^{-\frac{m}{2}} e^{-f_1} dv=1.
\end{align*}
Then we calculate
\begin{align*}
\boldsymbol{\mu}(g,\tau) &\leq \int \left ( \tau(2\Delta f-|\nabla f|^2+R)+f+\boldsymbol{\mu}-\frac{m}{2}\log \tau-m \right ) (4\pi \tau)^{-\frac{m}{2}} e^{-f_1} dv\\
&=(\tau-1)e^{-\boldsymbol{\mu}} \int(m-f) (4\pi)^{-\frac{m}{2}} e^{-f} dv+ \boldsymbol{\mu}-\frac{m}{2}\log \tau.
\end{align*}
Letting $\tau \to 1$ and taking limit of both sides of the above inequality, we arrive at (\ref{eqn:PC27_14a}).

We proceed to show (\ref{eqn:PC27_14b}). Letting $u$ be the minimizer of $\boldsymbol{\mu}(g, \tau)$ for some $\tau>0$, we have
\begin{align*}
&\quad \boldsymbol{\mu}(g, \tau) +m + \frac{m}{2}\log(4\pi\tau)\\
&=\int_{M} \left\{ \tau(4|\nabla u|^2+Ru^2)-u^2\log u^2 \right\} dv\\
&\geq \int_{M} \left\{ (4|\nabla u|^2+Ru^2)-u^2\log u^2 \right\} dv + (\tau-1) \int_{M} \left( 4|\nabla u|^2 + Ru^2 \right) dv\\
&\geq \boldsymbol{\mu} +\frac{m}{2} +\frac{m}{2} \log (4\pi) + (\tau-1) \int_{M} \left( 4|\nabla u|^2 + Ru^2 \right) dv.
\end{align*}
In short, the above inequalities read as
\begin{align}
\boldsymbol{\mu}(g, \tau) -\boldsymbol{\mu} \geq -\frac{m}{2} \log \tau + (\tau-1) \int_{M} \left( 4|\nabla u|^2 + Ru^2 \right) dv.
\label{eqn:PC27_19}
\end{align}
Recall that $R \geq 0$ on $M$. Then (\ref{eqn:PC27_14b}) follows from the above inequality by letting $\tau \to 1^{+}$.

Finally, we focus on the proof of (\ref{eqn:PC27_14c}). This step needs some a priori estimates of the minimizer function $u$ of
$\boldsymbol{\mu}(g, \tau)$ whenever $\tau \in (0.5, 1)$. Actually, by the compactness of $M$, we have $R \geq \epsilon$ for some $\epsilon$ depending on $M$.
Then it follows from standard $L^2$-Sobolev inequality on compact manifold that
\begin{align}
\left(\int u^{\frac{2m}{m-2}}\,dv\right)^{\frac{m-2}{m}} \leq C\int \left( 4|\nabla u|^2+Ru^2 \right) dv \label{eqn:PC27_17}
\end{align}
for every $u \in W^{1,2}(M,g)$. Here $C=C(M,g)$ is independent of $u$.
By Jensen's inequality, we have
\begin{align}
\int u^2\log u^2 dv =\frac{m-2}{2}\int u^2\log u^{\frac{4}{m-2}} dv \leq \frac{m-2}{2}\log\left(\int u^{\frac{2m}{m-2}} dv \right).
\label{eqn:PC27_16}
\end{align}
Combining the previous two steps, we obtain
\begin{align*}
\int u^2\log u^2 dv \leq C'+\frac{m}{2} \log \left\{ \int \left(4|\nabla u|^2+Ru^2 \right) dv\right\}.
\end{align*}
Now we choose $\tau \in (0.5,1)$ and let $u$ be the minimizer of $\boldsymbol{\mu}(g,\tau)$.
It follows from the definition of $u$ and the above inequality that
\begin{align}
&\quad \boldsymbol{\mu}(g, \tau)+m +\frac{m}{2}\log(4\pi\tau) \notag\\
&=\int_{M} \left\{ \tau(4|\nabla u|^2+Ru^2)-u^2\log u^2 \right\} dv \notag\\
&\geq \frac{1}{2} \int_M (4|\nabla u|^2+Ru^2) dv -\frac{m}{2} \log \left\{ \int \left(4|\nabla u|^2+Ru^2 \right) dv\right\}-C'. \label{eqn:PC27_15}
\end{align}
Since power function and logarithmic function tends to infinity at different rates, by (\ref{eqn:PC27_14a}), it is clear that
\begin{align*}
\lim_{\tau \to 1^{-}} \int_M (4|\nabla u|^2+Ru^2) dv < C'',
\end{align*}
which, together with (\ref{eqn:PC27_17}), in turn implies that
\begin{align*}
\lim_{\tau \to 1^{-}} \int \left\{ 4|\nabla u|^2+Ru^2 \right\} dv <C'''.
\end{align*}
Plugging the above inequality into (\ref{eqn:PC27_19}), we arrive at (\ref{eqn:PC27_14c}).
\end{proof}

Step 2 of the proof of Proposition~\ref{prn:PC30_1} can be generalized to show that $\boldsymbol{\mu}(g, \tau)$ is a continuous function of $\tau$ on each closed manifold, without any new technical difficulty.
If the underlying manifold is non-compact, this step's technique does not apply since it is unclear whether there is a uniform $L^2$-Sobolev constant a priori.
Proposition~\ref{prn:PC30_1} indicates that for any compact Ricci shrinker, the log-Sobolev inequality holds for all scales, and the constant depends only on $\boldsymbol{\mu}$.
Actually, (\ref{eqn:PC27_2}) holds even if $M$ is a non-compact Ricci shrinker, which is proved in a separate paper~\cite{LWkernel}.
One key of this generalization in~\cite{LWkernel} is an entirely new argument for step 2 in the proof of Proposition~\ref{prn:PC30_1}.

From the equivalence between $L^2$-Sobolev inequality and the log-Sobolev inequality (cf. Chapter $2$ of \cite{Dav89}) for all scales of the operator $-4\Delta+R$, we immediately obtain the following $L^2$-Sobolev constant estimate from Proposition~\ref{prn:PC30_1} (see Corollary $4$ of \cite{LWkernel}).

\begin{corollary}
Given a compact Ricci shrinker $(M,g,f)$, then we have
\begin{align}
\left(\int u^{\frac{2m}{m-2}}\,dv\right)^{\frac{m-2}{m}} \le C\int \left( 4|\nabla u|^2+Ru^2 \right) \,dv
\label{E234}
\end{align}
for any smooth function $u$. Here the constant $C=C(m,\boldsymbol{\mu})$.
\label{cly:PD15_1}
\end{corollary}

Corollary~\ref{cly:PD15_1} is very useful for the space-time Moser iteration. Indeed, the $L^2$-Sobolev inequality \eqref{E234} also holds for $(M,g(t))$ if $t<1$, since $g(t)$ is self-similar. Here as usual, we regard a Ricci shrinker $(M, g,f)$ as the $t=0$ time-slice of the associated Ricci flow space-time $M \times (-\infty, 1)$.
Now we define
\begin{align}
P(r) \coloneqq\{(x,t) \mid f(x,t) \le mr,\, -mr \le t \le 0\}. \label{eqn:PH26_5}
\end{align}

\begin{lemma}[\textbf{Mean value inequality}]
Given a Ricci shrinker $(M^m,g,f)$ regarded as a space-time $M \times (-\infty,0]$ solution of the Ricci flow, suppose $u$ is a nonnegative function such that
$$\square u \coloneqq (\partial_t-\Delta)u \le 0,$$
then we have
\begin{align*}
\max_{P(1)}u^2 \le C \iint_{P(2)} u^2 \,dv dt
\end{align*}
where $C=C(m,\boldsymbol{\mu})$.

\label{lma:PD16_1}
\end{lemma}

\begin{proof}
In the proof all $\nabla=\nabla_{g(t)}$, $\Delta=\Delta_{g(t)}$ and $dv=dv_{g(t)}$.

For any $k \in \mathbb N$, we set $\beta_k \coloneqq (\frac{m}{m-2})^k$. Let $\psi_k$ to be a monotone cutoff function such that $\psi_k=1$ on $(-\infty,1+2^{-k-1}]$ and $\psi_k=0$ on $[1+2^{-k},+\infty)$.
Moreover, $-2^{k+3} \leq \psi_k' \leq 0$.
We also define $\eta_k(x,t)=\psi_k(m^{-1}f(x,t))\psi_k(-m^{-1}t)$. Notice that $\eta_k$ is supported in $P(1+2^{-k})$. From a direct computation
\begin{align*}
|\nabla \eta_k|+|\partial_{t} \eta_k|\le C\left(1+|\nabla f(x,t)|+|\partial_t f(x,t)|\right )|\psi'_k| \le C2^k,
\end{align*}
where the last inequality follows from \eqref{X001}.

For each $\beta \ge 1$, we have
\begin{align}\label{E901}
\square u^{\beta }=\beta u^{\beta -1}\square u -\beta(\beta -1)u^{\beta-2}|\nabla u|^2 \le 0.
\end{align}
Multiplying both sides of (\ref{E901}) by $\eta_k^2 u^{\beta_k}$ and
integrating by parts, we obtain
\begin{align*}
\iint |\nabla (\eta_k u^{\beta_k})|^2\,dvdt &\le \iint (|\nabla \eta_k|^2+(\eta_k)_t^2/2)u^{2\beta_k}\,dv dt-\left. \int u^{2\beta}\eta_k^2/2\,dv \right|_{t=0} \\
&\le C4^k \iint_{P(1+2^{-k})} u^{2\beta_k} \,dvdt.
\end{align*}
It follows from $\eqref{E234}$, the $L^2$-Sobolev inequality, that
\begin{align*}
& \left(\iint_{P(1+2^{-k-1})} u^{2\beta_k} \,dVdt\right)^{\frac{m-2}{m}}\\
\le& \left(\iint (\eta_k u^{\beta_k} )^{\frac{2m}{m-2}} \,dVdt \right)^{\frac{m-2}{m}} \\
\le& C \iint \left\{ |\nabla (\eta_k u_k^{\beta_k})|^2+R \eta_k^2 u^{2\beta_k} \right\} dv dt \\
\le& C4^k \iint_{P(1+2^{-k})} u^{2\beta_k} \,dv dt,
\end{align*}
where for the last inequality we have used the fact that $R$ is uniformly bounded on $P(2)$, from \eqref{eqn:PH29_1}.
The above inequality can be rewritten as
\begin{align}\label{E902}
\left(\iint_{P(1+2^{-k-1})} u^{2\beta_k} \,dVdt\right)^{\frac{1}{2\beta_{k+1}}} \le C^{\frac{1}{\beta_k}}2^{\frac{k}{\beta_k}}\left(\iint_{P(1+2^{-k})} u^{2\beta_k} \,dv dt\right)^{\frac{1}{2\beta_k}}.
\end{align}
Then it follows immediately from \eqref{E902} by interation and taking $k \to +\infty$ that
\begin{align*}
\max_{P(1)}u \le C^{\sum_{k \ge 0} \frac{1}{\beta_k}}2^{\sum_{k \ge 0} \frac{k}{\beta_k}} \left(\iint_{P(2)} u^2 \,dv dt \right)^{\frac{1}{2}}\le C\left(\iint_{P(2)} u^2 \,dv dt \right)^{\frac{1}{2}}.
\end{align*}
\end{proof}

Now we prove the strong maximum principle in the integral form for the scalar curvature $R(x,t)$. We define
\begin{align}
P(r,\tau) \coloneqq P(r) \cap \left\{ (x,t) \mid\, R(x,t)<\tau \right\} \label{eqn:PH26_6}
\end{align}
and
\begin{align}
P^{a}(r,\tau) \coloneqq P(r,\tau) \cap M \times \{a\} \label{eqn:PH26_7}
\end{align}
for any $\tau>0$ and $a \le 0$.

\begin{lemma}
There is a constant $c=c(m,A)>0$ with the following properties.
Suppose $(M^{m},p,g,f)$ is a compact Ricci shrinker satisfying $\boldsymbol{\mu} \geq -A$ and
$$\sigma \coloneqq \inf_{B(p,1)} R(x).$$
Then we have
\begin{align}
|P^a(2,2\sigma)| \geq c \label{eqn:PC31_2}
\end{align}
for some $a \in [-2m,0]$.
\label{lma:PC31_1}
\end{lemma}

\begin{proof}
Recall that the scalar curvature $R(x,t)$ satisfies the equation
\begin{align*}
\square R =2|Rc|^2 \ge 0.
\end{align*}
Fix $\tau>\sigma$. Then we have
\begin{align*}
\square \left( \tau-R \right) \le 0.
\end{align*}
Now we define
\begin{align*}
(\tau-R)_{+} \coloneqq \max\left\{ \tau-R, 0 \right\}.
\end{align*}
It is clear that we have
\begin{align*}
\square \left( \tau-R \right)_{+} \le \square \left( \tau-R \right) \le 0,
\end{align*}
where the first inequality is in the distribution sense.
Then Lemma~\ref{lma:PD16_1} implies that
\begin{align*}
\max_{P(1)} (\tau-R)_{+}^{2} \leq C \iint_{P(2)} (\tau-R)_{+}^{2} dv dt \le C\tau^2 |P(2,\tau)|,
\end{align*}
where the last volume is with respect to the measure $dv dt$.
Setting $\tau=2\sigma$, the above inequality implies that
\begin{align*}
\sigma^2 \le \max_{P(1)} (2\sigma-R)_{+}^{2} \leq C \iint_{P(2)} (2\sigma-R)_{+}^{2} dv dt \le C\sigma^2|P(2,2\sigma)|.
\end{align*}
Here, the first inequality holds since $B(p,1) \subset P(1)$ by \eqref{eqn:PE01_11a}. Therefore, there must exist an $a \in [-2m,0]$ such that
\begin{align*}
|P^a(2,2\sigma)| \geq c.
\end{align*}
\end{proof}

\begin{proposition}
Suppose $(M_{\infty}, p_{\infty}, g_{\infty}, f_{\infty})$ is a limit space in (\ref{eqn:CA25_10}).
Suppose there is a point $x \in \mathcal{R}(M_{\infty})$ such that $R(x)=0$, then we have
\begin{align}
Rc \equiv 0, \quad \textrm{on} \quad \mathcal{R}(M_{\infty}). \label{eqn:PC30_4}
\end{align}
Furthermore, $M_{\infty}$ is isometric to a metric cone $C(Y)$, with $p_{\infty}$ as vertex.
\label{prn:PC30_2}
\end{proposition}

\begin{proof}
From the equation
\begin{align}\label{X002}
\Delta_{f_{\infty}}R-R=-2|Rc|^2 \le 0
\end{align}
and the strong maximum principle on the regular part of $M_{\infty}$, we conclude that $R$ is identically $0$ on $\mathcal{R}(M_{\infty})$. Moreover, it follows from \eqref{X002} that $Rc=0$ on $\mathcal{R}(M_{\infty})$ as well. It implies that
\begin{align*}
\text{Hess} f_{\infty}=\frac{g_{\infty}}{2}
\end{align*}
on $\mathcal{R}(M_{\infty})$ and $p_{\infty}$ is a minimum point of $f_{\infty}$.
The limit space has singularity of codimension at least four and $\mathcal{R}(M_{\infty})$ is geodesically convex (cf. Theorem~\ref{thm:PH26_1}).
Note that the cone rigidity theorem of Cheeger-Colding \cite{ChCo96} also holds in this setting, by Lemma 2.34 of~\cite{CW17A}, whose proof is purely Riemannian.
See also Section 10 of~\cite{Bam17}. Therefore $M_{\infty}$ is a metric cone.
\end{proof}

There are some similarities between Theorem~\ref{thmin:c} and earlier gap theorems for Ricci shrinkers in the literature.
For example, one can compare Theorem~\ref{thmin:c} with the main theorem of S. Zhang (cf. \cite{SJZhang17}).
However, Theorem~\ref{thmin:c} is based on an entirely different proof, which will be provided as follows to conclude this section.

\begin{proof}[Proof of Theorem~\ref{thmin:c}:]
For otherwise, there exists a sequence of compact Ricci shrinkers
$$(M_i, p_i, g_i,f_i) \in \mathcal{M}_m(A)$$
such that
\begin{align*}
\sigma_i =\inf_{x \in B(p_i, 1)} R_i(x) \to 0.
\end{align*}
It follows from Lemma~\ref{lma:PC31_1} that
\begin{align*}
|P_i^{a_i}(2,2\sigma_i)| \geq c
\end{align*}
where $a_i \in [-2m,0]$ and we add the subscript $i$ for $P$ to denote the set with respect to $g_i$.
Following the definition (\ref{eqn:PH26_7}) and the estimate of $f_i(x,t)$ (cf. \eqref{X001} and \eqref{E106}), it is clear that
\begin{align*}
P_i^{a_i}(2,2\sigma_i) \subset B_{g_i}(p_i, D)
\end{align*}
for some large $D$. Then we apply the density estimate (\ref{eqn:PE18_2}) for $\theta=0.5$ to obtain that
\begin{align*}
\int_{ P_i^{a_i}(2,2\sigma_i)} \mathbf{hr}_{g_i(a_i)}^{-3} dv_{g_i(a_i)} \le C \int_{B_{g_i}(p_i, D)} \mathbf{hr}_{g_i}^{-3} dv_{g_i} \le C,
\end{align*}
where the first inequality holds since $g_i(a_i)=(1-a_i) (\psi^{a_t})^* g_i$. Therefore, there is a point $z_i \in P_i^{a_i}(2,2\sigma_i)$ such that
\begin{align}
\mathbf{hr}_{g_i(a_i)} (z_i) \ge r_0:=\left( \frac{c}{C}\right)^{\frac{1}{3}}.
\label{eqn:PC31_3}
\end{align}
It follows from Theorem~\ref{thmin:a} and Corollary~\ref{cly:PH29_1} that
\begin{align}
&(M_i, p_i, g_i, f_i) \longright{pointed-\hat{C}^{\infty}-Cheeger-Gromov} (M_{\infty}, p_{\infty}, g_{\infty}, f_{\infty}), \\
&(M_i, p_i, g_i(a_i), f_i(a_i)) \longright{pointed-\hat{C}^{\infty}-Cheeger-Gromov} (M_{\infty}, p_{\infty}, g_{\infty}(a_{\infty}), f_{\infty}(a_{\infty})).\label{E602}
\end{align}
In light of (\ref{eqn:PC31_3}), we obtain a limit point $z_{\infty} \in \mathcal{R}(M_{\infty})$ such that
\begin{align*}
0 \leq R(z_{\infty}, a_{\infty}) \leq \lim_{i \to \infty} 2 \sigma_i=0.
\end{align*}
Then it follows from Proposition~\ref{prn:PC30_2} that $(M_{\infty},g_{\infty}(a_{\infty}))$ is isometric to a regular metric cone $C(Y)$, in the sense of Definition~\ref{dfn:PC30_1}.
However, since each $(M_i,g_i(a_i), f_i(a_i))$ is a compact Ricci shrinker up to scaling, this is impossible by Theorem~\ref{thm:PC30_1}.
Therefore, this contradiction establishes the proof of Theorem~\ref{thmin:c}.
\end{proof}

\section{Further discussion}
\label{sec:further}

It is not difficult to generalize Ricci shrinkers' weak compactness theory to more general metric measure spaces.

\begin{definition}
Given a positive smooth nondecreasing function $F(r)$ on $[0,\infty)$ and a constant $K$, we define $\mathcal M_m(F,K)$ to be the class of pointed smooth metric measure space $(M^m, p, g, \mu_f)$ satisfying the following conditions.
\begin{itemize}
\item[(a).] $(M^m, g, \mu_f)$ is a complete smooth Riemannian manifold with a weighted measure
\begin{align}
\mu_f \coloneqq e^{-f}dv. \label{eqn:PI04_2}
\end{align}
\item[(b).] $f$ is a $C^2$-function on $M$ such that
\begin{align}
\max \left\{ |f|(x), |R|(x)+|\nabla f|^2(x) \right\}\leq F^2(d_g(p,x)). \label{eqn:PI04_3}
\end{align}
\item[(c).] The Bakry-\'{E}mery Ricci tensor satisfies
\begin{align}
-Kg \le Rc_f \coloneqq Rc+\emph{Hess}\,f \le Kg. \label{eqn:PI04_4}
\end{align}
\end{itemize}
In addition, the subclass $\mathcal M_m(F,K;V_0) \subset \mathcal M_m(F,K)$ consists of all spaces satisfying further
\begin{itemize}
\item[(d).] (Noncollapsing condition)
\begin{align}
\left|B(p,1) \right|_{\mu_f} \geq V_0. \label{eqn:PI04_7}
\end{align}
\end{itemize}
\label{dfn:PI04_1}
\end{definition}

It is clear from (\ref{eqn:PC06_3}), (\ref{E106}) and the definition equation (\ref{eqn:PC27_0}) that every Ricci shrinker $(M^m, g, f)$ belongs to $\mathcal M_m(F,1/2)$ for the function $F(r)=\frac{1}{2}\left(r+\sqrt{2m} \right)$.
Therefore, $\mathcal M_m(F,1/2)$ is a generalization of the moduli of Ricci shrinkers.
Furthermore, by \eqref{eqn:PE01_11a} and \eqref{eqn:PE01_9}, we have
\begin{align*}
\mathcal{M}_{m}(A) \subset \mathcal{M}_{m}\left(F, \frac{1}{2}; c_m e^{-A} \right),
\end{align*}
where $c_m=(4\pi)^{\frac{m}{2}} e^{-2^{4m+7}-m}$.
Not surprisingly, a weak compactness theory similar to Theorem $1.1$ also holds for $\mathcal M_m(F,K;V_0)$.

\begin{theorem}[\textbf{Weak-compactness of $\mathcal M_m(F,K; V_0)$}]
Let $(M_i^m, p_i, g_i, \mu_{f_i})$ be a sequence of metric measure spaces in $\mathcal M_m(F,K;V_0)$.
Let $d_i$ be the length structure induced by $g_i$. By passing to a subsequence if necessary, we have
\begin{align}
(M_i^m, p_i, d_i, f_i) \longright{pointed-Gromov-Hausdorff} \left(M_{\infty}, p_{\infty}, d_{\infty},f_{\infty} \right),
\label{eqn:PI04_1}
\end{align}
where $(M_{\infty}, d_{\infty})$ is a length space, $f_{\infty}$ is a Lipschitz function on $(M_{\infty}, d_{\infty})$.
The space $M_{\infty}$ has a regular-singular decomposition $M_{\infty}=\mathcal{R} \cup \mathcal{S}$ with the following properties.
\begin{itemize}
\item[(a).] The singular part $\mathcal{S}$ is a closed set of Minkowski codimension at least $4$.
\item[(b).] The regular part $\mathcal{R}$ is an open manifold with a $C^{1,\alpha}$ metric $g_{\infty}$ and a $C^{1,\alpha}$ function $f_{\infty}$.
\end{itemize}
The convergence can be improved to
\begin{align}
(M_i, p_i, g_i,f_i) \longright{pointed-\hat{C}^{1,\alpha}-Cheeger-Gromov} \left(M_{\infty}, p_{\infty}, g_{\infty}, f_{\infty} \right). \label{eqn:CA25_101}
\end{align}
Furthermore, the metric structure induced by smooth curves in $(\mathcal{R}, g_{\infty})$ coincides with $d_{\infty}$.

\label{thm:PH04_1}
\end{theorem}

Theorem \ref{thm:PH04_1} follows almost verbatim from Theorem \ref{thmin:a}.
We will only sketch a proof with emphasis on the differences.

\begin{proof}[Sketch of proof of Theorem~\ref{thm:PH04_1}:]
Given a metric measure space $(M^m, p, g,\mu_f) \in \mathcal M_m(F,K)$, it follows from Lemma \ref{lma:PE01_1} that for any $D>10m$, there exists a constant $C=C(m,K,F,D)>0$ such that
\begin{align*}
\frac{|B(q,2r)|_{\mu_f}}{|B(q,r)|_{\mu_f}} \leq C, \quad \forall \; q \in B(p, D), \; r \in (0, D).
\end{align*}
In other words, the volume doubling condition holds uniformly locally. Then the standard ball-packing argument implies that
\begin{align}
(M_i^m, d_i, p_i) \longright{pointed-Gromov-Hausdorff} \left(M_{\infty}, d_{\infty}, p_{\infty}\right).
\end{align}
In addition, the functions $f_i$ have uniform increasing rate locally by (\ref{eqn:PI04_3}). Therefore, $f_i$ converges to a Lipschitz function $f_{\infty}$ and we arrive at (\ref{eqn:PI04_1}).
Note that we actually only used part of (\ref{eqn:PI04_3}) and the left hand side of (\ref{eqn:PI04_4}) to prove (\ref{eqn:PI04_1}).

The major difference between Theorem~\ref{thm:PH04_1} and Theorem~\ref{thmin:a} is the regularity issue.
In the situation of Theorem~\ref{thm:PH04_1}, by setting $\bar{f}=f-f(q)$ and $\bar{g}=e^{\frac{-2\bar{f}}{m-2}}g$, we have the system
\begin{subequations}
\begin{align}[left = \empheqlbrace \,]
&\overline{Rc}=Rc_f+\frac{1}{m-2} \left\{ df \otimes df +(\Delta f-|\nabla f|^2) e^{\frac{2\bar{f}}{m-2}} \bar{g} \right\}, \label{eqn:102}\\
&\bar{\Delta} f=e^{\frac{2\bar{f}}{m-2}} \left( \text{Tr}(Rc_f)-|\nabla f|^2-R \right ).\label{eqn:E103}
\end{align}
\label{eqn:PI04_5}
\end{subequations}
To improve the regularity, we fix a point $q$ such that $d(p,q) \leq D-1$.
It follows the definition that $\Delta f$ is locally bounded (cf. (\ref{eqn:PI04_3}) and (\ref{eqn:PI04_4})). Therefore, $\overline{Rc}$ is bounded on $B(q,1)$.
It is also easy to see from (\ref{eqn:PI04_3}), (\ref{eqn:PI04_4}) and (\ref{eqn:E103}) that $\left| \bar{\Delta} f \right|$ is uniformly bounded on $B(q,1)$.
Similar to the proof of Lemma~\ref{lma:PH04_1}, we can write the system (\ref{eqn:PI04_5}) in harmonic coordinate chart.
Studying the bootstrapping argument below system (\ref{eqn:CA25_9}), it is clear that the first step works under the conditions (\ref{eqn:PI04_3}) and (\ref{eqn:PI04_4}), i.e.,
the Bakry-\'{E}mery tensor $Rc+\text{Hess} f$ is uniformly bounded and $\norm{f}{C^1}$ is uniformly bounded.
In this case, the first step of the bootstrapping argument applies. For each $\alpha \in (0, 1)$, we have
\begin{align}
\norm{\bar{g}}{C^{1,\alpha}} + \norm{f}{C^{1,\alpha}} \leq C, \label{eqn:PI04_6}
\end{align}
where $C$ depends only on $m,D$ and the harmonic radius of $q$. Recall that $g=e^{\frac{2\bar{f}}{m-2}}\bar{g}$. Therefore, (\ref{eqn:PI04_6}) induces
uniform $C^{1,\alpha}$-estimate of $g$ and $f$.
The higher regularity fails since we do not have equations for manifolds in $\mathcal M_m(F,K;V_0)$.
After we obtain the regularity improvement estimate (\ref{eqn:PI04_6}), the same argument as in Proposition \ref{prn:PC25_2} shows that the regular part $\mathcal R \subset M_{\infty}$ is equipped with a $C^{1,\alpha}$ metric $g_{\infty}$.
Moreover, it follows from \eqref{eqn:PI04_6} that the limit function $f_{\infty}$ is also locally $C^{1,\alpha}$ on $\mathcal R$.
Modulo technical difference between $C^{1,\alpha}$-convergence (cf. Chapter 11 of P. Petersen~\cite{Petersen}) and $C^{\infty}$-convergence, the rest part of the proof follows identically from Theorem \ref{thmin:a}.
\end{proof}

Theorem~\ref{thmin:a} can also be easily modified to obtain a weak compactness theorem of some moduli of steady Ricci solitons.
Since the argument is almost identical, we omit the details here.
One can also replace (\ref{eqn:PI04_3}) by weaker condition $Rc_{f} \geq -Kg$ and drop the non-collapsing condition (\ref{eqn:PI04_7}).
Under such weak conditions, one can still discuss the regularity issue of the limit spaces, which is discussed in a separate paper~\cite{HLW}.


\vskip10pt

Haozhao Li, CAS Wu Wen-Tsun Key Laboratory of  Mathematics, 
 School of Mathematical Sciences, University of Science and 
 Technology of China, No. 96 Jinzhai Road, Hefei, Anhui Province, 230026,
  China; hzli@ustc.edu.cn.\\

Yu Li, The Institute of Geometry and Physics, University of Science and Technology of China, No. 96 Jinzhai Road, Hefei, Anhui Province, 230026, China; adterram.yu@gmail.com.\\

Bing Wang, The Institute of Geometry and Physics, 
School of Mathematical Sciences, University of Science and 
Technology of China, No. 96 Jinzhai Road, Hefei, Anhui Province, 230026, China; topspin@ustc.edu.cn.\\

\end{document}